\documentclass[a4paper, 12pt]{amsart}
\usepackage{amsmath,amsthm,amssymb}
\include{springsbib,Impact-Zones-bib,myrefs,citations2}
\usepackage[T1]{fontenc}
\usepackage{url}
\usepackage{amsmath,amsthm,amssymb}
\usepackage{verbatim}
\usepackage{mathrsfs}
\usepackage{graphics}
\usepackage{graphicx}
\usepackage{subfigure}
\usepackage{hyperref}
\usepackage{mathtools}
\usepackage{color}
\usepackage{wasysym}

\newtheorem{theorem}{Theorem}[section]
\newtheorem{proposition}[theorem]{Proposition}
\newtheorem{lemma}[theorem]{Lemma}
\newtheorem{corollary}[theorem]{Corollary}

\theoremstyle{remark}
\newtheorem{remark}[theorem]{Remark}
\newtheorem{definition}{Definition}
\newtheorem{example}{Example}

\numberwithin{equation}{section}

\newcommand{\vep}{\varepsilon}
\newcommand{\R}{{\mathbb{R}}}

\newcommand{\C}{{\mathbb{C}}}
\newcommand{\Z}{{\mathbb{Z}}}
\newcommand{\N}{{\mathbb{N}}}
\newcommand{\T}{{\mathbb{T}}}

\newcommand{\sgn}{\operatorname{sgn}}

\newcommand{\rp}{\operatorname{Re}}
\newcommand{\ip}{\operatorname{Im}}

\begin{document}

\title[Hamiltonian flows with impacts]{Non-uniform ergodic properties of Hamiltonian flows with impacts}
\author[K.\ Fr\k{a}czek]{Krzysztof Fr\k{a}czek}
\address{Faculty of Mathematics and Computer Science, Nicolaus
Copernicus University, ul. Chopina 12/18, 87-100 Toru\'n, Poland}
\email{fraczek@mat.umk.pl}

\author[V.\ Rom-Kedar]{Vered Rom-Kedar}
\address{The Estrin Family Chair of Computer Science and Applied Mathematics. Department of Computer Science and Applied Mathematics,
The Weizmann Institute, Rehovot, Israel 76100}
\email{Vered.Rom-Kedar@weizmann.ac.il}

\date{\today}

\subjclass[2000]{}
\thanks{Research partially supported by the Narodowe Centrum Nauki Grant 2017/27/B/ST1/00078 and by ISF  Grant 1208/16.}

\begin{abstract}
The ergodic properties of two uncoupled oscillators, a horizontal and vertical one, residing in a class of non rectangular star-shaped polygons with  only vertical and horizontal boundaries and impacting elastically from its boundaries are studied. We prove that the iso-energy level sets topology changes non-trivially; the flow on level sets is always conjugated to a translation flow on a translation surface, yet, for some segments of partial energies the genus of the surface is strictly larger than one. When at least one of the oscillators is un-harmonic, or when both are harmonic and non-resonant, we prove that
for almost all partial energies, including the impacting ones, the flow on  level sets is unique ergodic. When both oscillators are harmonic and resonant, we prove that there exist intervals of partial energies on which periodic ribbons and additional ergodic components co-exist. We prove that for almost all partial energies in such segments the motion is unique ergodic on the part of the level set that is not occupied by the periodic ribbons. This implies that ergodic averages  project to piecewise smooth weighted averages in the configuration space.
\end{abstract}

\maketitle

\section{Introduction}
Hamiltonian Impact Systems (HIS) describe the motion of a particle in a given Hamiltonian field within a billiard table: the Hamiltonian flow determines the particle trajectory in the configuration space till it reaches the billiard boundary, there it reflects elastically, and then it continues with the Hamiltonian flow \cite{KozlovBook}. For mechanical Hamiltonian flows with bounded energy surfaces, for small energy,  as long as  the  energy surface projection to the configuration space (the Hill's region) does not  touch the billiard boundary, the HIS reduces to the study of smooth mechanical Hamiltonian systems. On the other extreme, at the high energy limit, for compact billiard tables, mechanical HIS  limit to the billiard flow in the specified billiard table. The theory for intermediate energy values includes local analysis near periodic orbits \cite{dullin1998linear} and near smooth convex boundaries \cite{zharnitsky2000invariant,berglund1996billiards}, and, for some specific classes of  HIS, hyperbolic behavior \cite{Wojtkowski1999}, Liuoville integrable \cite{Kozlova99,Fedorov2001,Dragovic2014a,Radnovic2015,PRk2020} and near-integrable \cite{pnueli2018near} dynamics were established. A class of quasi-integrable HIS, which is related to the quasi-integrable dynamics in families of polygonal right angled corners, was  introduced in  \cite{Issi2019}.

The analysis of quasi-integrable dynamics in right angled billiards is related to several deep mathematical fields \cite{Zorich2006}.  A new family of billiards with quasi-integrable dynamics, consisting of confocal ellipses with confocal barriers was introduced in \cite{DrRa2014} (and, if the Birkhoff conjecture is correct, this family and polygonal billiards with rational angles are the only billiards with quasi-integrable dynamics). For such a billiard table, the quasi-integrable dynamics depend on a  parameter - the constant of motion associated with the caustic of the trajectories. By a change of coordinates, the dynamics  for any given caustic constant is conjugated to the directed motion in a  right angled billiard table \cite{DrRa2014}. Using tools of homogeneous dynamics it was established that the flow in an ellipse with a vertical barrier is uniquely ergodic  for almost all the caustic parameters  \cite{FrShUl2018}. Developing a different approach, a similar result was established for the more general case  of nibbled ellipses \cite{frkaczek2019recurrence}. Our methodology relies on the methods developed in   \cite{frkaczek2019recurrence}, where it was shown   that to prove unique ergodicity, the Minsky-Weiss criterion \cite{MiWe2014}  may be applied to a class of right angled polygons consisting of staircase polygons.

Here, we examine the dynamics of a horizontal and a vertical oscillator with stable fixed point at the origin that are restricted to lie within  star-shaped polygons with  only vertical and horizontal boundaries with a kernel that includes the origin (this is a subclass of the HIS introduced in   \cite{Issi2019}, and such polygons consist of \(4\) staircase polygons considered in   \cite{frkaczek2019recurrence}). These \(2\) degrees of freedom systems have \(2\) conserved integrals, so their motion is always restricted to level sets, yet, in contrast to the smooth case, the motion on the level sets is conjugated, for some of the level sets, to the motion in polygonal right angled billiards with more than \(4\) corners. Thus, the motion on such level sets is not conjugated to rotations \cite{Zorich2006}, and, since the shape of the polygonal billiard and the direction of motion on it vary, the dynamics may depend sensitively on the value of the conserved integrals, even for iso-energy level sets \cite{Issi2019}. In the first part of the paper we analyse our class of HIS in non-resonant cases and prove that the motion is unique-ergodic for almost all iso-energy level sets (in this part we rely on the tools and analysis developed in    \cite{frkaczek2019recurrence} for staircase polygons, verifying that the functional dependence of the corners in the induced family of polygons satisfies the needed conditions of smoothness, independence and monotonicity as in     \cite{frkaczek2019recurrence}). In the second part of the paper we study the case of HIS with quadratic resonant potentials in a staircase polygon. Here, ribbons of periodic orbits co-exist with quasi-periodic motion (similar to the motion presented in \cite{McM}, but in the corresponding HIS). By our new construction, we establish the unique ergodicity of the non-periodic component on a certain set. In Section~\ref{subsec:ergodicav} we show that these  findings imply non-uniform ergodic averages in the configuration space.

\subsection{Set up}

Consider integrable 2 d.o.f.\ (degrees of freedom) Hamiltonian systems of the form
\begin{equation}\label{eq:Hamfun}
H(p_x,x,p_{y},y)=\frac{p_x^2}{2}+\frac{p_y^2}{2}+V_1(x)+V_2(y),
\end{equation}
where  $(x,y)$ are the space coordinates,  $(p_x,p_y)$ are the corresponding momenta and the potentials  $V_1,V_2:\R\to\R_{\geq 0}$ are even uni-modal $C^2$-maps that tend monotonically to infinity with their argument (with no loss of generality we take \(V_{1}(0)=V_2(0)=0\), see precise assumption "Deck"  below). The Hamiltonian flow $(\varphi_t)_{t\in\R}$ of \eqref{eq:Hamfun}  describes a particle which oscillates in a potential well.   The Hamiltonian flow on a given energy surface \(E\) is foliated by the level sets with fixed partial energies, \((E_{1}=H_{1}(I_1),E_{2}=H_{2}(I_2)=E-E_{1})\), where \(I_{i}(E_i)\) is the action of the one d.o.f.\ system \(H_i\). For a given energy level $E\geq 0$ and any $0\leq E_{1}\leq E$ let
\[S_{E,E_{1}}:=\Big\{(p_x,x,p_y,y)\in\R^4:\frac{p_x^2}{2}+V_1(x)=E_{1},\frac{p_y^2}{2}+V_2(y)=E-E_{1},(x,y)\in \R^2\Big\}.\]
Then the phase space of the flow $(\varphi_t)_{t\in\R}$,  is foliated by the invariant sets $\{S_{E,E_{1}}:E\geq 0,0\leq E_{1}\leq E\}$, which are tori for  \(0<E_{1}<E,\) and, for \(E>0\) and  \(E_1\in\{0,E\}\), are circles.
 Denote the restriction of  $(\varphi_t)_{t\in\R}$ to $S_{E,E_{1}}$  by $(\varphi^{E,E_{1}}_t)_{t\in\R}$. The smooth flow without reflection is trivially integrable and oscillatory. The projection of \(S_{E,E_{1}}\) to the configuration space is the projected rectangle \cite{PRk2020}:
\begin{equation}\label{eq:rectangle}
R^{(E,E_1)}=[-x^{max}(E_{1}),x^{max}(E_{1})]\times[-y^{max}(E-E_{1}),y^{max}(E-E_{1})],
\end{equation}
where \(V_{1}(x^{max}(E_{1}))=E_{1}, V_{2}(y^{max}(E_{2}))=E_{2}, E_2=E-E_1\). The union of all iso-energy rectangles is the Hill region: \(\mathcal{D}_{Hill}(E)=\bigcup_{0\leq E_{1}\leq E} R^{(E,E_1)}=\{(x,y)|V_{1}(x)+V_{2}(y)\leqslant E\}\)  (see \cite{PRk2020} for more general formulation).

 Denote by   \(\omega_{i}(E_i)=\frac{2\pi }{T_{i}(E_i)}\) the frequency in each degree of freedom, where \(T_{i}(E_i)\) is the period of oscillation. The standard transformation to action angle coordinates \((I_{i},\theta_i)\) in each degree of freedom brings  \eqref{eq:Hamfun} to the form \(H(p_x,x,p_y,y)=H_{1}(I_{1})+H_2(I_2)\)  and in these coordinates the flow is simply  $\varphi^{E,E_1}_t=\{I_{1}(E_{1}),\theta_1(t)=\omega_{1}(E_1)t+\theta_1(0),I_{2}(E-E_{1}),\theta_2(t)=\omega_{2}(E-E_1)t+\theta_2(0)).$ Recall that \(\omega_{i}(I_i)=H_{i}'(I_{i})=\frac{dE_{i}}{dI_{i}}\) and that the Hamiltonian is said to satisfy the twist condition if \(det(\frac{\partial^{2}H}{\partial I_{i} \partial I_j})=\prod\omega_{i}'(I_i)\neq0\) and to  satisfy the iso-energy non-degeneracy condition if
\begin{displaymath}
\begin{vmatrix}\frac{\partial^{2}H}{\partial I_{i} \partial I_j} & \frac{\partial H}{\partial I_{i} } \\
\frac{\partial H}{\partial I_{j} } & 0 \\
\end{vmatrix}=
\begin{vmatrix}\omega_{1}'(I_1) & 0 &  \omega_1(I_1)  \\
0 & \omega_{2}'(I_2) &  \omega_2(I_2)\\
\omega_1(I_1) & \omega_{2}(I_2) &  0\\
\end{vmatrix}=-\omega_{1}'(I_1)\omega_2^{2}(I_2)-\omega_1^{2}(I_1)\omega_{2}'(I_2)\neq0.
\end{displaymath}
The character of the smooth flow  $\varphi^{E,E_1}_t$   on the level set  $S_{E,E_1}$ depends on the frequency ratio on this level set. If
\[\Omega=\Omega(E,E_1)=\frac{\omega_{1}(E_1)}{\omega_{2}(E-E_1)}=\frac{H'_{1}(I_{1}(E_1))}{H'_{2}(I_{2}(E-E_1))}\]
is rational the flow  is periodic (the resonant case) and if \(\Omega\) is irrational it is quasi-periodic (the non-resonant case).
Recall that \(\frac{d}{dE_{1}}\Omega(E,E_1)\neq0\) iff the iso-energy non-degeneracy condition is satisfied. 

    Now, assume that the particle is confined to a bounded polygonal room $P\subset\R^2$ whose walls consist of vertical and horizontal segments. When the particle meets the wall it reflects elastically. More precisely, if a trajectory meets a vertical segment at $(p_x,x,p_y,y)$ then it jumps to $(-p_x,x,p_y,y)$ and continues its movement in accordance with the Hamiltonian flow solving:
\begin{equation}\label{eq:Ham2}
\frac{dp_x}{dt}=-V_1'(x), \quad\frac{dx}{dt}=p_x,\quad \frac{dp_y}{dt}=-V_2'(y), \quad\frac{dy}{dt}=p_y.
\end{equation}
Similarly, if a trajectory meets a horizontal segment at $(p_x,x,p_y,y)$ then it jumps to $(p_x,x,-p_y,y)$ and continues its movement with \eqref{eq:Ham2}, see \cite{Issi2019} for the general construction, a mechanical example and the description of the resulting dynamics on energy surfaces and \cite{PRk2020} for global structure of energy surfaces of such systems.

\medskip

In particular, since all the walls are either horizontal or vertical, the partial energies are preserved under these reflections, so the motion  remains restricted to  level sets:
\[S^{P}_{E,E_{1}}:=\{(p_x,x,p_y,y)\in\R^4:H_{1}(x,p_x)=E_1,H_{2}(y,p_y)=E-E_{1},(x,y)\in P\Big\}.\]
Denote the restriction of the impact Hamiltonian flow $(\varphi_t)_{t\in\R}$ to $S^{P}_{E,E_{1}}$  by $(\varphi^{P,E,E_{1}}_t)_{t\in\R}$, see  for example a trajectory segment projected to the configuration space in Figure~\ref{fig:particleroom}. Notice that if a particle hits any corner of $P$ then it dies. { Namely, the flow is fully defined for regular trajectories and is only partially defined on the set which corresponds to all forward and backward images of corner points.}

 We are interested in studying the topology of the level set  $S^{P}_{E,E_{1}}$ and the invariant measures of $(\varphi^{P,E,E_{1}}_t)_{t\in\R}$ when the total and partial energies $E,E_{1}$  vary. In particular, we ask when the flow is uniquely ergodic. Recall  that the flow $(\varphi^{P,E,E_{1}}_t)_{t\in\R}$ is  uniquely ergodic if (i) each of its orbits is forward or backward infinite and (ii) if there exists a probability
measure $\mu_{E,E_1}$ on $S^P_{E,E_1}$ such that for every continuous map $f:S^P_{E,E_1}\to\R$ and any $(p_x,x,p_y,y)\in S^P_{E,E_1}$ for which the $(\varphi^{P,E,E_1}_t)_{t\in\R}$-orbit of $(p_x,x,p_y,y)$ is either forward or backward infinite
\[\lim_{T\to\pm\infty}\frac{1}{T}\int_{0}^Tf(\varphi_{ t}(p_x,x,p_y,y))dt=\int_{S^P_{E,E_1}}f\,d\mu_{E,E_1},\]
where the $+$ (respectively $-$) sign is taken if the $(\varphi^{P,E,E_1}_t)_{t\in\R}$-orbit of $(p_x,x,p_y,y)$ is forward (respectively backward)  infinite.
In our case the measure $\mu_{E,E_1}$ is equivalent to the Lebesgue  measure on $S^P_{E,E_1}$.

\medskip

To formally determine the living space, $P$, of the particle, by following \cite{frkaczek2019recurrence}, denote by $\Xi$ the set of  sequences $(\overline{x},\overline{y})=(x_i,y_i)_{i=1}^k$ of points in $\R^2_{>0}$ such that
\[0<x_1<x_2<\ldots<x_{k-1}<x_k\quad\text{and}\quad 0<y_k<y_{k-1}<\ldots<y_2<y_1.\]
\begin{figure}[h]
\includegraphics[width=0.6 \textwidth]{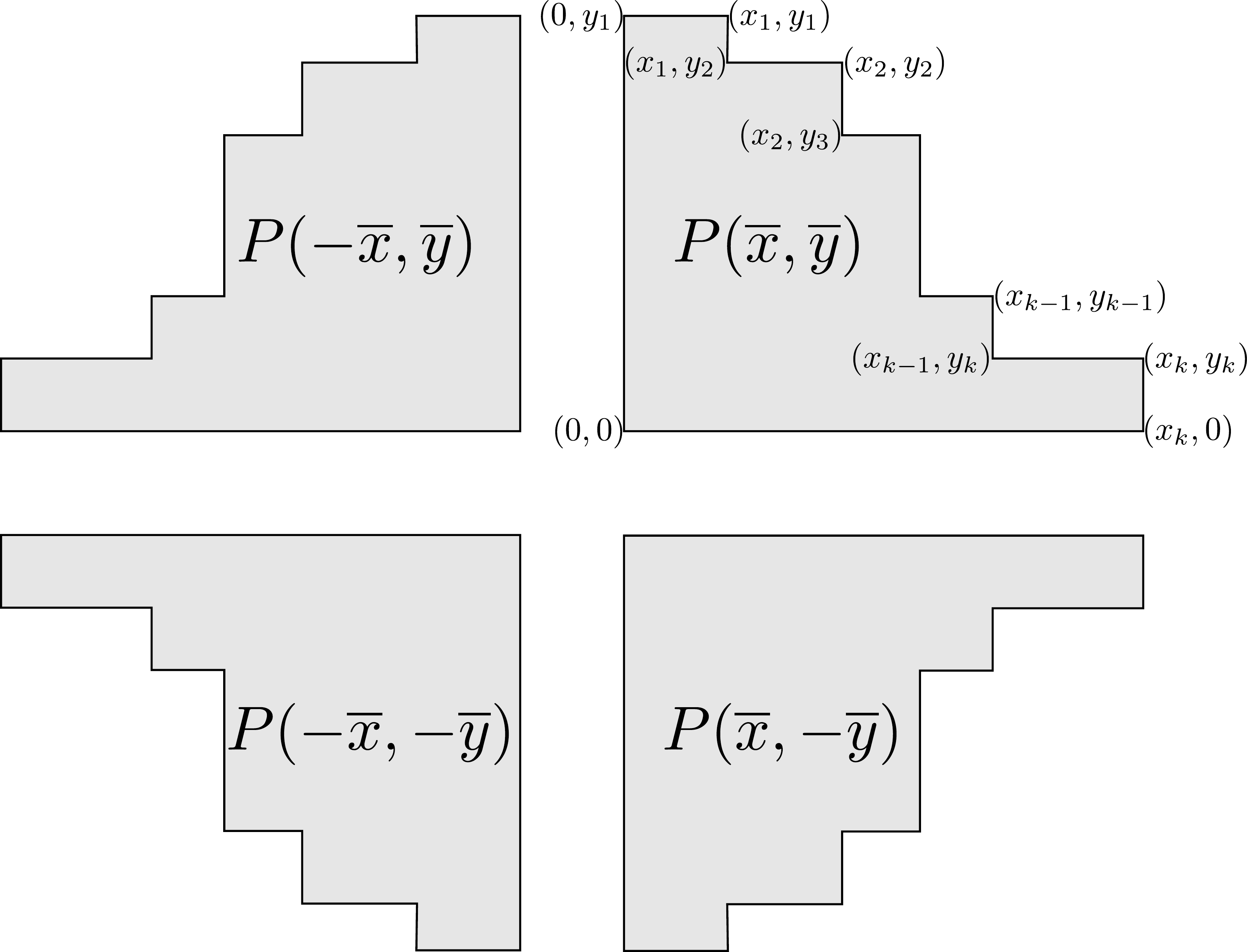}
\caption{Staircase polygons $P(\overline{x},\overline{y})$, $P(-\overline{x},\overline{y})$, $P(\overline{x},-\overline{y})$, $P(-\overline{x},-\overline{y})$.}\label{fig:basicpolygon}
\end{figure}
For every $(\overline{x},\overline{y})\in\Xi$ set $k(\overline{x},\overline{y}):=k$ and
denote by $P(\overline{x},\overline{y})$ the right-angle  staircase polygon on $\R^2$  with consecutive vertices:
\[(0,0),(0,y_1),(x_1,y_1),(x_1,y_2),\ldots,(x_{k-1},y_{k-1}),(x_{k-1},y_{k}),(x_k,y_k),(x_k,0),\]
see Figure~\ref{fig:basicpolygon}.

Denote by $\Gamma$ the four element group
generated by the vertical and the horizontal reflections $\gamma_v, \gamma_h:\R^2\to\R^2$.
The polygons of the form
\begin{align*}
P(\overline{x},\overline{y}), \; P(-\overline{x},\overline{y})=\gamma_v P(\overline{x},\overline{y}),\;
P(\overline{x},-\overline{y})=\gamma_h P(\overline{x},\overline{y}),\;
P(-\overline{x},-\overline{y})=\gamma_v\circ\gamma_h P(\overline{x},\overline{y})
\end{align*}
are called \emph{staircase polygons}, see Figure~\ref{fig:basicpolygon}. The numbers $x_1,\ldots, x_{k-1}$ are called staircase lengths
and $y_2,\ldots, y_k$ are called staircase heights of the staircase polygon $P(\pm\overline{x},\pm\overline{y})$.
The number $x_{k}$ is called the width
and $y_1$ is called the height of $P(\pm\overline{x},\pm\overline{y})$.

We assume that the living space $P$ of the particle is the union of four staircase polygons determined by the four sequences
$(\overline{x}^{\varsigma_1\varsigma_2},\overline{y}^{\varsigma_1\varsigma_2})\in\Xi$, $\varsigma_1,\varsigma_2\in\{\pm\}$ (in the symmetric case all sequences are identical):
\[P(\overline{x}^{++},\overline{y}^{++}), \; P(-\overline{x}^{-+},\overline{y}^{-+}),\;
P(\overline{x}^{+-},-\overline{y}^{+-}),\;
P(-\overline{x}^{--},-\overline{y}^{--}),\]
such that
\[y^{++}_1=y^{-+}_1,\ y^{+-}_1=y^{--}_1,\ x^{++}_{k(\overline{x}^{++},
\overline{y}^{++})}=x^{+-}_{k(\overline{x}^{+-},\overline{y}^{+-})},\
x^{-+}_{k(\overline{x}^{-+},
\overline{y}^{-+})}=x^{--}_{k(\overline{x}^{--},\overline{y}^{--})}.\]
\begin{figure}[h]
\includegraphics[width=0.6 \textwidth]{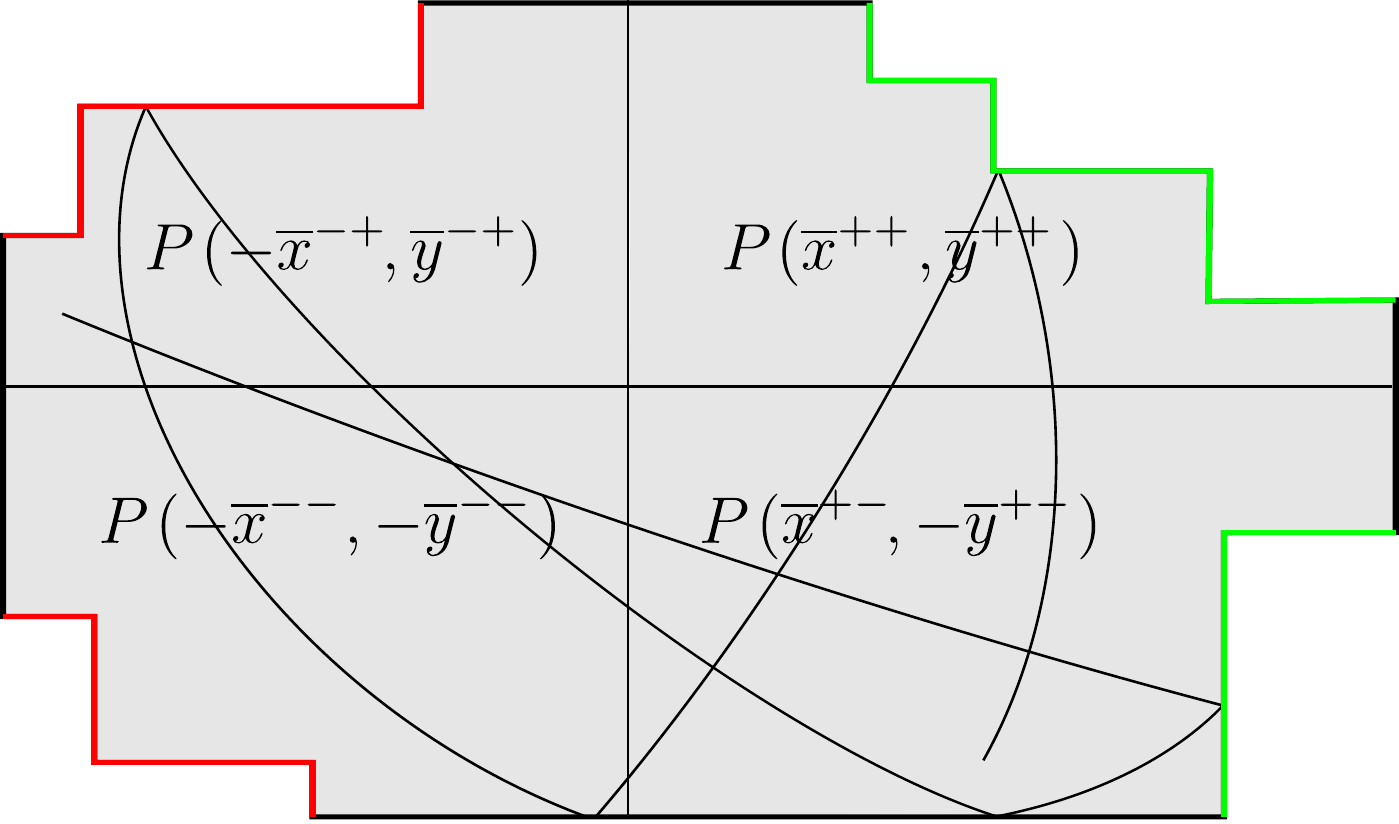}
\caption{The living space of the particle.}\label{fig:particleroom}
\end{figure}
The space of all such polygons we denote by $\mathscr{R}$. It is the set of star-shaped polygons with a kernel that includes the origin and with only vertical and horizontal boundaries. The corners
 $(x^{\varsigma_1,\varsigma_2}_{j},y^{\varsigma_1,\varsigma_2}_{j})$ are $90^o$ corners and are called hereafter  convex corners and
 $(x^{\varsigma_1,\varsigma_2}_{j},y^{\varsigma_1,\varsigma_2}_{j+1})$ are $270^o$ and are called hereafter  concave corners. The four dimensional vector  \(\{k(\overline{x}^{\varsigma_1,\varsigma_2},\overline{y}^{\varsigma_1,\varsigma_2})\}_{ \varsigma_1,\varsigma_2\in\{\pm\}}\) is called the \textit{topological data} of the polygon \(P\), whereas  the set of four vectors,  \(\{(\overline{x}^{\varsigma_1\varsigma_2},\overline{y}^{\varsigma_1\varsigma_2}), \varsigma_1,\varsigma_2\in\{\pm\}\}\) is called the  \textit{numerical data} of the polygon\footnote{We continue to call \( x^{\varsigma_1,\varsigma_2}_k\)/ \( y^{\varsigma_1,\varsigma_2}_k\) the length /height of a step, yet, notice that this is measuring the lengths and heights from the axes of the corresponding axes  and not of the full polygon.} \(P\).

\begin{figure}[h]
\includegraphics[width=0.6 \textwidth]{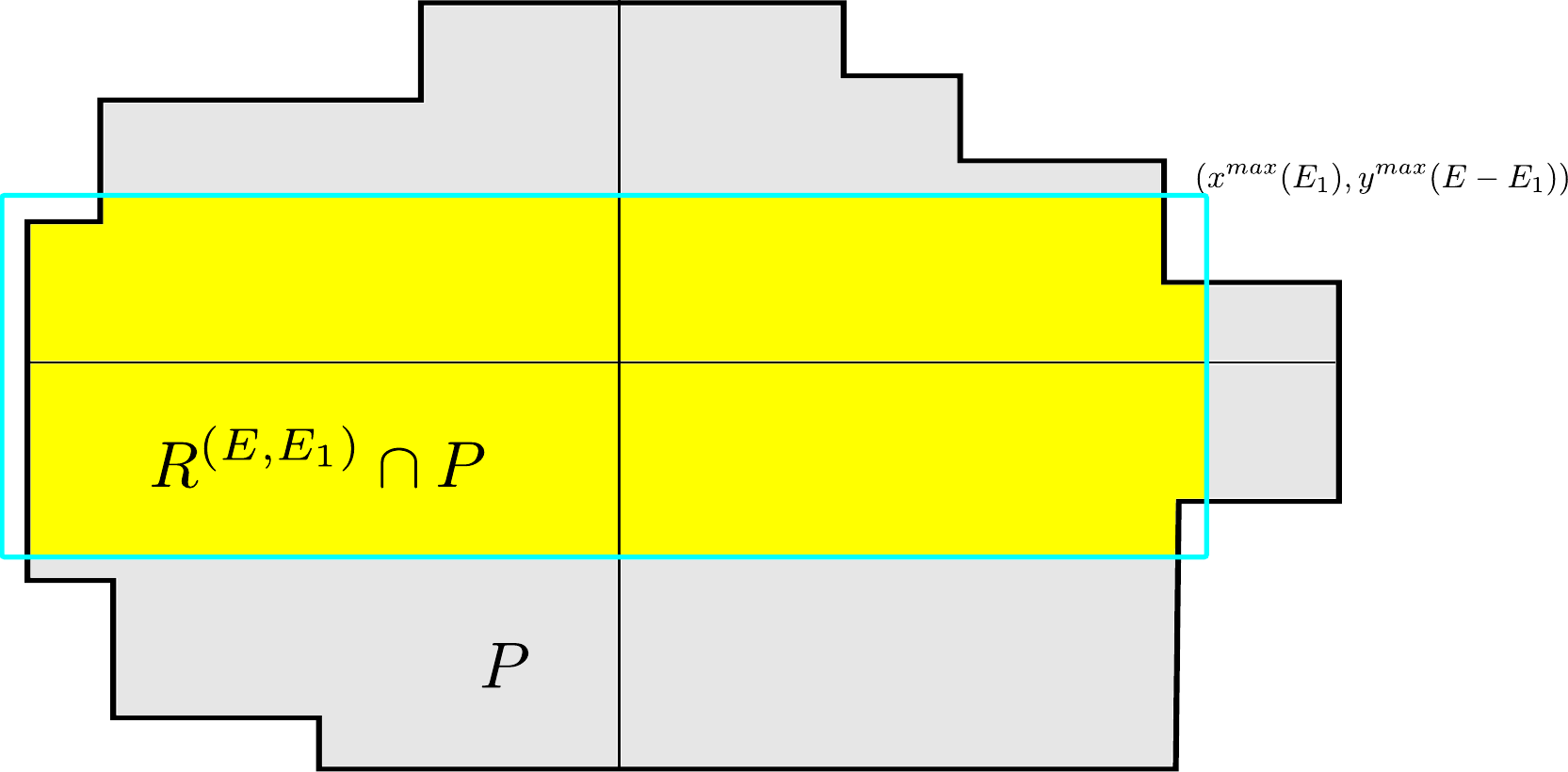}
\caption{A projected rectangle and the staircase polygon. Cyan - the projected rectangle. Yellow - the polygon  \(R^{(E,E_{1})}\cap P\) includes  3 concave corners, thus, by Theorem \ref{thm:genus}, \(g(E,E_{1})=4\).}\label{fig:particleroom2}
\end{figure}

The level set topology of \(S^{P}_{E,E_{1}}\)  is determined by the properties of \(R^{(E,E_{1})}\cap P\) (see Figure~\ref{fig:particleroom2}) and can be found under mild conditions on the potential:
 \begin{theorem}\label{thm:genus}
Assume $V_1,V_2:\R\to\R_{\geq 0}$ are $C^2$ uni-modal potentials (satisfy condition  \eqref{eq:i}). Let $P$ be any  polygon in $\mathscr R$.
Then for any  \(E>0\), for  \(E_{1}\in(0,E)\),  the genus of the level set \((E_{1}, E-E_1)\), \(S^{P}_{E,E_{1}}\), is given by one plus the number of concave corners in  \(R^{(E,E_{1})}\cap P\) : \begin{gather*}
g(E,E_1)=1+\sum_{\varsigma_1,\varsigma_2\in\{\pm\}}\#\{1\leq k<k(\bar{x}^{\varsigma_1\varsigma_2},\bar{y}^{\varsigma_1\varsigma_2}):V_1(x^{\varsigma_1\varsigma_2}_k)<E_1<E-V_2(y^{\varsigma_1\varsigma_2}_{k+1})\}.
\end{gather*} Specifically, for \(E> 0\) the interval \(E_{1}\in(0,E)\) is divided to a finite number of segments  on which the level sets have a constant genus. This partition is non-trivial for \(E>\min _{\varsigma_1,\varsigma_2,k}V_1(x^{\varsigma_1\varsigma_2}_k)+V_2(y^{\varsigma_1\varsigma_2}_{k+1}) \). Close to the end points of $(0,E)$ the genus is $1$ whereas for  \(E\) sufficiently large there exists an interval of level sets, \(E_{1}\in I_{max}\) with genus \(g_{max}=\sum_{\varsigma_1,\varsigma_2\in\{\pm\}}k(\bar{x}^{\varsigma_1\varsigma_2},\bar{y}^{\varsigma_1\varsigma_2})-3\) and on which for almost all \(E_{1}\) values the motion is uniquely ergodic.\end{theorem}

\begin{figure}[h]
\includegraphics[width=0.45 \textwidth]{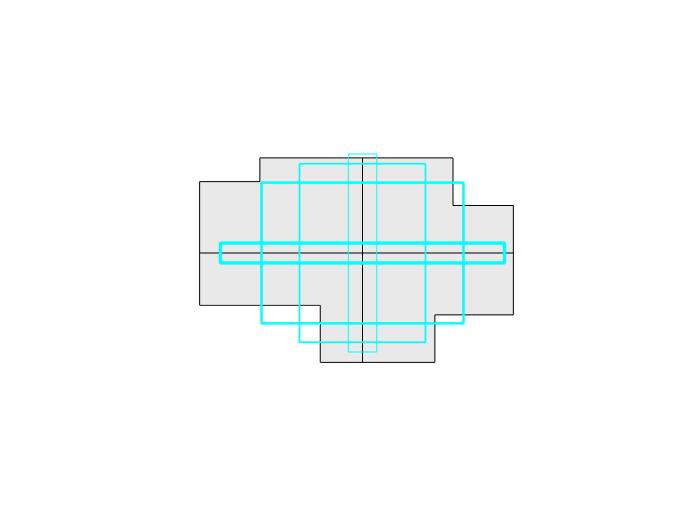}
\includegraphics[width=0.45 \textwidth]{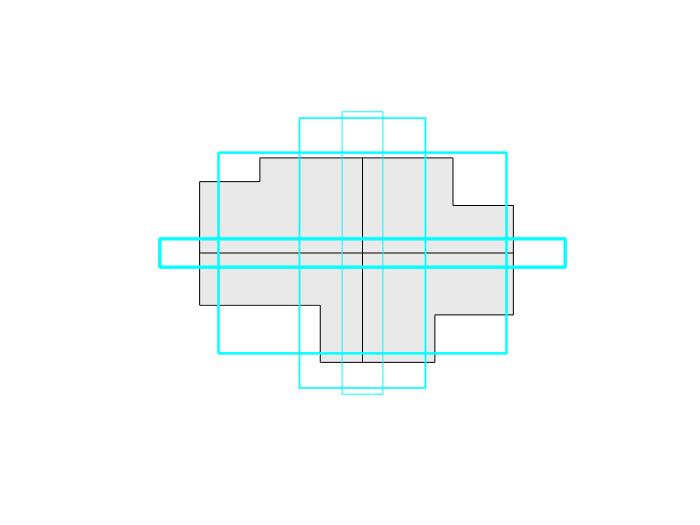}
\caption{The intersection of a star-shaped polygon (grey) with four iso-energy projected rectangles (cyan) at (a) \(E=2.7\)  and (b) \(E=5.7 \)   energy values.   By Theorem \ref{thm:genus}, the genus of the corresponding iso-energy level sets for the 4 shown \(R^{(E,E_{1})}\)  rectangles are, for increasing \(E_{1}\)  (a)  \( g(E=2.7,E_1)=\{1,2,4,1\}\) (b) \( g(E=5.7,E_1)=\{1,2,5=g_{max},1\}\). The potentials  here are quadratic with \(\omega_{1}=1,\omega_{2}=0.8\sqrt{2}\).}\label{fig:rectangleandpolygon}
\end{figure}

\begin{figure}[h]
\includegraphics[width=0.4 \textwidth]{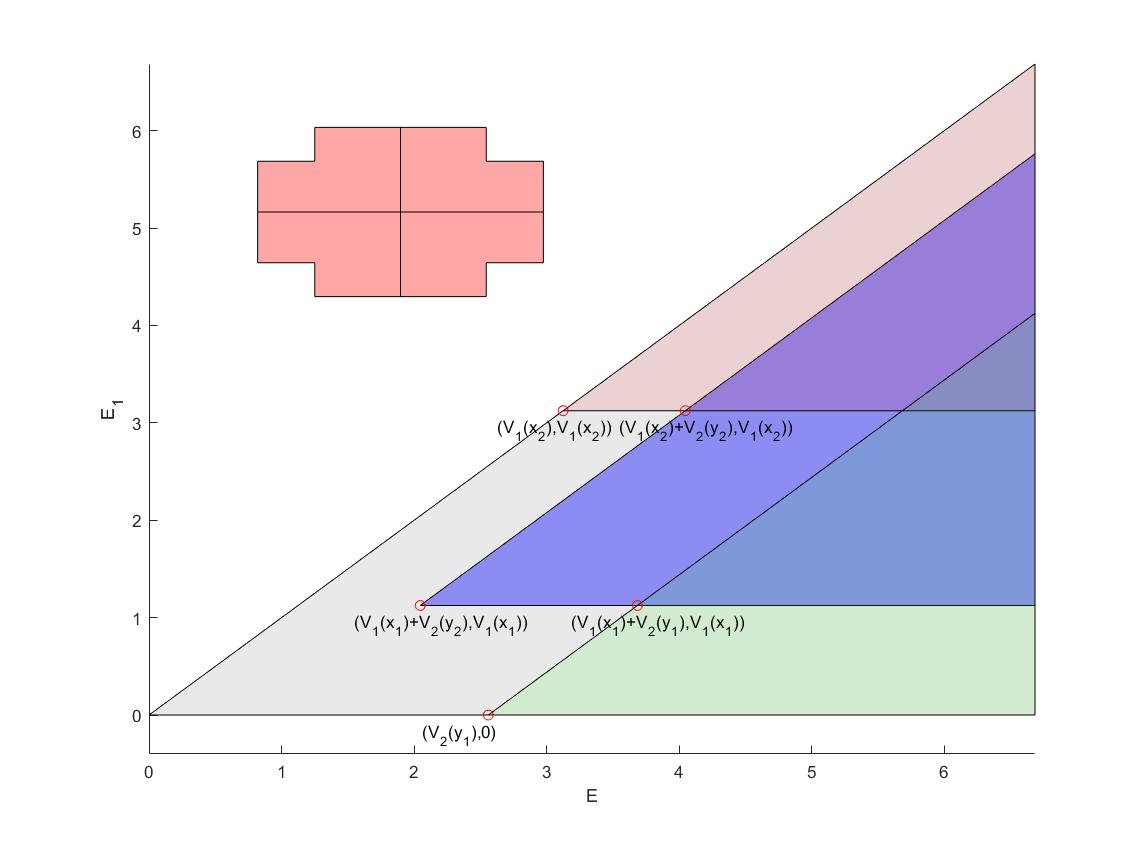}
\includegraphics[width=0.4 \textwidth]{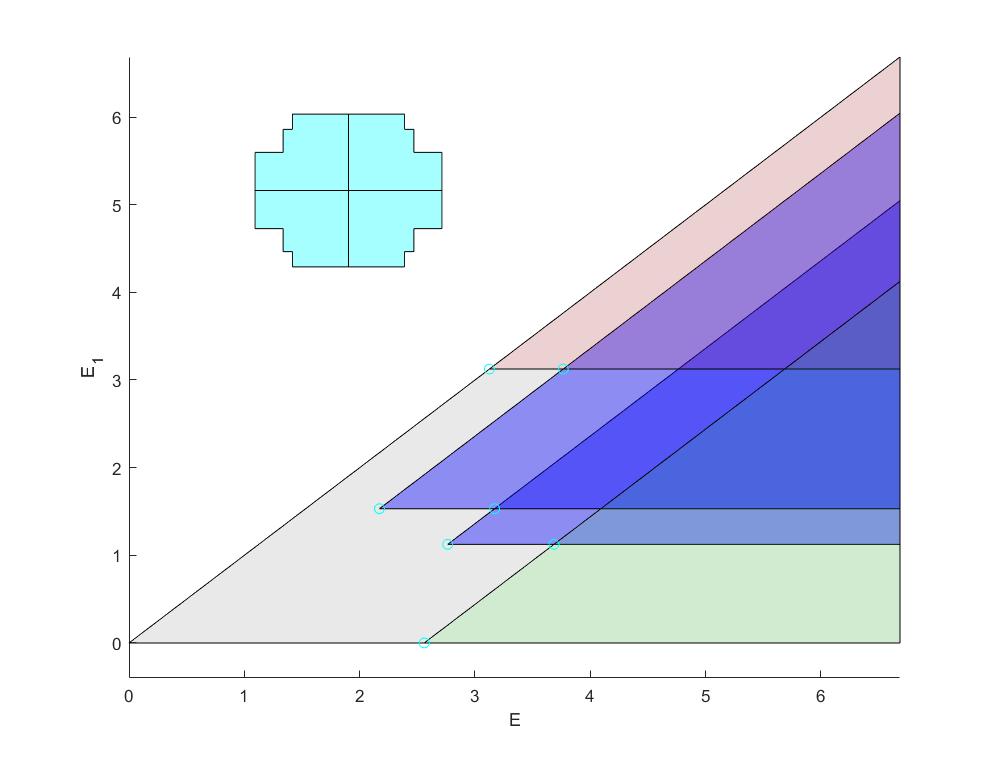}
\caption{Impact Energy-Momentum bifurcation diagram: (a) for a symmetric cross (one concave corner with multiplicity four), b)  symmetric two-steps cross (two concave corners each with multiplicity four). Here, due to symmetry - each blue wedge corresponds to 4 overlapping wedges, so  \(R^{(E,E_{1})}\cap P\) includes \(4k\) concave corners if and only if \((E,E_{1})\) is  in a region covered by \(4k\) overlapping shaded blue regions. Only for these regions the level sets genus is larger than $1$.  The shaded pink (respectively light-green) regions correspond to level sets that impact the extreme vertical (respectively horizontal) sides.}\label{fig:embdsym}
\end{figure}

\begin{figure}[h]
\includegraphics[width=0.6 \textwidth]{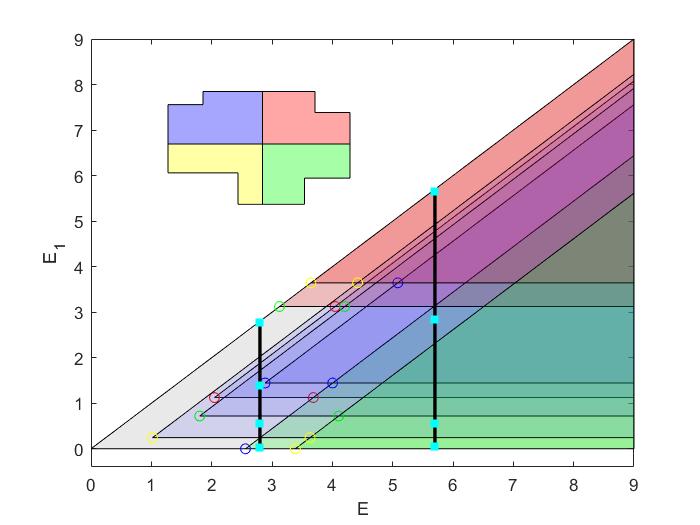}
\caption{Impact Energy-Momentum bifurcation diagram for the asymmetric cross shown in Figure~\ref{fig:rectangleandpolygon} (four distinct concave corners). The colored circles correspond to the energies of the corners of the corresponding colored staircase polygons shown in the inset. All corners have distinct partial energies, so    \(R^{(E,E_{1})}\cap P\) includes \(k\) concave corners if and only if \((E,E_{1})\) is  in a region covered by \(k\) overlapping shaded blue regions.  Only for these regions the level sets genus is larger than $1$. The shaded pink (respectively light green) regions include  extremal vertical (respectively horizontal) boundaries. The two vertical lines indicate the energies \(E=2.7,5.7 \) and the cyan squares on these lines correspond to the \(E_{1}\) values of the rectangles shown in  Figure~\ref{fig:rectangleandpolygon}.   }\label{fig:embd}
\end{figure}
Figures~\ref{fig:particleroom2} and~\ref{fig:rectangleandpolygon} provide the geometrical interpretation of this theorem in the configuration space, where several iso-energy rectangles are plotted on top of an asymmetric staircase polygon \(P\). Figure~\ref{fig:rectangleandpolygon} demonstrates  that the number of concave corners that are included in  \(R^{(E,E_{1})}\cap P\) can  vary  at a fixed energy.  Figures~\ref{fig:embdsym} and \ref{fig:embd}  demonstrate the division of the intervals \(E_{1}\in[0,E]\) of iso-energy level sets to a finite number of intervals, each having level sets with a fixed genus. The figures show how this division depends on the energy \(E\). In these plots, called Impact Energy-Momentum bifurcation diagrams (IEMBD, see \cite{pnueli2018near,PRk2020})   the regions in the  \((E,E_{1}) \)  plane at which level sets include impacts with certain parts of the boundaries are shown. Figure~\ref{fig:embdsym} shows these plots for symmetric polygons and Figure~\ref{fig:embd}  shows it for the asymmetric polygon of Figure~\ref{fig:rectangleandpolygon}. The grey wedge in the IEMBD  corresponds to all allowed level sets (since \(E_{1}\in[0,E]\)). A family of iso-energy level sets corresponds to a vertical line in this plot. The projected rectangles of the iso-energy level sets shown in Figure~\ref{fig:rectangleandpolygon} at two energies correspond to the cyan squares on the two vertical black  lines of Figure~\ref{fig:embd}. Each blue colored wedge in the IEMBD corresponds to level sets that include impacts with a concave corner of one of the polygons \(P(\overline{x}^{\varsigma_{1},\varsigma_{2}},
\overline{y}^{\varsigma_{1},\varsigma_{2}})\). In case \(j\) polygons have the same concave corner we say that this wedge has multiplicity \(j\), so in Figure~\ref{fig:embdsym}  each concave corner has multiplicity 4. The shaded pink (respectively light-green) regions correspond to level sets that impact the extreme vertical (respectively horizontal) boundaries a polygon. Regions that are in the complement to the blue wedges correspond to level sets with genus 1 (so in particular small and large (\(E_{1}\approx E\)) values are included in this set). Regions that are in \( k_g\) shaded blue wedges (counting multiplicities) have genus \( k_g+1\). Thus, the regions in the intersection of all the blue wedges have the maximal genus \(g_{max}\). The intersection of this region with  the two pink and two  light green wedges corresponds to level sets that, for almost all \(E_{1}\), have unique ergodic dynamics with the very mild assumption   (condition \eqref{eq:i}) on the potentials. As described next, with stronger assumptions on the potentials we prove unique ergodicity for almost all level sets in the allowed region, whereas for quadratic resonant potentials the IEMBD provides a more delicate division to segments as will be explained in Section~\ref{sec:loc}.
\subsection{Deck potentials}
To study the properties of the invariant measures of $(\varphi^{P,E,E_{1}}_t)_{t\in\R}$ we need additional assumptions on the potentials. For all $z_0\in \C$ and $r>0$  define the ball centered at \(z_0\) and the droplet emanating from \(z_0\) by:
\[B(z_0,r)=\{z\in \C:|z-z_0|<r\}\quad\text{and}\quad C(z_0,r)=\bigcup_{s\in(0,1]}sB(z_0,r)\;\text{if}\;r<|z_0|.\]
We define a special class of even potentials $V_1,V_2$ denoted by $Deck$. An even  $C^2$-map $V:\R\to\R$ belongs to $Deck$ if
\begin{gather}
\tag{$\diamondsuit$}\label{eq:i}
V(0)=V'(0)=0,\quad V'(x)>0\ \text{for all}\ x>0\ \text{and}\ \lim_{x\to+\infty}V(x)=+\infty;\\
\tag{$\heartsuit$}\label{eq:ii}
V:(0,+\infty)\to(0,  +\infty)\quad\text{is an analytic map;}
\end{gather}
{then $V:(0,+\infty)\to(0,  +\infty)$ has a holomorphic extension $V:U\to\C$ on an open neighborhood $U\subset\C$ of $(0,+\infty)$;
\begin{gather}
\tag{$\clubsuit$}\label{eq:iii}
\begin{aligned}
&\text{for every $E>0$ there exist $0<r<E$ and a bounded open set $U_{E}\subset U$}\\
&\text{such that $V:U_E\to V(U_E)$
is biholomorphic with $C(E,r)\subset V(U_E)$;}
\end{aligned}\\
\tag{$\spadesuit$}\label{eq:iv}
\text{there exists}\ C_{E}>0 \ \text{such that} \
\left|\frac{V''(z)V(z)}{(V'(z))^2}\right|\leq C_E\ \text{ for all}\ z\in U_E.
\end{gather}
}

The class $Deck$ contains all unimodal analytic maps (i.e.\ satisfying \eqref{eq:i}), see Proposition~\ref{prop:anDeck}.
Further examples of $Deck$-potentials which are not analytic at $0$ (such as \(V(x)=|x|^m\exp(-1/|x|)\)) are presented in Section~\ref{sec:Deckpot}.
The Deck assumption insure that the period depends analytically on the energy.
Additionally, we will most often assume that $V\in Deck$ also satisfies
\begin{gather}\tag{$\smiley$}\label{eq:v}
\frac{V(x)V''(x)}{(V'(x))^2}\geq \frac{1}{2}\ \text{for all}\ x>0,\ \text{or}\\
\tag{$\sun$}\label{eq:vi}
V\ \text{satisfies \eqref{eq:v} and}\ \frac{VV''}{(V')^2}\neq \frac{1}{2}.
\end{gather}
The condition \eqref{eq:v} is equivalent  to \(V\) being the square of a convex function, and insures that the period is a decreasing function of the energy, while the condition \eqref{eq:vi} means additionally not being a quadratic function, so the period is strictly decreasing with the energy, see Lemma~\ref{lem:negposder}~and~\ref{lem:square}. For example, all non-trivial non-quadratic even polynomials with non-negative coefficients are Deck and satisfy \eqref{eq:vi}, whereas   \(V(x)=x^{2}-\sqrt{2}x^4+x^6=(x+x^3)^2-(2+\sqrt{2})x^4\) is Deck but does not satisfy   \eqref{eq:v} (cf.\ Proposition~\ref{prop:smpot}).

\subsection{Main results for non-quadratic Deck potentials}

The main result (Theorem~\ref{thm:albega}) says that for every energy level $E>0$ and typical (a.e.) partial energy $E_{1}\in[0,E]$ the local flow $(\varphi^{P,E,E_{1}}_t)_{t\in\R}$ is uniquely ergodic whenever at least one potential is not a quadratic function or both are quadratic functions and non-resonant (i.e.\ \(\Omega\) is irrational\footnote{Recall that when both potentials are quadratic the ratio of frequencies $\Omega(E,E_1)$ does not depend on $E$ and $E_1$.}):

\begin{theorem}\label{thm:albega}
Assume $V_1,V_2:\R\to\R_{\geq 0}$ are $Deck$-potentials satisfying \eqref{eq:v}. Let $P$ be any  polygon in $\mathscr R$.
Suppose that
\begin{itemize}
\item[($\alpha$)] at least one potential $V_1$ or $V_2$ satisfies \eqref{eq:vi} or;
\item[($\beta$)] both $V_1,V_2$ are quadratic maps such that $V_1=\Omega^{2} V_2$ with $\Omega$ irrational.
\end{itemize}
Then for every energy level $E>0$ and almost every $E_1\in [0,E]$ the restricted Hamiltonian flow
$(\varphi^{P,E,E_1}_t)_{t\in\R}$ is uniquely ergodic.
\end{theorem}

Notice that for small energy (i.e.\ satisfying \eqref{eq:lowE}), and, in fact, for all the grey areas in the IEMBD figures (Figures~\ref{fig:embdsym} and \ref{fig:embd}), the motion does not impact the polygon walls and the above theorem trivially holds as the motion on most of the tori is of irrational rotation {(in the non-quadratic case, \((\alpha)\), the iso-energy non-degeneracy condition holds since  \(\omega_{i}'(I_i)\geqslant0,i=1,2\) and at least for one oscillator the inequality is strict).} The non-trivial statement is that even when impacts occur (the pink, light-green and blueish regions in the IEMBD figures), the flow is usually uniquely ergodic. \ By the definition of unique ergodicity, the  theorem tells us that for most level sets time averages are equivalent to phase-space average for every initial condition on these level set. The complementary set could have periodic and quasi-periodic motion coexisting on the same level set as in \cite{McM}, \cite{Issi2019}.

\subsection{Linear oscillators case} \label{sec:loc} We study separately the case when $V_1$ and $V_2$ are quadratic. This boundary case (in the class of $Deck$-potentials satisfying \ref{eq:v}) is significantly different from the general case. In the quadratic case we have  $V_1=\Omega^{2} V_2$, and we consider rational $\Omega$ so the harmonic motion is resonant.

 If the energy level $E$ is low enough then, the {impacting resonant quadratic  flow,} $(\varphi^{P,E,E_1}_t)_{t\in\R}$ does not reach the boundary and the motion is {trivially} identical to the resonant periodic linear oscillator motion (the grey area in the IEMBD figures, before any of the blue wedges emerge):

\begin{proposition}\label{prop:lowE}
For
 energies satisfying\begin{equation}\label{eq:lowE}
E\leq\min\{V_1(x^{\varsigma_1\varsigma_2}_k)+V_2(y^{\varsigma_1\varsigma_2}_{k+1}):\varsigma_1,\varsigma_2\in\{\pm\},0\leq k\leq k(\bar{x}^{\varsigma_1\varsigma_2},\bar{y}^{\varsigma_1\varsigma_2})\},
\end{equation}
where $x^{\varsigma_1\varsigma_2}_0=y^{\varsigma_1\varsigma_2}_{k(\bar{x}^{\varsigma_1\varsigma_2},\bar{y}^{\varsigma_1\varsigma_2})+1}=0$,
 the restricted Hamiltonian flow
$(\varphi^{P,E,E_1}_t)_{t\in\R}$ is identical to
$(\varphi^{E,E_1}_t)_{t\in\R}$ for all $E_1\in[0,E]$ and, for the impacting resonant quadratic  flow it corresponds to periodic motion.
\end{proposition}
On the other hand,
{a non-trivial statement, with a proof which is similar to that of Theorem \ref{thm:albega},} is that if the energy level $E$ is high enough so that at least one of the extremal horizontal or vertical boundaries is reached  by orbits in $S^P_{E,E_1}$ for all $E_1\in[0,E]$ then $(\varphi^{P,E,E_1}_t)_{t\in\R}$ is uniquely ergodic for a.e.\ $E_1$ ({energies beyond which the segment \([0,E] \) is covered by the union of  the pink and light green wedges in the IEMBD figures, e.g. \(E \approx5.7 \) in Figure   \ref{fig:embd}, see also the corresponding projected rectangles in Figure  ~\ref{fig:rectangleandpolygon}b}):

\begin{theorem}\label{thm:highE}
If
\begin{equation}\label{eq:bigEmin}
E \geq\min_{\varsigma_1,\varsigma_2\in\{\pm\}}V_1(x_{k(\bar{x}^{\varsigma_1\varsigma_2},\bar{y}^{\varsigma_1\varsigma_2})}^{\varsigma_1\varsigma_2})
+\min_{\varsigma_1,\varsigma_2\in\{\pm\}}V_2(y_1^{\varsigma_1\varsigma_2})
\end{equation}
then the impacting resonant quadratic  flow
$(\varphi^{P,E,E_1}_t)_{t\in\R}$ is uniquely ergodic for a.e.\ $E_1\in[0,E]$.
\end{theorem}
Studying the case of intermediate  $E$  is much more complex {and requires new methods.}
Then the interval $[0,E]$ splits into at most countably many intervals so that
for every interval $I$ from this partition we have three possible scenarios:
\begin{itemize}
\item[\textbf{(ue)}] the flow $(\varphi^{P,E,E_1}_t)_{t\in\R}$ is uniquely ergodic for a.e.\ $E_1\in I$;
\item[\textbf{(cp)}] the flow $(\varphi^{P,E,E_1}_t)_{t\in\R}$ is completely periodic for all $E_1\in I$;
\item[\textbf{(coex)}] for a.e.\ $E_1\in I$ the phase space of $(\varphi^{P,E,E_1}_t)_{t\in\R}$ splits into two completely periodic cylinders and two uniquely ergodic components.
\end{itemize}

For  $E>0$ let  $\mathcal{J}_E$ denote a partition (into open intervals) of the interval $[0,E]$ determined by the numbers
\[V_1(x^{\varsigma_1\varsigma_2}_k), E-V_2(y^{\varsigma_1\varsigma_2}_k)\text{ for all }\varsigma_1,\varsigma_2\in\{\pm\},1\leq k\leq k(\bar{x}^{\varsigma_1\varsigma_2},\bar{y}^{\varsigma_1\varsigma_2}).\]
 {In the IEMBD figures,   $\mathcal{J}_E$ corresponds to the partition of the vertical interval $[0,E]$ by the colored wedges. }
 Now we formulate two results relating to the cases \textbf{(cp)} and \textbf{(ue)}.
For low energies, there are intervals of \(E_{1}\) values for which no impacts occur (these are the grey regions in the IEMBD figures):
\begin{proposition}\label{prop:formIcp}
Suppose that $I\in\mathcal J_E$ is an interval such that for every $\varsigma_1,\varsigma_2\in\{\pm\}$
there exists $1\leq {l}^{\varsigma_1\varsigma_2}
\leq {k}(\bar{x}^{\varsigma_1\varsigma_2},\bar{y}^{\varsigma_1\varsigma_2})$  such that
\begin{equation}\label{eq:formIcp}
I\subset \bigcap_{\varsigma_1,\varsigma_2\in\{\pm\}}\big[E-V_2(y^{\varsigma_1\varsigma_2}_{{l}^{\varsigma_1\varsigma_2}}),V_1(x^{\varsigma_1\varsigma_2}_{{l}^{\varsigma_1\varsigma_2}})\big].
\end{equation}
Then  the impacting resonant quadratic  flow
$(\varphi^{P,E,E_1}_t)_{t\in\R}$ is completely periodic for every $E_1\in I$.
\end{proposition}
Denote by \(I_{nonimp}(E)\) the collection of \(E_1\) intervals on which no impacts occur:
 \begin{equation}
I_{nonimp}(E)=\bigcup_{\substack{
1\leq {l}^{++} \leq {k}(\bar{x}^{++},\bar{y}^{++})\\
1\leq {l}^{+-} \leq {k}(\bar{x}^{+-},\bar{y}^{+-})\\
1\leq {l}^{-+} \leq {k}(\bar{x}^{-+},\bar{y}^{-+})\\
1\leq {l}^{--} \leq {k}(\bar{x}^{--},\bar{y}^{--})
}}
\bigcap_{\varsigma_1,\varsigma_2\in\{\pm\}}\big[E-V_2(y^{\varsigma_1\varsigma_2}_{{l}^{\varsigma_1\varsigma_2}}),V_1(x^{\varsigma_1\varsigma_2}_{{l}^{\varsigma_1\varsigma_2}})\big].
\label{eq:inonimpact}\end{equation}
For sufficiently large \(E \) this set is empty, whereas for sufficiently small \(E,\ I_{nonimp}(E)= [0,E]\).
For intermediate values  \(I_{nonimp}(E)\) may be composed of several disjoint intervals (e.g.  the grey segments for \(E=3 \) in Figure ~\ref{fig:embdsym}).

When at least one of the extremal boundaries is reached (the union of the pink and light green wedges in the IEMBD figures), {similar to the general case of Theorem \ref{thm:albega},} we get again unique ergodicity for a.e. \(E_{1}\):
\begin{theorem}\label{thm:highmedH}
For every energy level $E>0$ and almost every
\[E_1\in \Big[0,E-\min_{\varsigma_1,\varsigma_2\in\{\pm\}}V_2(y_1^{\varsigma_1\varsigma_2})\Big]
\cup \Big[\min_{\varsigma_1,\varsigma_2\in\{\pm\}}V_1(x_{k(\bar{x}^{\varsigma_1\varsigma_2},\bar{y}^{\varsigma_1\varsigma_2})}^{\varsigma_1\varsigma_2}),E\Big]\]
the  impacting resonant quadratic  flow
$(\varphi^{P,E,E_1}_t)_{t\in\R}$  is uniquely ergodic.
\end{theorem}

Now assume that $I\in\mathcal J_E$ is an interval such that impacts occur and not with the extremal boundaries (pure blue regions in the IEMBD figures):
\begin{equation}\label{eq:assumInti}
I\subset I_{intimp}  =\big(E-\min_{\varsigma_1,\varsigma_2\in\{\pm\}}V_2(y_1^{\varsigma_1\varsigma_2}),\min_{\varsigma_1,\varsigma_2\in\{\pm\}}V_1(x_{k(\bar{x}^{\varsigma_1\varsigma_2},\bar{y}^{\varsigma_1\varsigma_2})}^{\varsigma_1\varsigma_2})\big)\backslash I_{nonimp}(E)
\end{equation}
 {This case presents non-uniform ergodic properties and requires new constructions}. Assume that $\Omega=n/m$ with $m$, $n$ coprime natural numbers. Let $\{red,green\}$
be a partition of the \((\varsigma_1,\varsigma_2)\)  set $\{++,+-,-+,--\}$ into two-element set so that
\begin{equation}\label{eq:green}
green=(pair1,pair2)=\left\{
\begin{array}{cl}
\{++,+-\}&\text{ if $m$ is odd and $n$ is even,}\\
\{++,-+\}&\text{ if $m$ is even and $n$ is odd,}\\
\{++,--\}&\text{ if $m$ and $n$ are odd.}
\end{array}
\right.
\end{equation}
Recall that a billiard in a rectangle can be reflected 3 times so that the billiard flow is conjugated to the directed flow on the torus. By rescaling we can always consider the flow to be in the direction $45^o$. Let the left lower corner of the original rectangle denote by \((--)\), the upper left corner by  \((-+)\), the  lower right  corner  \((+-)\) and the upper right corner by \((++)\) (see Figure~\ref{fig:surface} of Section~\ref{sec:billtotrans}).   In the rational situation, when the impacts from the boundary of the polygon $P$ are ignored, the torus is filled with periodic orbits. We will see that the partition corresponds to having a periodic orbit of the scaled torus which connects the \(pair 1\) and \(pair2\) corners.  These coloured periodic  orbits induce colouring of the staircase non extremal boundaries of \(P(\overline{x}^{\varsigma_1\varsigma_2},\overline{y}^{\varsigma_1\varsigma_2})\), which we call coloured sides (see Figure~\ref{fig:particleroom}, {with colouring induced by taking $m$ odd and $n$  even as in the figures of Section~\ref{sec:quadrat}).}

 Next for every $colour\in\{green,red\}$ let
\begin{align*}
\delta^{colour}(E_1):
=\max_{\substack{\varsigma_1\varsigma_2\in colour\\
1\leq k< {k}(\bar{x}^{\varsigma_1\varsigma_2},\bar{y}^{\varsigma_1\varsigma_2})\\
V_1(x_k^{\varsigma_1\varsigma_2})<E_1<E-V_2(y_{k+1}^{\varsigma_1\varsigma_2})}}
\Big(m\arccos\sqrt{\frac{V_1(x_k^{\varsigma_1\varsigma_2})}{E_1}}+n\arccos\sqrt{\frac{V_2(y_{k+1}^{\varsigma_1\varsigma_2})}{E-E_1}}\Big).
\end{align*}
We will see that \(\delta^{colour}(E_1)\) is related to the measure of orbits that impact the coloured sides of the polygon (see details in Section~\ref{sec:quadrat}).
 Since $\delta^{green}(E_1)+\delta^{red}(E_1)=\pi$ for at most countably many $E_1\in I$ (see Lemma~\ref{lem:dotyk}),
the interval $I$ has a partition into open intervals of two kinds $\mathcal U^+_I$ and $\mathcal U^-_I$ so that
\begin{align}\label{eq:uipm}
\begin{split}
\delta^{green}(E_1)+\delta^{red}(E_1)>\pi&\text{ for all $E_1\in J$ if $J\in \mathcal U^+_I$,  }  \\
\delta^{green}(E_1)+\delta^{red}(E_1)<\pi&\text{ for all $E_1\in J$ if $J\in \mathcal U^-_I$. }
\end{split}
\end{align}
 We will see that in the first case, there are no non-impacting orbits, whereas in
the second case there are also periodic orbits which do not impact any coloured side. In this latter case, the impacting orbits can be divided to two separate sets, the red/green set,  consisting of orbits impacting the  red/green sides.

The following theorem asserts that in the first case, for almost all \(E_{1}\), all orbits are equi-distributed, whereas in the second case, for almost all  \(E_{1}\), all orbits that impact the red/green sides are equi-distributed among the red/green set:
\begin{theorem}\label{thm:U+-}
Let \(I\subset  I_{intimp}\) of Eq.~\eqref{eq:assumInti} and divide \(I\) to the subintervals \( \mathcal U^\pm_I \) satisfying (\ref{eq:uipm}). If $J\in \mathcal U^+_I$ then   the  impacting resonant quadratic  flow
$(\varphi^{P,E,E_1}_t)_{t\in\R}$ is uniquely ergodic for a.e.\ $E_1\in J$.
If $J\in \mathcal U^-_I$ then the phase space of  the  flow
$(\varphi^{P,E,E_1}_t)_{t\in\R}$ splits for all  \ $E_1\in J$ into four components, two of which are completely periodic and  for a.e.\ $E_1\in J$ the two other components are uniquely ergodic.
\end{theorem}
This result completes the description of invariant measure of the restricted Hamiltonian flow $(\varphi^{P,E,E_1}_t)_{t\in\R}$
for every $E>0$ and almost every $E_1\in[0,E]$, when $V_1$, $V_2$ are quadratic potentials with $V_1=\Omega^{2} V_2$ and $\Omega$ is rational.
We discuss the non-trivial implications of this theorem on ergodic averages of the  impacting resonant quadratic  flow in Section~\ref{subsec:ergodicav}.

\subsection{Strategy of the proof}
In Section~\ref{sec:osctobil}, using a standard change of coordinates (as observed in \cite{Issi2019}) we construct an isomorphism between the restricted Hamiltonian flow
$(\varphi^{P,E,E_1}_t)_{t\in\R}$ and the directional billiard flow in direction $\pm \pi/4,\pm 3\pi/4$ on a polygon $\mathbf P_{E,E_1}\in\mathscr R$.
For every $E>0$ this gives a piecewise analytic curve $E_1\in[0,E]\mapsto \mathbf P_{E,E_1}\in\mathscr R$ of billiard flows on polygons in $\mathscr R$.
More precisely, this curve is analytic on every interval $I\in \mathcal J_E$. The unique ergodicity problem for this type of curves was recently
studied by the first author in \cite{frkaczek2019recurrence}. In fact, a slight modification of Theorem~4.2 in \cite{frkaczek2019recurrence} (see Theorem~\ref{thm:mainfr})
is applied to curves $E\ni E_1\mapsto \mathbf P_{E,E_1}\in\mathscr R$ to show unique ergodicity $(\varphi^{P,E,E_1}_t)_{t\in\R}$ for a.e.\ $E_1\in[0,E]$
whenever at least one potential $V_1$ or $V_2$ is not a quadratic function or both are quadratic functions with $\Omega$ irrational,
see Theorem~\ref{thm:albega} in Section~\ref{sec:genequ}. Theorem~\ref{thm:mainfr} relies on the analysis of functions indicating the width, height, length and height of steps of staircase polygons that make up the polygon $\mathbf P_{E,E_1}$, when the parameter $E_1$ varies.
The relevant results involving these functions necessary for  applications of Theorem~\ref{thm:mainfr}  are presented and proved in Section~\ref{sec:propertiesa}.

The case when $V_1$, $V_2$ are quadratic with $V_1=\Omega^{2} V_2$ and $\Omega$ is rational needs a more subtle version of Theorem~\ref{thm:mainfr}, this is Theorem~\ref{thm:mainquadrational}. Recall that any directional billiard flow on any right angle polygon is isomorphic to the
translation flow on a translation surface obtained using so called unfolding procedure from the polygon. This leads to the study of analytic curves
of translation surfaces and their translation flows in a fixed direction. {Theorem~\ref{thm:mainquadrational} gives a criterion for unique ergodicity
of the translation flow for almost every translation surface lying on such a curve (this theorem is set in an abstract framework to allow other applications). While the idea of the proof of Theorem~\ref{thm:mainquadrational}
is similar to Theorem~4.2 in \cite{frkaczek2019recurrence}, it needs more subtle reasoning, as it involves a new type of partition  of the translation surface into polygons with sides that can be parallel to the direction of the flow. This is the main innovation in relation to the approach used in Theorem~\ref{thm:mainfr} (and in \cite{frkaczek2019recurrence}).  Another important novelty is the use of so-called distinguished sides of partitions.
The key assumption of Theorem~\ref{thm:mainquadrational} is that every orbit of the directional flow (finite or half-infinite or double-infinite) hits a distinguished side, and the key construction in proving Theorem \ref{thm:U+-} is of a glued surface for which this assumption holds.}

\section{Oscillations in one dimension}\label{sec:osc1}
 Assume that $V_{1}(x):\R\to\R_{\geq 0}$ is an even $C^2$-potential satisfying \eqref{eq:i}.
Then $x^{max}(E_{1}):[0, +\infty)\to[0,+\infty)$, the inverse of the positive branch of $V_1$, is continuous and $C^2$ on $(0, +\infty)$ with $x^{max}(0)=0$ and $(x^{max})'(E_{1})>0$ for $E_{1}>0$. Similar definitions apply to \(y^{max}(E_{2})=y^{max}(E-E_{1})\).

Fix an energy level $E_1>0$. The particle oscillates in the interval $[-x^{max}(E_{1}),x^{max}(E_{1})]$ wandering between the ends back and forth. We change the space coordinate to obtain a new isomorphic model of the oscillation in which the mass point moves periodically with \(\omega_{1}(E_1)>0\)  speed on the interval \( \psi_1\in[-\frac{\pi}{2},\frac{\pi}{2}]\):
\begin{equation}
\frac{d\psi_1}{dt}=\sgn(p_x)\omega_{1}(E_1),\:\psi(0)=0,\:\psi_1(\pm x^{max}(E_{1}))=\pm\frac{\pi}{2}.
\label{eq:psiprop}
\end{equation}
We call these coordinates action-angle-like coordinate, as they are simply related to the transformation to action angle  coordinates, see \cite{Issi2019}. Using the symmetry of \(V_{1}\) and the notation \(p_{x}(E_{1},x)=\pm\sqrt{2}\sqrt{E_1-V_{1}(x)}\) for denoting the dependence of \(p_x\) on position and energy, we have:
  \begin{equation}\label{def:psi}
\psi_{1}(x,E_{1})=\frac{2\pi}{T_{1}(E_1)}\int_0^{x}\frac{1}{|p_{x}(E_{1},s)|}ds=\omega_{1}(E_1)\int_0^{x}\frac{1}{\sqrt{2}\sqrt{E_1-V_{1}(s)}}ds,
\end{equation}
where \(T_{1}(E_1),\omega_{1}(E_1)\) are the period and frequency of the periodic flow on the \(E_{1}\) level set, i.e.:
\begin{equation}\label{def:a}
\frac{1}{4}T_{1}(E_1)=\frac{\pi }{2\omega_{1}(E_1)}=\int_0^{x^{max}(E_{1})}\frac{1}{|p_{x}(E_{1},s)|}ds=\int_0^{x^{max}(E_{1})}\frac{1}{\sqrt{2}\sqrt{E_1-V_{1}(s)}}ds.
\end{equation}
satisfies Eq.~\eqref{eq:psiprop}. Now suppose additionally that our oscillator meets an elastic barrier at a point $x^{wall}>0$. Then its trajectories are described by the equations (\ref{eq:Ham2})  if $x\leq x^{wall}$
according to the rule that if a trajectory meets a point $(p_{x},x^{wall})$ then it jumps to $ (-p_{x},x^{wall})$ and continues its movement in accordance with (\ref{eq:Ham2}).
Thus, if  $x^{wall}<x^{max}(E_{1})$ the particle  oscillates in the interval $[-x^{max}(E_{1}), x^{wall}]$ and after changing the space coordinate to the action-angle-like coordinate, \(\psi_{1}\) it oscillates with the   \(\omega_{1}(E_1)\)  speed on the interval
\( \psi_1\in[-\frac{\pi}{2},\psi_{1}(x^{wall},E_{1})]\) with elastic reflections from the ends. The maps
\(\psi_{2}(y,E_{2})\) and \(T_{2}(E_2)\)  are similarly defined.

\section{From oscillations in dimension two to billiards on polygons}\label{sec:osctobil}

{Recall that the motion in configuration space on a given level set is restricted to the  polygon  \( P\cap R^{(E,E_{1})}\). Using the transformation to the \(\psi\) coordinates, we find the topological and numerical data of the corresponding polygon in the  \(\psi\) space.}

We consider the Hamiltonian flow \ref{eq:Ham2} restricted to the polygon $P\in\mathscr R$: \[P=P(\overline{x}^{++},\overline{y}^{++}) \cup P(-\overline{x}^{-+},\overline{y}^{-+})\cup
P(\overline{x}^{+-},-\overline{y}^{+-})\cup
P(-\overline{x}^{--},-\overline{y}^{--}),\]
where  $(\overline{x}^{\varsigma_1\varsigma_2},\overline{y}^{\varsigma_1\varsigma_2})\in\Xi$ for $\varsigma_1,\varsigma_2\in\{\pm\}$ and we are interested in the properties of the flow restricted to iso-energy level sets \(E_{1},E_{2}=E-E_1\) which we denote by $(\varphi^{P,E,E_1}_t)_{t\in\R}$.
Fix $E>0$ and $0<E_1<E$. By the definition of $S^P_{E,E_1}$  the level set is contained in
\[\R^2\times R^{(E,E_1)}=\R^2\times[-x^{max}(E_{1}),x^{max}(E_{1})]\times[-y^{max}(E-E_{1}),y^{max}(E-E_{1})] .\]
Let us consider new coordinates on \(R^{(E,E_1)}\) given by
\begin{align*}
  \psi(x,y) & =(\psi_{1}(x,E_{1}),\psi_{2}(y,E-E_{1}))
\end{align*}
and notice that \(\frac{d\psi(x(t),y(t))}{dt}=(\sgn(p_x(t))\omega_{1}(E_1),\sgn(p_y(t))\omega_{2}(E-E_1))\).
It follows that the flow $\big(\varphi^{P,E,E_1}_t\big)_{t\in\R}$ in the new coordinates coincides
with the directional billiard flow on
\[\mathbf{P}_{E,E_1}:=\psi \left(P\cap R^{(E,E_1)}\right)\]
so that the directions of its orbit are: $(\pm\omega_{1}(E_1),\pm\omega_{2}(E-E_1))$. As
\(P\cap R^{(E,E_1)}\in\mathscr R\)
and $\psi$ sends vertical/horizontal segments to vertical/horizontal segments, we have $\mathbf{P}_{E,E_1}\in\mathscr R$, namely:
\[\mathbf{P}_{E,E_1}=\mathbf{P}^{++}_{E,E_1}\cup \mathbf{P}^{+-}_{E,E_1}\cup \mathbf{P}^{-+}_{E,E_1}\cup \mathbf{P}^{--}_{E,E_1}\]
Notice that the number  of corner points in each quadrant of \(P\cap R^{(E,E_1)}\) and of \(\mathbf{P}_{E,E_1}\) are identical, whereas their dimensions are related by the transformation \(\psi\) (which depends on \(E_{1}\) and \(E\)). We need to find these dimensions to determine the properties of the flow.

It is convenient to first rescale \(\mathbf{P}_{E,E_1}\) so that the directional motion occurs in the standard directions \((\frac{\pi}{4},-\frac{\pi}{4},\frac{3\pi}{4},-\frac{3\pi}{4})\). Thus we scale
\[\hat \psi_{1}(x,E_{1})=\frac{\psi_{1}(x,E_{1})}{\omega_{1}(E_1)},\quad\hat \psi_{2}(x,E-E_{1})=\frac{\psi_{2}(y,E-E_{1})}{\omega_{2}(E-E_1)},\] then \(\mathbf{P}_{E,E_1}\) is scaled to \( \mathbf{\hat P}_{E,E_1}\) and for shorthand we omit hereafter the hats.
After the rescaling,

  \begin{equation}\label{def:psires}
\psi_{1}(x,E_{1})=\int_0^{x}\frac{1}{\sqrt{2}\sqrt{E_1-V_{1}(s)}}ds\ \text{ for }\ E_1\geq V_1(x),
\end{equation}
so \(\psi_{1}(x_{max}(E_{1}),E_{1})=\frac{1}{4}T_1(E_1)\) and, similarly,  \(\psi_{1}(x,V_{1}(x))=\frac{1}{4}T_1(V_{1}(x)).\)

  The {topological and} numerical data of the polygon $\mathbf{P}_{E,E_1}\in\mathscr R$ for any $E>0$ and $E_1\in(0,E)$, namely, the number of corners it has in each quadrant and their locations in the scaled \(\psi\) plane, is found by computing the sequences \(\bar \Psi_{i}^{\varsigma_1,\varsigma_2}(E,E_1)\) (\(i=1,2\)) of the corner points of \(\mathbf{P}^{\varsigma_1,\varsigma_2}_{E,E_1}\).

{By the definition of staircase polygons and condition  \eqref{eq:i}, the sequence \(\{V_1({x}_j^{\varsigma_1\varsigma_2})\}_{j=1}^{k(\bar{x}^{\varsigma_1\varsigma_2},\bar{y}^{\varsigma_1\varsigma_2})}\) is monotonically increasing and  \(\{V_2({y}_j^{\varsigma_1\varsigma_2})\}_{j=1}^{k(\bar{x}^{\varsigma_1\varsigma_2},\bar{y}^{\varsigma_1\varsigma_2})}\) is monotone decreasing. Hence, }the number of convex corners of  \(\mathbf{P}^{\varsigma_1,\varsigma_2}_{E,E_1}\), namely the length of  \(\bar \Psi_{i}^{\varsigma_1,\varsigma_2}(E,E_1)\), is
\(\\\max\{1,\overline{k}^{\varsigma_1\varsigma_2}(E,E_1)-\underline{k}^{\varsigma_1\varsigma_2}(E,E_1)+1\}\) where:
\begin{align}\label{eq:kbarunder}
\begin{split}
\overline{k}^{\varsigma_1\varsigma_2}(E,E_1)&=\left\{
\begin{array}{l}
 \min\{1\leq k\leq k(\bar{x}^{\varsigma_1\varsigma_2},\bar{y}^{\varsigma_1\varsigma_2}): V_1({x}_k^{\varsigma_1\varsigma_2})\geq E_1 \} \\
  k(\bar{x}^{\varsigma_1\varsigma_2},\bar{y}^{\varsigma_1\varsigma_2})\ \text{ if }  V_1({x}_k^{\varsigma_1\varsigma_2})<E_1 \text{ for } 1\leq k\leq k(\bar{x}^{\varsigma_1\varsigma_2},\bar{y}^{\varsigma_1\varsigma_2})
\end{array}
\right.\\
\underline{k}^{\varsigma_1\varsigma_2}(E,E_1)&=\left\{
\begin{array}{l}
\max\{1\leq k\leq k(\bar{x}^{\varsigma_1\varsigma_2},\bar{y}^{\varsigma_1\varsigma_2}):V_2({y}_k^{\varsigma_1\varsigma_2})\geq E-E_1\}\\
1 \text{ if } V_2({y}_k^{\varsigma_1\varsigma_2})< E-E_1 \text{ for } 1\leq k\leq k(\bar{x}^{\varsigma_1\varsigma_2},\bar{y}^{\varsigma_1\varsigma_2}).
\end{array}
\right.
\end{split}
\end{align}
If $\underline{k}^{\varsigma_1\varsigma_2}(E,E_1)\geq\overline{k}^{\varsigma_1\varsigma_2}(E,E_1)$ then $\mathbf{P}^{\varsigma_1\varsigma_2}_{E,E_1}$ is a rectangle. Moreover,
\begin{gather}\label{eq:rectab}
\begin{split}
&\text{if }V_1({x}_l^{\varsigma_1\varsigma_2})\geq E_1,\ V_2({y}_l^{\varsigma_1\varsigma_2})\geq E-E_1\text{ for some }1\leq l\leq k(\bar{x}^{\varsigma_1\varsigma_2}, \bar{y}^{\varsigma_1\varsigma_2})\\ &\text{ then }\underline{k}^{\varsigma_1\varsigma_2}\geq l\geq \overline{k}^{\varsigma_1\varsigma_2}
\text{ and }
\mathbf{P}^{\varsigma_1\varsigma_2}_{E,E_1}=P\Big( \frac{1}{4}\varsigma_1T_{1}(E_1), \frac{1}{4}\varsigma_2T_{2}(E-E_1)\Big).
\end{split}
\end{gather}
Otherwise, we have $\underline{k}^{\varsigma_1\varsigma_2}(E,E_1)\leq\overline{k}^{\varsigma_1\varsigma_2}(E,E_1)$ and
\[\mathbf{P}^{\varsigma_1\varsigma_2}_{E,E_1}=P(\varsigma_1 \bar\Psi_{1}^{\varsigma_1\varsigma_2}(E_1),\varsigma_2\bar\Psi_{2}^{\varsigma_1\varsigma_2}(E-E_1)),\]
where the vectors \(( \bar\Psi_{1}^{\varsigma_1\varsigma_2}(E_1),\bar\Psi_{2}^{\varsigma_1\varsigma_2}(E_1))\):
\begin{align*}
{\bar\Psi}_{1}^{\varsigma_1\varsigma_2}(E_1)=\{\Psi_{1,k}^{\varsigma_1\varsigma_2}(E_1)_{}\}_{k=\underline{k}^{\varsigma_1\varsigma_2}}^{\bar{k}^{\varsigma_1\varsigma_2}},\quad
\bar\Psi_{2}^{\varsigma_1\varsigma_2}(E-E_1)=\{\Psi_{2,k}^{\varsigma_1\varsigma_2}(E-E_1)_{}\}_{k=\underline{k}^{\varsigma_1\varsigma_2}}^{\bar{k}^{\varsigma_1\varsigma_2}}
\end{align*}
are found from the sequences \((\bar{x}^{\varsigma_1\varsigma_2},\bar{y}^{\varsigma_1\varsigma_2})\) by the \((E_{1},E)\) dependent \(\psi\) transformation of the corner points that are inside \(\mathbf{P}^{\varsigma_1\varsigma_2}_{E,E_{1}}\). Since, for \(k<\bar{k}^{\varsigma_1\varsigma_2}\), the horizontal motion exceeds \({x}_k^{\varsigma_1\varsigma_2}\)  (as \(V_1({x}_k^{\varsigma_1\varsigma_2})< E_{1}\)) we have
\begin{equation}\label{eq:psinumk}
{\Psi}_{1,k}^{\varsigma_1\varsigma_2}(E,E_1)=\psi _{1}({x}_k^{\varsigma_1\varsigma_2},E_{1}),\quad \underline{k}^{\varsigma_1\varsigma_2}\leqslant k<\bar{k}^{\varsigma_1\varsigma_2}.
\end{equation}
The last  value of \( \bar\Psi_{1}^{\varsigma_1\varsigma_2}(E,E_1)\)  depends on whether  \(R^{(E,E_{1})}\) intersect the extremal vertical side of   \(P^{\varsigma_1\varsigma_2}\) (in the \((x,y)\) plane):
\begin{equation}
\Psi_{1,\bar{k}^{\varsigma_1\varsigma_2}}(E,E_{1})=\begin{cases}\psi _{1}({x}_{\bar{k}^{\varsigma_1\varsigma_2}}^{\varsigma_1\varsigma_2},E_{1}) & \text{ if }\ V_1({x}_{\bar{k}^{\varsigma_1\varsigma_2}}^{\varsigma_1\varsigma_2})< E_1  \\
\frac{1}{4}T_1(E_1) & \text{ if }\ V_1({x}_{\bar{k}^{\varsigma_1\varsigma_2}}^{\varsigma_1\varsigma_2})\geq E_1.
\end{cases}
\end{equation}
Similarly, since  for \(k>\underline{k}^{\varsigma_1\varsigma_2} \), the vertical motion exceeds \({y}_k^{\varsigma_1\varsigma_2}\)   (since  \(E-E_{1}>V_2({y}_k^{\varsigma_1\varsigma_2})\)) we have
\begin{equation}
\quad{\Psi}_{2,k}^{\varsigma_1\varsigma_2}(E,E_1)=\psi _{2}({y}_k^{\varsigma_1\varsigma_2},E-E_{1}),\quad \underline{k}^{\varsigma_1\varsigma_2}< k\leqslant\bar{k}^{\varsigma_1\varsigma_2}.
\end{equation}
The first  value of \( \bar\Psi_{2}^{\varsigma_1\varsigma_2}(E,E_1)\)  depends on whether  \(R^{(E,E_{1})}\) intersect the extremal  horizontal side of   \(P^{\varsigma_1\varsigma_2}\) (in the \((x,y)\) plane):
\begin{align}
\label{eq:psinumlast}\qquad\Psi_{2,\underline{k}^{\varsigma_1\varsigma_2}}(E,E_1)&=\begin{cases}\psi _{2}({y}_{\underline{k}^{\varsigma_1\varsigma_2}}^{\varsigma_1\varsigma_2},E-E_{1}) & \text{ if }\ V_2({y}_{\underline{k}^{\varsigma_1\varsigma_2}}^{\varsigma_1\varsigma_2})< E-E_1  \\
\frac{1}{4}T_2(E-E_1) & \text{ if }\ V_2({y}_{\underline{k}^{\varsigma_1\varsigma_2}}^{\varsigma_1\varsigma_2})\geq E-E_1.
\end{cases}
\end{align}
Summarizing, the above computations show that the  topological data of the polygon  \(\mathbf{P}_{E,E_{1}}\) is given by  \(\{\max\{1,\overline{k}^{\varsigma_1\varsigma_2}(E,E_1)-\underline{k}^{\varsigma_1\varsigma_2}(E,E_1)+1\}\}_{\varsigma_1,\varsigma_2\in\{\pm\}}\)  and the numerical data by  \( \{(\bar\Psi_{1}^{\varsigma_1\varsigma_2}(E,E_1),\bar\Psi_{2}^{\varsigma_1\varsigma_2}(E,E_1))\}_{\varsigma_1,\varsigma_2\in\{\pm\}}\).

\subsection{Regions of fixed topological data}
Next,  we show that the topological data of  the polygons corresponding to iso-energy level sets is fixed on a finite number of intervals of   $E_1$ values, and this partition depends piecewise smoothly on \(E\) (so the topological data is fixed in certain parallelograms of the IEMBD, see Figure~\ref{fig:embd}).

Let \(X,Y\) denote the collection of widths and heights of the steps in all quadrants:
\begin{align*}
X:=\{x_k^{\varsigma_1\varsigma_2}:\varsigma_1, \varsigma_2\in\{\pm\},1\leq k\leq k(\bar{x}^{\varsigma_1\varsigma_2},\bar{y}^{\varsigma_1\varsigma_2})\}\subset\R_{>0}\\
Y:=\{y_k^{\varsigma_1\varsigma_2}:\varsigma_1, \varsigma_2\in\{\pm\},1\leq k\leq k(\bar{x}^{\varsigma_1\varsigma_2},\bar{y}^{\varsigma_1\varsigma_2})\}\subset\R_{>0}.
\end{align*}
For any $E>0$ let us consider the partition $\mathcal{J}_E$ (into open intervals) of the interval $[0,E]$ determined by the numbers
\[V_1(x), E-V_2(y)\text{ for all }x\in X\text{ and }y\in Y.\]
Then for every $I\in \mathcal{J}_E$, by Eq.~\eqref{eq:kbarunder} the numbers  $\overline{k}^{\varsigma_1\varsigma_2}_I=\overline{k}^{\varsigma_1\varsigma_2}(E,E_1)$ and $\underline{k}^{\varsigma_1\varsigma_2}_I=\underline{k}^{\varsigma_1\varsigma_2}(E,E_1)$
do not depend on $E_1\in I$. Therefore, the numerical data  $E_1\in I\mapsto \mathbf P_{E,E_1}\in\mathscr R$ is a smooth (as shown in Section~\ref{sec:propertiesa}, analytic if $V_1,V_2\in Deck$) curve of polygons  in $\mathscr{R}$.

\begin{remark}\label{rem:XYI}
Fix $I=(E_{min},E_{max})\in \mathcal{J}_E$. Then the sets
\begin{equation}\label{def:XYI}
X_I:=\{x\in X:V_1(x)<E_1\}
\text{ and }Y_I:=\{y\in Y:V_2(y)<E-E_1\}
\end{equation}
do not depend on the choice of $E_1\in I$.

 Summarizing, in view of \eqref{eq:psinumk}-\eqref{eq:psinumlast}, for any $E_1\in I$ and $\varsigma_1,\varsigma_2\in\{\pm\}$
\begin{itemize}
\item every staircase lengths of $\mathbf{P}^{\varsigma_1\varsigma_2}_{E,E_1}$ is of the form $\psi_{1}(x,E_1)$ for some $x\in X_I$;
\item every staircase heights of $\mathbf{P}^{\varsigma_1\varsigma_2}_{E,E_1}$ is of the form  $\psi_{2}(y,E-E_1)$ for some $y\in Y_I$;
\item if $E_{min}>V_1(x_{{k}(\bar{x}^{\varsigma_1\varsigma_2},\bar{y}^{\varsigma_1\varsigma_2})}^{\varsigma_1\varsigma_2})$   then the width of $\mathbf{P}^{\varsigma_1\varsigma_2}_{E,E_1}$ is of the form  $\psi_{1}(x,E_1)$  for some $x\in X_I$, otherwise, the width of $\mathbf{P}^{\varsigma_1\varsigma_2}_{E,E_1}$ is of the form $\tfrac{1}{4}T_1(E_1)$;
\item if $E_{max}<E-V_2(y_{1}^{\varsigma_1\varsigma_2})$    then the height of $\mathbf{P}^{\varsigma_1\varsigma_2}_{E,E_1}$ is of the form  $\psi_{2}(y,E-E_1)$  for some $y\in Y_I$, otherwise, the height of $\mathbf{P}^{\varsigma_1\varsigma_2}_{E,E_1}$ is of the form $\tfrac{1}{4}T_2(E-E_2) $.
\end{itemize}
\end{remark}
 The IEMBD figures (Figures \ref{fig:embdsym} and~\ref{fig:embd}) provide a graphical representation of the above summary: the intersection of the  wedges boundaries with a vertical line provide the partition to the intervals \(\mathcal{J}_E\),  each  blue wedge corresponds to a region in which another step in the staircase (a concave corner) is included, and {the pink (respectively light green) wedges correspond to regions in which the corresponding staircase polygons widths (respectively heights) are of the form $\psi_1(x,E_1)$ (respectively $\psi_2(y,E-E_1)$).}
Notice that the dependence of the above partition of \(\mathcal{J}_E\) is piecewise smooth in \(E\):  it changes exactly at the singular \(E\) values \(E^{sin}=\{E:E=V_{1}(x)+V_2(y),x\in X,y\in Y\}\), namely, at the energy values  at which the wedges in the IEMBD figures emanate and/or start to intersect each other.

\section{Properties of numerical data in a given topological region}\label{sec:propertiesa}
 In this section we present some basic properties of functions $\psi_{1}(x,E_{1}),T_1(E_{1})$, defined by Eq.~\eqref{def:psires}, \eqref{def:a} as a function of \(E_{1}\) (so $T_1:(0,  +\infty)\to\R_{>0}$ and $\psi_{1}(x,\cdot):[V_{1}(x),  +\infty)\to\R_{>0}$). We show that when $V_{1}:\R\to\R_{\geq 0}$ is an even $C^2$-potential satisfying \eqref{eq:i} the function  $\psi_{1}(x,E_{1})$ is analytic and that when $V_1$ is a $Deck$-potential the function \(T_1(E_{1})\) is analytic (the same properties apply to  $\psi_{2}(y,E_{2}),T_2(E_{2})$ with corresponding assumptions on \(V_{2}\)).

\begin{proposition}\label{prop:a}
Suppose that  $V_{1}:\R\to\R_{\geq 0}$ is a $Deck$-potential.
Then the map $T_{1}:(0,  +\infty)\to\R_{>0}$ given by \eqref{def:a} is analytic and
\begin{equation}\label{eq:det}
\frac{1}{4}T_{1}'(E_1)=\frac{1}{E_1\sqrt{2}}\int_0^{x^{max}(E_1)}\frac{1}{\sqrt{E_1-V_{1}(x)}}\left(\frac{1}{2}-\frac{V_{1}''(x)V_{1}(x)}{(V_{1}'(x))^2}\right)dx.
\end{equation}
\end{proposition}
\begin{proof}
Using integration by substitution twice we have
\begin{align}\label{ibs}
\begin{aligned}
\frac{\sqrt{2}}{4}T_{1}(E_1)&=\int_0^{x^{max}(E_1)}\frac{1}{\sqrt{E_1-V_{1}(x)}}dx=|_{u=V_{1}(x), x=x^{max}(u)}
\int_0^{E_1}\frac{(x^{max})'(u)}{\sqrt{E_1-u}}du \\
&=|_{ s=u/E_1}
\sqrt{E_1}\int_0^{1}\frac{(x^{max})'(E_1 s)}{\sqrt{1-s}} ds.
\end{aligned}
\end{align}
Let us consider an auxiliary map  $A:(0,  +\infty)\to\R_{>0}$ defined by
\begin{equation}\label{def:A}
A(E_1)=\sqrt{\frac{1}{8E_1}}T_{1}(E_1)=\int_0^{1}\frac{(x^{max})'(E_1 s)}{\sqrt{1-s}} ds
\end{equation}
for $E_1>0$.
We will show that $A$ is analytic, which obviously implies the analyticity of $T_{1}(E_1)$ for  $E_1>0$.
 To this aim, we first establish some properties of the function \(x^{max}\) and its holomorphic extension.

Suppose that $V_{1}:\R\to\R_{\geq 0}$ is a $Deck$-potential. In view of
\eqref{eq:ii}, $V_{1}:(0,+\infty)\to(0,  +\infty)$ has a holomorphic extension $V_{1}:U\to\C$ on an open neighborhood $U\subset\C$ of $(0,+\infty)$ such that $V_{1}'(z)\neq 0$ for every $z\in U$. Then $V_{1}$ is locally invertible and its inverse functions are holomorphic (i.e.\ $V_{1}$ is locally biholomorphic).
More subtle assumption on the domain of biholomorphicity is formulated in \eqref{eq:iii}. By \eqref{eq:iii},
for every $E_{10}>0$ there exist $0<r<E_{10}$ and a bounded open set $U_{E_{10}}\subset U$ such that $V_{1}:U_{E_{10}}\to V_{1}(U_{E_{10}})$
is biholomorphic with $C(E_{10},r)\subset V_{1}(U_{E_{10}})$. Then $z^{max}=V_{1}^{-1}:V_{1}(U_{E_{10}})\to U_{E_{10}}$ is a holomorphic extension of
$x^{max}:(0,E_{10}+r)\to\R$.

Fix any $E_{10}>0$. By the above, there exists $0<r<E_{10}$ such that the inverse map $z^{max}$ is analytic on $E_{1}\in C(E_{10},r)$ (so \(E_1\) is a complex variable in the proof and \(E_{10}\) is a real positive number).
As
\begin{equation}\label{eq:WV}
\frac{(z^{max})''(E_{1})}{(z^{max})'(E_{1})}=-\frac{V_{1}''(z^{max}(E_{1}))}{(V_{1}'(z^{max}(E_{1})))^2}\;\text{ for every }\;E_{1}\in V_{1}(U_{E_{10}}),
\end{equation}
by the assumption \eqref{eq:iv}, it follows that
\begin{equation}\label{ineq:CR}
\left|\frac{(z^{max})''(E_{1})E_{1}}{(z^{max})'(E_{1})}\right|\leq C_{E_{10}}\;\text{ for every }\;E_{1}\in C(E_{10},r).
\end{equation}
Assume that $E_1\in B(E_{10},r)$ and $s\in(0,1]$. Then
\begin{align*}
\left|\log\frac{|(z^{max})'(E_1 s)|}{|(z^{max})'(E_{10} s)|}\right|&=\big|\log |(z^{max})'(E_1 s)|-\log{|(z^{max})'(E_{10} s)|}\big|\\
&=\left|\int_0^1\frac{d}{dt}\log{|(z^{max})'((E_{10}+(E_1-E_{10})t) s)|}dt\right|\\
&\leq\int_0^1\left|\frac{\frac{d}{dt}|(z^{max})'((E_{10}+(E_1-E_{10})t) s)|}{|(z^{max})'((E_{10}+(E_1-E_{10})t) s)|}\right|dt\\
&\leq\int_0^1\frac{|(z^{max})''((E_{10}+(E_1-E_{10})t) s)||E_1-E_{10}||s|}{|(z^{max})'((E_{10}+(E_1-E_{10})t) s)|}dt.
\end{align*}
In view of \eqref{ineq:CR}, it follows that
\begin{align*}
\left|\log\frac{|(z^{max})'(E_1 s)|}{|(z^{max})'(E_{10} s)|}\right|&
\leq C_{E_{10}}\int_0^1\frac{|E_1-E_{10}||s|}{|(E_{10}+(E_1-E_{10})t) s|}dt\leq C_{E_{10}}\int_0^1\frac{|E_1-E_{10}|}{E_{10}-|E_1-E_{10}|t}dt\\
&=C_{E_{10}}\log\frac{E_{10}}{E_{10}-|E_1-E_{10}|}\leq C_{E_{10}}\log\frac{E_{10}}{E_{10}-r}.
\end{align*}
Hence
\begin{equation}\label{ineq:w'}
|(z^{max})'(E_1 s)|\leq \left(\frac{E_{10}}{E_{10}-r}\right)^{C_{E_{10}}} (z^{max})'(E_{10} s)
\end{equation}
for all $E_1\in B(E_{10},r)$ and $s\in (0,1]$.

We now show that the holomorphic extension  $A:B(E_{10},r)\to\C$ is given by \eqref{def:A} for complex \(E_{1}\).
By \eqref{ineq:w'},
\[\left|\frac{(z^{max})'(E_1 s)}{\sqrt{1-s}}\right|\leq C(E_{10},r)\frac{(z^{max})'(E_{10} s)}{\sqrt{1-s}}\]
for all $E_1\in B(E_{10},r)$ and $s\in (0,1]$. Therefore, \(|A(E_{1})|<C(E_{10},r)|A(E_{{10}})|\), so, since \(A(E_{{10}})\) is finite for positive \(E_{10}\),  $A:B(E_{10},r)\to\C$ is well defined.

We now show that $A:B(E_{10},r)\to\C$  is holomorphic. For every parameter $s\in(0,1]$ let us consider the map $\Phi_s:B(E_{10},r)\to\C$
given by
\[\Phi_s(E_1):=\frac{(z^{max})'(E_1 s)}{\sqrt{1-s}},\]
so \(A(E_{1})=\int_0^1\Phi_s(E_{1})ds\). The map $\Phi_s$ is holomorphic with
\[\Phi'_s(E_1):=\frac{(z^{max})''(E_1 s)s}{\sqrt{1-s}}.\]
In view of \eqref{ineq:CR} and \eqref{ineq:w'}, we obtain
\begin{align}\label{ineq:|u|}
\begin{aligned}
|\Phi'_s(E_1)|&=\left|\frac{(z^{max})''(E_1 s)s}{\sqrt{1-s}}\right|\leq \frac{C_{E_{10}}}{|E_1|}\left|\frac{(z^{max})'(E_1 s)}{\sqrt{1-s}}\right|\\
&\leq \frac{C_{E_{10}}E_{10}^{C_{E_{10}}}}{(E_{10}-r)^{C_{E_{10}}+1}}\frac{(z^{max})'(E_{10} s)}{\sqrt{1-s}}=C\Phi_s(E_{10})
\end{aligned}
\end{align}
for all $E_1\in B(E_{10},r)$ and $s\in(0,1]$. Since \(A(E_{10})\) is finite,  $A(E_1)=\int_0^1\Phi_s(E_1)\,ds$ is holomorphic and
\[A'(E_1)=\int_0^1\Phi'_s(E_1)\,ds=\int_0^1\frac{(z^{max})''(E_1 s)s}{\sqrt{1-s}}\,ds.\]
Using integration by substitution the same as in \eqref{ibs} (used in reverse order), we obtain, for real \(E_{1}>0\),
\begin{align*}
A'(E_1)&=\int_0^1\frac{(x^{max})''(E_1 s)s}{\sqrt{1-s}}\,ds=\frac{1}{E_1\sqrt{E_1}}\int_0^{E_1}\frac{(x^{max})''(u)u}{\sqrt{E_1-u}}\,du\\
&=\frac{1}{E_1\sqrt{E_1}}\int_0^{x^{max}(E_1)}\frac{(x^{max})''(V_{1}(x))V_{1}(x)V_{1}'(x)}{\sqrt{E_1-V_{1}(x)}}\,dx.
\end{align*}
Therefore, by \eqref{eq:WV}, we have
\[A'(E_1)=-\frac{1}{E_1\sqrt{E_1}}\int_0^{x^{max}(E_1)}\frac{1}{\sqrt{E_1-V_{1}(x)}}\frac{V_{1}''(x)V_{1}(x)}{(V_{1}'(x))^2}\,dx.\]
Finally, it follows that
\begin{align*}
\frac{1}{4}T_{1}'(E_1)&=\frac{d}{dE_1}\Big(\sqrt{\frac{E_1}{2}}A(E_1)\Big)=\frac{1}{2\sqrt{2}\sqrt{E_1}}A(E_1)+\sqrt{\frac{E_1}{2}}A'(E_1)\\&=
\frac{1}{E_1\sqrt{2}}\int_0^{x^{max}(E_1)}\frac{1}{\sqrt{E_1-V_{1}(x)}}\left(\frac{1}{2}-\frac{V_{1}''(x)V_{1}(x)}{(V_{1}'(x))^2}\right)dx.
\end{align*}
\end{proof}
\begin{proposition}\label{prop:ay}
Suppose that  $V_{1}:\R\to\R_{\geq 0}$ is a  continuous potential such that $V_{1}:\R_{\geq 0}\to\R_{\geq 0}$ is strictly increasing and $V_{1}(0)=0$.
Then for every $x>0$ the map $\psi_{1}(x,\cdot):(V_{1}(x),  +\infty)\to\R_{>0}$ given by \eqref{def:psires} is analytic and, for every $E_1>V_{1}(x)$ and $n\geq 1$
\begin{equation}\label{eq:dera1}
\frac{d^{n}}{dE_{1}^n}\psi_{1}(x,E_1)=\frac{(-1)^n(2n-1)!!}{2^{\frac{2n+1}{2}}}\int_0^{x}\frac{1}{(E_1-V_{1}(y))^{\frac{2n+1}{2}}}dy
\end{equation}
Additionally, if $V_{1}\in Deck$,
for every $E_1>V_{1}(x)$ 
\begin{align}\label{eq:dera2}
\begin{split}
\frac{d}{dE_{1}}\psi_{1}(x,E_1)=&\frac{1}{E_1\sqrt{2}}\int_0^{x}\!\frac{1}{\sqrt{E_1\!-\!V_{1}(y)}}\left(\frac{1}{2}\!-\!\frac{V_{1}''(y)V_{1}(y)}{(V_{1}'(y))^2}\right)dy \\
&-  \frac{1}{E_1\sqrt{E_1\!-\!V_{1}(x)}\sqrt{2}}\frac{V_{1}(x)}{V'_{1}(x)}
\end{split}
\end{align}
and
\begin{equation}\label{eq:dera3}
\lim_{E_1\searrow V_{1}(x)}\frac{d}{dE_{1}}\psi_{1}(x,E_1)=-\infty.
\end{equation}
\end{proposition}

\begin{proof}
Take any $E_{10}>V_{1}(x)$ and let $0<r<E_{10}-V_{1}(x)$. Then for every $y\in[0,x]$ the map $\Phi_y:B(E_{10},r)\to\C$ given by $\Phi_y(E_1)=\tfrac{1}{\sqrt{E_1-V_{1}(y)}}$
is holomorphic (taking again \(E_{1}\) to be complex) with
\[\Phi'_y(E_1)=-\frac{1}{2}\frac{1}{(E_1-V_{1}(y))^{3/2}}\;\text{ for every }\;E_1\in B(E_{10},r).\]
Since for all $y\in[0,x]$ and $E_1\in B(E_{10},r)$ we have
\[|\Phi'_y(E_1)|=\frac{1}{2}\frac{1}{|E_1-V_{1}(y)|^{3/2}}\leq\frac{1}{2(E_{10}-V_{1}(x)-r)^{3/2}},\]
it follows that $\psi_{1}(x,E_1):B(E_{10},r)\to\C$ is holomorphic and
\[\frac{d}{dE_{1}}\psi_{1}(x,E_1)=\int_0^{x}\frac{\Phi'_y(E_1)}{\sqrt{2}}\,dy=-\frac{1}{2\sqrt{2}}\int_0^{x}\frac{1}{(E_1-V_{1}(y))^{3/2}}dy\]
for every $E_1\in B(E_{10},r)$. This gives the analyticity of $\psi_{1}(x,E_1)$ and \eqref{eq:dera1} for $n=1$. Repeating the same reasoning for higher order derivatives
we obtain
\[\frac{d^{n}}{dE_{1}^n}\psi_{1}(x,E_1)=\int_0^{x}\frac{\Phi^{(n)}_y(E_1)}{\sqrt{2}}\,dy=\frac{(-1)^n(2n-1)!!}{2^{\frac{2n+1}{2}}}\int_0^{y_0}\frac{1}{(E_1-V_{1}(y))^{\frac{2n+1}{2}}}dy.\]

In order to show \eqref{eq:dera2} we first notice that  using integration by substitution  as in \eqref{ibs}, for every $E_1>V_{1}(x)$ we obtain
\[
\sqrt{2}\psi_{1}(x,E_1)=\int_0^{x}\frac{1}{\sqrt{E_1-V_{1}(y)}}dy
=\int_0^{V_{1}(x)}\frac{(x^{max})'(u)}{\sqrt{E_1-u}}du
=\sqrt{E_1}\int_0^{V_{1}(x)/E_1}\frac{(x^{max})'(E_1 s)}{\sqrt{1-s}} ds.
\]
In view of \eqref{ineq:|u|} and  $\int_0^t\frac{(x^{max})'(E_1 s)}{\sqrt{1-s}} ds<\int_0^1\frac{(x^{max})'(E_1 s)}{\sqrt{1-s}} ds=A(E_1)<+\infty$ for every $0\leq t<1$ and $E_1>0$, using arguments similar to those in the proof of Proposition~\ref{prop:a}, for every $E_1>V_{1}(x)$ we obtain
\begin{align*}
\sqrt{2}\frac{d}{dE_{1}}\psi_{1}(x,E_1)&=\frac{\psi_{1}(x,E_1)}{E_1\sqrt{2}}-\sqrt{E_1}\frac{(x^{max})'(V_{1}(x))}{\sqrt{1-V_{1}(x)/E_1}}\frac{V_{1}(x)}{E_1^2}+\sqrt{E_1}\int_0^{V_{1}(x)/E_1}\frac{(x^{max})''(E_1 s)s}{\sqrt{1-s}} ds\\
&=\frac{1}{E_1}\int_0^{x}\!\frac{1}{\sqrt{E_1\!-\!V_{1}(y)}}\left(\frac{1}{2}\!-\!\frac{V_{1}''(y)V_{1}(y)}{(V_{1}'(y))^2}\right)dy
\!-\!\frac{1}{E_1\sqrt{E_1\!-\!V_{1}(x)}}\frac{V_{1}(x)}{V_{1}'(x)}.
\end{align*}

In view of \eqref{eq:iv}, for every $x>0$ we have
\[\left|\frac{V_{1}''(y)V_{1}(y)}{(V'_{1}(y))^2}\right|\leq C_{x}\text{ for every }y\in(0,x].\]
Therefore for every $E_1>V_1(x)$ we have
\begin{align*}
\left|\int_0^{x}\!\frac{1}{\sqrt{E_1\!-\!V_{1}(y)}}\left(\frac{1}{2}\!-\!\frac{V_{1}''(y)V_{1}(y)}{(V_{1}'(y))^2}\right)dy\right|&\leq
\left(\frac{1}{2}+C_{x}\right)\int_0^{x}\!\frac{1}{\sqrt{V_1(x)\!-\!V_{1}(y)}}\,dy\\
&=\frac{\sqrt{2}}{4}\left(\frac{1}{2}+C_{x}\right)T_{1}(V_1(x))<+\infty.
\end{align*}
As
\[\lim_{E_1\searrow V_1(x)}\frac{1}{E_1\sqrt{E_1-V_1(x)}}\frac{V_1(x)}{V_{1}'(x)}=+\infty,\]
this gives \eqref{eq:dera3}.
\end{proof}

 Suppose that $V_1,V_2:\R\to\R_{\geq 0}$ are even $C^2$-potentials satisfying \eqref{eq:i}.
Let $E$, $x_0$, $y_0$ be positive numbers such that $V_1(x_0)+V_2(y_0)<E$. In view of Proposition~\ref{prop:ay},
$\psi_1(x_0,\cdot):(V_1(x_0),+\infty)\to\R_{>0}$ and $\psi_2(y_0,E-\cdot):[0,E-V_2(y_0))\to\R_{>0}$ are analytic. If additionally $V_1,V_2\in Deck$ then
\begin{equation}\label{eq:-infty}
\lim_{E_1\searrow V_1(x_0)}\frac{d}{dE_1}\psi_1(x_0,E_1)=-\infty,\quad\lim_{E_1\nearrow E-V_2(y_0)}\frac{d}{dE_1}\psi_2(y_0,E-E_1)=+\infty.
\end{equation}
Moreover, by Proposition~\ref{prop:a}, $T_1,T_2(E-\cdot):(0,E)\to\R_{> 0}$ are also analytic.

\begin{proposition}\label{prop:indep2}
Let  $V_1,V_2:\R\to\R_{\geq 0}$ be two $Deck$-potentials.
Fix an energy level $E>0$.
Assume that  $0<x_1<\ldots<x_N$ and $0<y_1<\ldots<y_K$ are such that $V_1(x_N)+V_2(y_K)<E$. Then,
for any sequence $(\gamma_j)_{j=0}^{N+K+1}$ with  \(\sum^{N+K}_{j=1}|\gamma_j|\neq0\), for all but countably many $E_1\in [V_1(x_N),E-V_2(y_K)]$, we have:  \begin{equation}\label{neq:ga1}
\gamma_0\,T_1(E_1)+\sum_{j=1}^N\gamma_j\,\psi_1({x_j},E_1)+\sum_{j=1}^K\gamma_{N+j}\,\psi_2({y_j},E-E_1)+\gamma_{N+K+1}\,T_2(E-E_1)\neq 0.
\end{equation}

\end{proposition}
\begin{proof}
Suppose, contrary to our claim, that \eqref{neq:ga1} does not hold for uncountably many $E_1 \in (V_1(x_N),E-V_2(y_K))$. Since $T_1$, $T_2(E-\cdot)$, $\psi_1({x_j},\cdot)$ for $1\leq j\leq N$ and $\psi_2({y_j},E-\cdot)$, $1\leq j\leq K$ are analytic on $(V_1(x_N),E-V_2(y_K))$,
we have
\begin{equation}\label{eq:gamma}
\gamma_0\,T_1(E_1)+\sum_{j=1}^N\gamma_j\,\psi_1({x_j},E_1)+\sum_{j=1}^K\gamma_{N+j}\,\psi_2({y_j},E-E_1)+\gamma_{N+K+1}\,T_2(E-E_1)= 0
\end{equation}
for all $E_1\in (V_1(x_N),E-V_2(y_K))$.
Without loss of generality we can assume that $\gamma_N$ or $\gamma_{N+K}$ are non-zero.
To simplify the writing, since \(T_{1}\) and \(T_{2}\) are also analytic functions,   we denote $\psi_1({x_0},\cdot):=T_1(\cdot)$ and $\psi_2({y_{K+1}},\cdot):=T_2(\cdot)$.

Suppose that $\gamma_N\neq 0$.
In view of Propositions~\ref{prop:a}~and~\ref{prop:ay}, the maps   $\psi_1({x_j},\cdot)$ for $0\leq j\leq N-1$ and $\psi_2({y_j},\cdot)$ for $1\leq j\leq K+1$ are analytic on  $(V_1(x_{N-1}),E-V_2(y_K))$,
in particular at $V_1(x_N)$.
As $\gamma_N\neq 0$, if  \eqref{eq:gamma} holds, the limit
\begin{align*}
\lim_{E_1\searrow V_1(x_N)}\frac{d}{dE_1}\psi_1({x_N},E_1)&=-\sum_{j=0}^{N-1}\frac{\gamma_j}{\gamma_N}\frac{d}{dE_1}\psi_1({x_j},E_1)|_{E_1=V_{1}(x_N)}\\
&\quad
-\sum_{j=1}^{K+1}\frac{\gamma_{N+j}}{\gamma_N}\frac{d}{dE_1}\psi_2({y_j},E-E_1)|_{E_1=V_{1}(x_N)}
\end{align*}
is finite.
On the other hand, by \eqref{eq:-infty}, the limit of \(\lim_{E_1\searrow V_1(x_N)}\frac{d}{dE_1}\psi_1({x_N},E_1)\) is $-\infty$. This contradiction completes the proof when $\gamma_N\neq 0$.


If $\gamma_{N+K}\neq 0$, then a contradiction follows from similar arguments based on studying the left-side limit of $\tfrac{d}{dE_1}\psi_2({y_K},E-E_1)$ at $E- V_2(y_K)$.
This completes the proof.
\end{proof}

\begin{lemma}\label{lem:negposder}
Suppose that $V_1,V_2:\R\to\R_{\geq 0}$ are are even $C^2$-potentials satisfying \eqref{eq:i}.
Then for any $x_0,y_0>0$ we have
\begin{align}
&\frac{d}{dE_1}\psi_1({x_0},E_1)<0\;\text{ for all }\;  E_1>V_1(x_0)\label{ineq:dera2}\\
&\frac{d}{dE_1}\psi_2({x_0},E-E_1)>0\;\text{ for all }\; E_1\in(0,E-V_2(y_0)).\label{ineq:dera4}
\end{align}

Assume additionally that $V_1, V_2\in Deck$ and satisfy \eqref{eq:vi}.
Then for $i=1,2$ we have
\begin{equation}\label{neq:>12}
\frac{V_i(x)V_i''(x)}{(V_i'(x))^2}> \frac{1}{2}\;\text{ for all but countably many }\; x>0
\end{equation}
and
\begin{align}
&\frac{d}{dE_1}T_1(E_1)<0\;\text{ for every }\; E_1>0\label{ineq:dera1}\\
&\frac{d}{dE_1}T_2(E-E_1)>0\;\text{ for every }\; E_1\in(0,E)\label{ineq:dera3}.
\end{align}
\end{lemma}

\begin{proof}
The inequalities \eqref{ineq:dera2} and \eqref{ineq:dera4} follow immediately from \eqref{eq:dera1}.
The inequality  \eqref{neq:>12} follows immediately from the analyticity of $V_i$ and \eqref{eq:vi}.
Finally \eqref{ineq:dera1} and \eqref{ineq:dera3} follow immediately from  \eqref{eq:det} and \eqref{neq:>12}.
\end{proof}

\begin{lemma}\label{lem:square}
Let $V_1:\R\to\R_{\geq 0}$ be a $Deck$-potential. Then $V_1$ satisfies \eqref{eq:v} if and only if $V_1^{1/2}$ is convex.
The following three conditions are equivalent:
\begin{itemize}
\item[(a)] $V_1$ satisfies \eqref{eq:v} and does not meet \eqref{eq:vi};
\item[(b)] $\displaystyle\frac{V_1(x)V_1''(x)}{(V_1'(x))^2}= \frac{1}{2}$ for all $x>0$;
\item[(c)] $V_1(x)=\frac{1}{2}\omega _{1}^{2}x^2$ for some $\omega_1>0$.
\end{itemize}
Moreover, if $V_{1}(x)=\frac{1}{2}\omega _{1}^{2}x^2$ then
\[T_1(E_1)=\frac{2\pi}{\omega _{1}}\quad\text{and}\quad \psi_1({x_0},E_1)=\frac{1}{\omega _{1}}\arcsin\frac{\omega _{1}x_0}{\sqrt{2E_1}}=\frac{1}{\omega _{1}}\arcsin\sqrt{\frac{V_1(x_0)}{E_1}}.\]
\end{lemma}

\begin{proof}
The first part of the lemma follows directly from the formula (derived by taking derivatives of \((V_1^{1/m}(x))^{m}\))
:\begin{equation}
\frac{V_1(x)V_1''(x)}{(V_1'(x))^2}=\frac{m-1}{m}+\frac{V_1^{1/m}(x)(V_1^{1/m})''(x)}{m((V_1^{1/m})'(x))^2}
\end{equation}
that holds for all $m\in\N$ and  $x>0$. Since $V_1(x)$, $V_1'(x)$, $V_1^{1/m}(x)$ and $(V_1^{1/m})'(x)$ are
positive for all $x>0$, we have
\[\frac{V_1(x)V_1''(x)}{(V_1'(x))^2}\geq \frac{m-1}{m}\;\text{ for all }\; x>0\]
if and only if the map $V_1^{1/m}$ has non-negative second derivative on $\R\setminus\{0\}$.

\medskip

The implications (c)$\Rightarrow$(b)$\Rightarrow$(a)$\Rightarrow$(b) are obvious.
To prove (b)$\Rightarrow$(c) suppose that $V_1V_1''=\frac{1}{2}V_1^2$. Then
\[\frac{d}{dx}\log (V_1'(x))=\frac{V_1''(x)}{V_1'(x)}=\frac{1}{2}\frac{V_1'(x)}{V_1(x)}=\frac{d}{dx}\log \sqrt{V_1(x)}.\]
Hence, for some $C>0$ we have
$V_1'(x)=C\sqrt{V_1(x)}$.
Thus
\[\frac{d}{dx}\sqrt{V_1(x)}=\frac{1}{2}\frac{V_1'(x)}{\sqrt{V_1(x)}}=\frac{C}{2}.\]
As $V_1(0)=0$, it follows that $V_1(x)=(\frac{C}{2})^2x^2=\frac{1}{2}\omega _{1}^{2}x^2$.
The form of the maps $T_1(E_1)$ and $\psi_1({x_0},E_1)$ when $V_1(x)= \frac{1}{2}\omega _{1}^{2}x^2$ follows from direct computations.
\end{proof}

The following lemma is the counterpart of Proposition~\ref{prop:indep2} in the case where all $\gamma_j$ for $1\leq j\leq N+K$ are zero.

\begin{lemma}\label{lem:indep3}
Assume $V_1,V_2:\R\to\R_{\geq 0}$ are $Deck$-potentials satisfying \eqref{eq:v} and $\gamma_0,\gamma_{N+K+1}$ are real numbers with $|\gamma_0|+|\gamma_{N+K+1}|\neq 0$.

If at least one potential $V_1$ or $V_2$ satisfies \eqref{eq:vi} (i.e.\ is not quadratic) then
\begin{equation}\label{eq:gamex}
\gamma_0\,T_1(E_1)+\gamma_{N+K+1}\,T_2(E-E_1)\neq 0 \text{  for all but countably many } E_1\in(0,E).
\end{equation}

If both potentials $V_1$ and $V_2$ are quadratic and
\begin{equation}\label{eq:gamex0}
\gamma_0\,T_1(E_1)+\gamma_{N+K+1}\,T_2(E-E_1)= 0\ \text{  for some }\ E_1\in(0,E),
\end{equation}
then $\gamma_{0}/\gamma_{N+K+1}=-\Omega$.\footnote{If $V_1$ and $V_2$ are quadratic then $\Omega(E,E_1)$ does not depend on $E,E_1$.}
\end{lemma}

\begin{proof}
Suppose that $V_1$ is not quadratic and \eqref{eq:gamex} does not hold.   Since $T_1(\cdot)$ and $T_2(E-\cdot)$ are analytic and take only positive values, we have:
\[\gamma_0\,T_1(E_1)+\gamma_{N+K+1}\,T_2(E-E_1)= 0 \text{  for all  } E_1\in(0,E),\]
both \(\gamma_{0}\) and \(\gamma_{N+K+1}\) do not vanish, and  $\gamma:= -\gamma_{N+K+1}/\gamma_0>0$. It follows that
\[\frac{d}{dE_1}T_1(E_1)=\gamma\frac{d}{dE_1}T_2(E-E_1) \text{  for all  } E_1\in(0,E).\]
On the other hand, by Lemma~\ref{lem:negposder}, $\frac{d}{dE_1}T_1(E_1)>0$ and $\frac{d}{dE_1}T_2(E-E_1)\leq 0$ for all $E_1\in E$. This gives a contradiction.

Suppose that $V_i(x)=\tfrac{1}{2}\omega_i^2x^2$ for $i=1,2$ and \eqref{eq:gamex0} holds, since \(T_{i}=\frac{2\pi }{\omega_{i}}\), Eq.~\eqref{eq:gamex0} implies $\tfrac{\gamma_0}{\omega_1}+\tfrac{\gamma_{N+K+1}}{\omega_2}=0$, so
\[\frac{\gamma_0}{\gamma_{N+K+1}}=-\frac{\omega_1}{\omega_2}=\Omega.\]
\end{proof}

Recall that for $n$ real-valued $C^{n-1}$-functions $f_1, \ldots , f_n$ on an interval $I$ their Wronskian at $x\in I$ is defined by
\[\mathscr{W}(f_1,\ldots,f_n)(x)=\det\big[f_i^{(j-1)}(x)\big]_{1\leq i,j\leq n}.\]
We will also deal with the bracket
\[[f,g](x)=\mathscr{W}(g,f)=f'(x)g(x)-f(x)g'(x)\]
for $C^1$-maps $f,g:I\to\R$.

\begin{proposition}\label{prop:indep1}
Let  $V_1,V_2:\R\to\R_{\geq 0}$ be even $C^2$-potentials satisfying \eqref{eq:i}.
Assume that  $0<x_1<\ldots<x_N$ and $0<y_1<\ldots<y_K$ are such that $V_1(x_N)+V_2(y_K)<E$.
Then for all $E_1\in (V_1(x_N), E-V_2(y_K))$ we have
\begin{gather}
\label{neq:ga}
\mathscr{W}\big(\psi_1({x_1},E_1),\ldots,\psi_1({x_N},E_1),\psi_2({y_1},E-E_1),\ldots,\psi_2({y_K},E-E_1)\big)\neq 0.
\end{gather}
\end{proposition}

\begin{proof}
Since $\mathscr{W}$ is an alternating linear form, for all $E_1\in (V_1(x_N), E-V_2(y_K))$ we  have
\begin{align*}\label{eq:W=0}
\begin{aligned}
&\mathscr{W}\big(\psi_1({x_1},E_1),\ldots,\psi_1({x_N},E_1),\psi_2({y_1},E\!-\! E_1),\ldots,\psi_2({y_K},E\!-\! E_1)\big)
\\&
=\mathscr{W}\big(\psi_1({x_1},E_1),\psi_1({x_2},E_1)-\psi_1({x_1},E_1),\ldots,\psi_1({x_N},E_1)-\psi_1({x_{N-1}},E_1),\\
&\quad\psi_2({y_1},E\!-\! E_1),\psi_2({y_2},E\!-\! E_1)\!-\!\psi_2({y_1},E\!-\! E_1),\ldots,\psi_2({y_K},E\!-\! E_1)\!-\!\psi_2({y_{K-1}},E\!-\! E_1)\big).
\end{aligned}
\end{align*}
Moreover, in view of \eqref{eq:dera1} (in Proposition~\ref{prop:ay}), we have
\begin{gather*}
\frac{d^k}{dE_1^k}\big(\psi_1({x_j},E_1)-\psi_1({x_{j-1}},E_1)\big)= \frac{(-1)^k(2k-1)!!}{2^{\frac{2k+1}{2}}}\int_{x_{j-1}}^{x_j}\frac{1}{(E_1-V_{1}(s_j))^{\frac{2k+1}{2}}}ds_j\\
\frac{d^k}{dE_1^k}\big(\psi_2({y_j},E-E_1)-\psi_2({y_{j-1}},E-E_1)\big)= \frac{(2k-1)!!}{2^{\frac{2k+1}{2}}}\int_{y_{j-1}}^{y_j}\frac{1}{(E-E_1-V_{2}(u_j))^{\frac{2k+1}{2}}}du_j.
\end{gather*}
Hence
\begin{align*}
&\mathscr{W}\big(\psi_1({x_1},E_1),\psi_1({x_2},E_1)-\psi_1({x_1},E_1),\ldots,\psi_1({x_N},E_1)-\psi_1({x_{N-1}},E_1),\\
&\quad\psi_2({y_1},E\!-\! E_1),\psi_2({y_2},E\!-\! E_1)\!-\!\psi_2({y_1},E\!-\! E_1),\ldots,\psi_2({y_K},E\!-\! E_1)\!-\!\psi_2({y_{K-1}},E\!-\! E_1)\big)\\
&=\frac{\displaystyle\prod_{n=1}^{N+K-1}(2n-1)!!}{2^{\frac{(N+K)^2}{2}}}\int_{0}^{x_1}\ldots\int_{x_{N-1}}^{x_N}\int_{0}^{y_1}\ldots\int_{y_{K-1}}^{y_{K}}\det C\,ds_1\ldots ds_N\, du_1\ldots du_K,
\end{align*}
where $C=[c_{jk}]_{1\leq j,k\leq N+K}$ is an $(N+K)\times(N+K)$-matrix given by
\[c_{jk}=\left\{
\begin{array}{ccc}
\displaystyle\frac{(-1)^{k-1}}{(E_1-V_1(s_j))^{\frac{2k-1}{2}}}&\text{ if }&1\leq j\leq N\\
\displaystyle\frac{1}{(E-E_1-V_2(u_{j-N}))^{\frac{2k-1}{2}}}&\text{ if }&N+1\leq j\leq N+K.
\end{array}
\right.\]
Fix
\begin{gather*}
s_1\in(0,x_1),\, s_2\in(x_1,x_2),\ldots, s_N\in(x_{N-1},x_N),\\
u_1\in(0,y_1),\, u_2\in(y_1,y_2),\ldots, u_K\in(y_{K-1},y_K).
\end{gather*}
Then
\begin{gather*}
V_1(s_j)<V_1(s_{j'})\text{ if }1\leq j<j'\leq N,\quad V_2(u_j)<V_2(u_{j'})\text{ if }1\leq j<j'\leq K\text{ and }\\
V_1(s_j)+V_2(u_{j'})\leq V_1(x_N)+V_2(y_K)<E\text{ if }1\leq j\leq N\text{ and }1\leq j'\leq K.
\end{gather*}
By the Vandermonde determinant formula,  we have
\begin{align*}
\det C=&\prod_{j=1}^N\frac{-1}{\sqrt{E_1-V_1(s_j)}}\prod_{1\leq j<j'\leq N}
\Big(\frac{-1}{E_1-V_1(s_{j'})}-\frac{-1}{E_1-V_1(s_j)}\Big)\\
&\cdot\prod_{j=1}^K\frac{1}{\sqrt{E-E_1-V_2(u_j)}}\prod_{1\leq j<j'\leq K}
\Big(\frac{1}{E-E_1-V_2(u_{j'})}-\frac{1}{E-E_1-V_2(u_j)}\Big)\\
&\cdot \prod_{\substack{1\leq j\leq N\\ 1\leq j'\leq K}}
\Big(\frac{1}{E-E_1-V_2(u_{j'})}+\frac{1}{E_1-V_1(s_{j})}\Big)\\
=&(-1)^{\frac{N(N+1)}{2}}\prod_{j=1}^N\frac{1}{(E_1-V_1(s_j))^{N+K-\frac{1}{2}}}
\prod_{j=1}^K\frac{1}{(E-E_1-V_2(u_j))^{N+K-\frac{1}{2}}}\\
&\cdot\prod_{1\leq j<j'\leq N}(V_1(s_{j'})-V_1(s_j))
\prod_{1\leq j<j'\leq K}(V_2(u_{j'})-V_2(u_j))\\
&\cdot\prod_{\substack{1\leq j\leq N\\ 1\leq j'\leq K}}(E-V_1(s_j)-V_2(u_{j'})).
\end{align*}
where, by the assigned intervals of \(s_j,s_{j'},u_j,u_{j'}\),  all elements under the products signs are well defined and positive.  It follows that
\[(-1)^{\frac{N(N+1)}{2}}\mathscr{W}\big(\psi_1({x_1},E_1),\ldots,\psi_1({x_N},E_1),\psi_2({y_1},E-E_1),\ldots,\psi_2({y_K},E-E_1)\big)> 0\]
for all $E_1\in (V_1(x_N), E-V_2(y_K))$.
\end{proof}

\begin{remark}\label{rem:indep}
Since all maps $\psi_1({x_1},\cdot),\ldots,\psi_1({x_N},\cdot),\psi_2({y_1},E-\cdot),\ldots,\psi_2({y_K},E-\cdot)$ are analytic on the interval $(V_1(x_N),E-V_2(y_K))$ (see Proposition~\ref{prop:ay}),
the condition \eqref{neq:ga} implies that
for any sequence $(\gamma_j)_{j=1}^{N+K}$ of at least one non-zero real numbers we have
\begin{equation}\label{eq:sumpsinonvan}
\sum_{j=1}^N\gamma_{j}\psi_1({x_j},E_1)+\sum_{j=1}^k\gamma_{j+N}\psi_2({y_j},E-E_1)\neq 0
\end{equation}
for all but countable many $E_1\in[V_1(x_N),E-V_2(y_K)]$.  Notice that Proposition~\ref{prop:indep2} also implies  Eq.~\eqref{eq:sumpsinonvan}, but under the stronger Deck conditions on the potentials.
\end{remark}

\section{General criterion for unique ergodicity and its application}\label{sec:genequ}
Now consider an interval \(I \) of \(E_1\) values on which the topological data is fixed, so that the numerical data on \(I\), as proved above, depends smoothly on \(E_1\). More generally let $I\ni E_1\mapsto \mathbf{P}(E_1)\in \mathscr{R}$ be a $C^\infty$ curve of polygonal billiard tables in $\mathscr{R}$, i.e.\
\[\mathbf{P}(E_1)=\bigcup_{\varsigma_1,\varsigma_2\in\{\pm\}} P(\overline{x}^{\varsigma_1\varsigma_2}(E_1),
\overline{y}^{\varsigma_1\varsigma_2}(E_1)),
\]
where $x_k^{\varsigma_1\varsigma_2},y_k^{\varsigma_1\varsigma_2}:I\to\R_{>0}$ are $C^\infty$ maps for all $\varsigma_1,\varsigma_2\in\{\pm\}$, $1\leq k\leq k(\overline{x}^{\varsigma_1\varsigma_2},\overline{y}^{\varsigma_1\varsigma_2})$.
Let us consider two finite sets of real $C^\infty$ maps on $I$ given by
\begin{align*}
\mathscr{X}_{\mathbf{P}}:&=\{x_k^{\varsigma_1\varsigma_2}(\cdot): \varsigma_1,\varsigma_2\in\{\pm\},1\leq k\leq k(\overline{x}^{\varsigma_1\varsigma_2},\overline{y}^{\varsigma_1\varsigma_2}) \}\\
\mathscr{Y}_{\mathbf{P}}:&=\{y_k^{\varsigma_1\varsigma_2}(\cdot): \varsigma_1,\varsigma_2\in\{\pm\},1\leq k\leq k(\overline{x}^{\varsigma_1\varsigma_2},\overline{y}^{\varsigma_1\varsigma_2}) \}.
\end{align*}

\begin{theorem}[cf.\ Theorem~4.2 in \cite{frkaczek2019recurrence}]\label{thm:mainfr}
Suppose that
\begin{itemize}
\item[($i$)] for any choice of integer numbers $n_{\mathbf{x}}$ for $\mathbf{x}\in\mathscr{X}_{\mathbf{P}}$ and $m_{\mathbf{y}}$ for $\mathbf{y}\in\mathscr{Y}_{\mathbf{P}}$
such that not all of them are zero, we have
\begin{equation}\label{cond:i}
\sum_{\mathbf{x}\in\mathscr{X}_{\mathbf{P}}}n_{\mathbf{x}}\mathbf{x}(E_1)+\sum_{\mathbf{y}\in\mathscr{Y}_{\mathbf{P}}}m_{\mathbf{y}}\mathbf{y}(E_1)\neq 0\text{ for a.e. }E_1\in I;
\end{equation}
\item[($ii_{+-}$)] for all $\mathbf{x}\in\mathscr{X}_{\mathbf{P}}$ and $\mathbf{y}\in\mathscr{Y}_{\mathbf{P}}$ we have $\mathbf{x}'(E_1)\geq 0$ and $\mathbf{y}'(E_1)\leq 0$ for all $E_1\in I$,
moreover for at least one $\mathbf{x}\in\mathscr{X}_{\mathbf{P}}$ or $\mathbf{y}\in\mathscr{Y}_{\mathbf{P}}$ the inequality is sharp for a.e.\ $E_1\in I$ or;
\item[($ii_{-+}$)] for all $\mathbf{x}\in\mathscr{X}_{\mathbf{P}}$ and $\mathbf{y}\in\mathscr{Y}_{\mathbf{P}}$ we have $\mathbf{x}'(E_1)\leq 0$ and $\mathbf{y}'(E_1)\geq 0$ for all $E_1\in I$,
moreover for at least one $\mathbf{x}\in\mathscr{X}_{\mathbf{P}}$ or $\mathbf{y}\in\mathscr{Y}_{\mathbf{P}}$ the inequality is sharp for a.e.\ $E_1\in I$.
\end{itemize}
Then for a.e.\ $E_1\in I$ the billiard flow on $\mathbf P(E_1)$ in directions $\pm\pi/4,\pm 3\pi/4$ is uniquely ergodic.
\end{theorem}

\begin{proof}
We show that the above conditions imply some intermediate steps of Theorems~4.2 in \cite{frkaczek2019recurrence} which are used to show that the results of Theorem~2.11 in \cite{frkaczek2019recurrence} about unique ergodicity on surfaces imply  the unique ergodicity on the related polygons.

First, we take the reference function in Theorems~4.2  to be a constant  ($\ell=1$). Second, conditions ($i$)  in Theorem~4.2 is used to prove the above condition \eqref{cond:i} which is then used to prove that condition  ($i$) in Theorem~2.11 in \cite{frkaczek2019recurrence} holds. Hence, by the same reasoning as in  Theorem~4.2 in \cite{frkaczek2019recurrence}  assumption  ($i$)   implies assumption ($i$) of Theorem~2.11 in \cite{frkaczek2019recurrence}.

Third, similarly, conditions ($ii$)  in Theorem~4.2 are used to prove the above conditions  ($ii$)   which are then used to prove that condition  ($ii$) in Theorem~2.11 in \cite{frkaczek2019recurrence} holds. Hence, by the same reasoning as in  Theorem~4.2 in \cite{frkaczek2019recurrence}  assumption ($ii$)  implies assumption ($ii$)  of Theorem~2.11 in \cite{frkaczek2019recurrence}.

We conclude that by Theorem~2.11 in \cite{frkaczek2019recurrence} the billiard flow on $\mathbf P(E_1)$ in directions $\pm\pi/4,\pm 3\pi/4$ is uniquely ergodic for a.e. \(E_{1}\in I\).
\end{proof}

Assume $V_1,V_2:\R\to\R_{\geq 0}$ are two $Deck$-potentials satisfying \eqref{eq:v}.
Recall that, by Lemma~\ref{lem:square}, if $V_i$ does not satisfy \eqref{eq:vi}, then $V_i$ is quadratic.

\begin{theorem}\label{thm:albegacopdet}
Assume $V_1,V_2:\R\to\R_{\geq 0}$ are $Deck$-potentials satisfying \eqref{eq:v}. Let $P$ be any  polygon in $\mathscr R$.
Suppose that
\begin{itemize}
\item[($\alpha$)] at least one potential $V_1$ or $V_2$ satisfies \eqref{eq:vi} or;
\item[($\beta$)] both $V_1,V_2$ are quadratic maps such that $V_1=\Omega^{2} V_2$  and ${\Omega}$ is irrational.
\end{itemize}
Then for every energy level $E>0$ and almost every $E_1\in [0,E]$ the restricted Hamiltonian flow
$(\varphi^{P,E,E_1}_t)_{t\in\R}$ is uniquely ergodic.

Suppose that
\begin{itemize}
\item[($\gamma$)] both $V_1,V_2$ are quadratic maps such that $V_1=\Omega^{2} V_2$  and ${\Omega}$ is rational.
\end{itemize}
Then for any $E>\min\{\max_{\varsigma_1,\varsigma_2\in\{\pm\}}V_2(y_1^{\varsigma_1\varsigma_2}),\max_{\varsigma_1,\varsigma_2\in\{\pm\}}V_1(x_{k(\bar{x}^{\varsigma_1\varsigma_2},\bar{y}^{\varsigma_1\varsigma_2})}^{\varsigma_1\varsigma_2})\}
$ and almost every
\[E_1\in \Big[0,E-\max_{\varsigma_1,\varsigma_2\in\{\pm\}}V_2(y_1^{\varsigma_1\varsigma_2})\Big]
\cup \Big[\max_{\varsigma_1,\varsigma_2\in\{\pm\}}V_1(x_{k(\bar{x}^{\varsigma_1\varsigma_2},\bar{y}^{\varsigma_1\varsigma_2})}^{\varsigma_1\varsigma_2}),E\Big]\]
the restricted Hamiltonian flow
$(\varphi^{P,E,E_1}_t)_{t\in\R}$ is uniquely ergodic.
\end{theorem}

\begin{proof}
We fix an energy \(E>0\). First note that we can restrict our attention to any subinterval $I\in \mathcal J_E$. As we already have observed,
for every $E_1\in I$ the flow $(\varphi^{P,E,E_1}_t)_{t\in\R}$ is topologically conjugated to the billiard flow in directions $\pm\pi/4,\pm3\pi/4$ on the polygon $\mathbf P(E_{1}):=\mathbf P_{E_{1},E}=\psi (P\cap R^{(E_{1},E)})\in \mathscr R$. Moreover, by Remark~\ref{rem:XYI} and Propositions~\ref{prop:a}~and~\ref{prop:ay}, we have
\begin{gather}\label{eq:XYP}
\begin{split}
\mathscr X_\mathbf{P}&\subset\{\psi_{1}(x,E_1):{x}\in X_I\}\cup\{\frac{1}{4}T_{1}(E_1)\},\\
 \mathscr Y_\mathbf{P}&\subset\{\psi_{2}(y,E-E_1):{y}\in Y_I\}\cup\{\frac{1}{4}T_{2}(E-E_1)\}
\end{split}
\end{gather}
and the curve $I\ni E_1\mapsto \mathbf P(E_1)\in \mathscr R$
is analytic.

\medskip

\noindent\textbf{Cases ($\alpha$) and ($\beta$).}
Assume that the sets $\mathscr X_\mathbf{P}$, $\mathscr Y_\mathbf{P}$ do not satisfy the condition ($i$) in Theorem~\ref{thm:mainfr}.
In view of Proposition~\ref{prop:indep2}, there exists a rational positive number $\gamma>0$ such that $T_{2}(E-\cdot)=\gamma T_{1}(\cdot)$ (since condition  ($i$) involves integer coefficients).
However, by Lemma \ref{lem:indep3} this is impossible if ($\alpha$) is satisfied. If ($\beta$) is satisfied then, by Lemma \ref{lem:indep3}, we have $\gamma=-{\Omega}$ which  contradicts  the
irrationality of \(\Omega\).

In summary, it follows that either under assumption ($\alpha$) or ($\beta$), the condition ($i$) of Theorem~\ref{thm:mainfr} holds.

\medskip

Now we verify (in both cases ($\alpha$) and ($\beta$)) the condition ($ii_{-+}$) of Theorem~\ref{thm:mainfr} holds. First suppose that
($\alpha$) holds. Then, by Lemma~\ref{lem:negposder}, for all $\mathbf{x}\in\mathscr{X}_{\mathbf{P}}$ and $\mathbf{y}\in\mathscr{Y}_{\mathbf{P}}$
and for every
$E_1\in I$ we have
\begin{itemize}
\item $\mathbf{x}'(E_1)<0$ and $\mathbf{y}'(E_1)\geq 0$, if $V_1$ satisfies \eqref{eq:vi};
\item $\mathbf{x}'(E_1)\leq 0$ and $\mathbf{y}'(E_1)>0$, if $V_2$ satisfies \eqref{eq:vi},
\end{itemize}
so we have ($ii_{-+}$) of Theorem~\ref{thm:mainfr}.

Suppose that ($\beta$) holds, then $T_{1}'(E_1)=T_{2}'(E-E_1)=0$ for every
$E_1\in I$.  Assume first that there is at least one impact for the level sets in \(I\), so $\mathscr{X}_\mathbf{P}\neq\{\frac{1}{4}T_{1}\}$ or $\mathscr{Y}_\mathbf{P}\neq\{\frac{1}{4}T_{2}\}$.
 Then, by Lemma~\ref{lem:negposder}, for all $\mathbf{x}\in\mathscr{X}_{\mathbf{P}}\setminus\{\frac{1}{4}T_{1}\}$ and $\mathbf{y}\in\mathscr{Y}_{\mathbf{P}}\setminus\{\frac{1}{4}T_{2}\}$ we have $\mathbf{x}'(E_1)<0$ and $\mathbf{y}'(E_1)>0$ for every
$E_1\in I$, so we also have ($ii_{-+}$) of Theorem~\ref{thm:mainfr}, whenever the union of these sets is non-empty.

In summary, in both these cases the unique ergodicity of $(\varphi^{P,E,E_1}_t)_{t\in\R}$ for a.e.\ $E_1\in I$ follows directly from Theorem~\ref{thm:mainfr}.

\medskip

Finally, when ($\beta$) holds and $\mathscr{X}_\mathbf{P}=\{\frac{1}{4}T_{1}\}$ and $\mathscr{Y}_\mathbf{P}=\{\frac{1}{4}T_{2}\}$ the motion on all level sets in \(I\) occurs with no impacts at all, namely, the motion corresponds to   the billiard flow on the rectangle  $\mathbf{P}_I=\mathbf{P}(E_1)$  in directions $\pm\pi/4,\pm 3\pi/4$ and since  \(\Omega\) is irrational, the motion is also unique ergodic for all \(E_{1}\in I\).

\medskip
\noindent
\textbf{Case ($\gamma$).} We consider a subinterval $I\in \mathcal J_E$ such that impacts occur with either all the horizontal boundaries of \(P\) (the intersection of all light green wedges in the IEMBD figures) or with all the vertical boundaries of \(P\) (the intersection of all pink regions in the IEMBD figures):
\[I\subset \Big[0,E-\max_{\varsigma_1,\varsigma_2\in\{\pm\}}V_2(y_1^{\varsigma_1\varsigma_2})\Big] \ \text{ or }\
I\subset \Big[\max_{\varsigma_1,\varsigma_2\in\{\pm\}}V_1(x_{k(\bar{x}^{\varsigma_1\varsigma_2},\bar{y}^{\varsigma_1\varsigma_2})}^{\varsigma_1\varsigma_2}),E\Big].\]
Then, in the first case we have
\(Y_I=Y\) and in the second case \(X_I=X\) and
 \begin{equation}\label{eq:XYgamma}
\mathscr Y_\mathbf{P}=\{\psi _{2}(y,E-E_1):{y}\in Y\} \ \text{ or }\ \mathscr X_\mathbf{P}\subset\{\psi _{1}(x,E_1):{x}\in X\}\ \text{ respectively}
\end{equation}
(see Remark \ref{rem:XYI}).
Suppose that the condition ($i$) in Theorem~\ref{thm:mainfr} does not hold, then, in the first case we have
 \[k\frac{1}{4}T_{1}(E_1)+\sum_{{x}\in X_I}n_x\psi _{1}(x,E_1)+\sum_{{y}\in Y_I}m_y\psi _{2}(y,E-E_1)=0\]
whereas in the second case we have
\[\sum_{{x}\in X_I}n_x\psi _{1}(x,E_1)+\sum_{{y}\in Y_I}m_y\psi _{2}(y,E-E_1)+k\frac{1}{4}T_{2}(E-E_1)=0\]
on a subset of positive measure, where \(|k|+\sum _{x\in X_I}|n_{x}|+\sum _{y\in Y_I}|m_{y}|\neq 0\). It follows that in both cases at least one $n_x$, $x\in X_I$ or $m_y$, $y\in Y_I$ is non-zero. This contradicts the conclusion of Proposition~\ref{prop:indep2}.

Finally, we check that condition ($ii_{-+}$)\ in Theorem~\ref{thm:mainfr}   is satisfied when ($\gamma$) holds. For every
$E_1\in I$, by Lemma~\ref{lem:negposder}, for all $\mathbf{x}\in\mathscr{X}_{\mathbf{P}}\setminus\{\frac{1}{4}T_{1}\}$ and $\mathbf{y}\in\mathscr{Y}_{\mathbf{P}}\setminus\{\frac{1}{4}T_{2}\}$ we have $\mathbf{x}'(E_1)<0$ and $\mathbf{y}'(E_1)>0$.  Moreover, by \eqref{eq:XYgamma}, when ($\gamma$) holds, at least one of these sets is non-empty.  Hence we also have ($ii_{-+}$) in Theorem~\ref{thm:mainfr} holds. Finally the unique ergodicity of $(\varphi^{P,E,E_1}_t)_{t\in\R}$ for a.e.\ $E_1\in I$  follows directly from Theorem~\ref{thm:mainfr}.
This completes the proof.
\end{proof}

\begin{corollary}\label{cor:highE}
Suppose that $V_1,V_2$ are quadratic so that $V_1=\Omega^{2} V_2$  and ${\Omega}$ is rational.
If the energy is sufficiently large:
\begin{equation}\label{eq:bigE}
E\geq\max_{\varsigma_1,\varsigma_2\in\{\pm\}}V_1(x_{k(\bar{x}^{\varsigma_1\varsigma_2},\bar{y}^{\varsigma_1\varsigma_2})}^{\varsigma_1\varsigma_2})
+\max_{\varsigma_1,\varsigma_2\in\{\pm\}}V_2(y_1^{\varsigma_1\varsigma_2})
\end{equation}
then the flow
$(\varphi^{P,E,E_1}_t)_{t\in\R}$ is uniquely ergodic for a.e.\ $E_1\in[0,E]$.
\end{corollary}
This is a weaker version of Theorem~\ref{thm:highE} and it will be helpful in the proof of Theorem~\ref{thm:highE}.

\medskip

At the end of this section we give the following partial result, which is met with very slight assumptions on potentials $V_1$ and $V_2$.
\begin{proposition}\label{prop:largegmax}
Suppose that $V_1,V_2:\R\to\R_{\geq 0}$ are even $C^2$-potentials satisfying \eqref{eq:i}. Assume that
\begin{equation*}
E>\max_{\varsigma_1,\varsigma_2\in\{\pm\}}V_1(x_{k(\bar{x}^{\varsigma_1\varsigma_2},\bar{y}^{\varsigma_1\varsigma_2})}^{\varsigma_1\varsigma_2})
+\max_{\varsigma_1,\varsigma_2\in\{\pm\}}V_2(y_1^{\varsigma_1\varsigma_2}).
\end{equation*}
and let
\[E_1\in I:=\Big[\max_{\varsigma_1,\varsigma_2\in\{\pm\}}V_1(x_{k(\bar{x}^{\varsigma_1\varsigma_2},\bar{y}^{\varsigma_1\varsigma_2})}^{\varsigma_1\varsigma_2}),
E-\max_{\varsigma_1,\varsigma_2\in\{\pm\}}V_2(y_1^{\varsigma_1\varsigma_2})\Big].\]
Then the flow
$(\varphi^{P,E,E_1}_t)_{t\in\R}$ is uniquely ergodic for a.e.\ $E_1\in I$.
\end{proposition}

\begin{proof}
The argument used at the beginning of the proof of the part ($\gamma$) in Theorem~\ref{thm:albegacopdet} shows that
\begin{equation*}
 \mathscr X_\mathbf{P}\subset\{\psi _{1}(x,E_1 ):{x}\in X\} \ \text{ and }\ \mathscr Y_\mathbf{P}=\{\psi _{2}(y,E-E_1 ):{y}\in Y\}.
\end{equation*}
Notice that Proposition~\ref{prop:indep1} combined with Remark~\ref{rem:indep} shows that the condition ($i$) in Theorem~\ref{thm:mainfr}
is satisfied. The negativity of the derivative for all maps from $\mathscr X_\mathbf{P}$ and the positivity of the derivative for all maps from $\mathscr Y_\mathbf{P}$ follow directly from Lemma~\ref{lem:negposder}. Thus the application of Theorem~\ref{thm:mainfr} again completes the proof.
\end{proof}
The interval \(I \) in Proposition \ref{prop:largegmax} corresponds to the case of impacts with all boundaries of \(P\) (the intersection of all the coloured wedges in the IEMBD figures), and, as will be shown in Section~\ref{sec:billtotrans}, this corresponds to motion on surface of genus \(g_{max}\).

\section{Topological data revisited}

\subsection{Short introduction to translation surfaces}
 Since our main criterion for unique ergodicity (Theorem~\ref{thm:mainquadrational}) is formulated in the framework of translation surfaces, in this section we give a short introduction to this subject. For further background material we refer the reader to \cite{Viana2008lecturenotes}, \cite{Yo} and \cite{Zorich2006}.

\medskip

A \emph{translation surface} $(M,\omega)$ is a compact connected orientable topological surface $M$, together with a finite set
of points $\Sigma$ (called \emph{singular} points) and an atlas of charts $\omega=\{\zeta_\alpha:U_\alpha\to \C:\alpha\in\mathcal{A}\}$ on $M\setminus \Sigma$ such that every transition map
$\zeta_\beta\circ\zeta^{-1}_\alpha:\zeta_\alpha(U_\alpha\cap U_\beta)\to \zeta_\beta(U_\alpha\cap U_\beta)$ is a translation, i.e.\ for every connected component $C$ of $U_\alpha\cap U_\beta$ there exists
$v_{\alpha,\beta}^C\in \C$ such that  $\zeta_\beta\circ\zeta^{-1}_\alpha(z)=z+v^C_{\alpha,\beta}$ for $z\in \zeta_\alpha^{-1}(C)$. All points in $M\setminus \Sigma$ are called \emph{regular}. For every point $x\in M$ the translation structure $\omega$ allow us to define the total angle around $x$.
If $x$ is regular then the total angle is $2\pi$. If $\sigma$ is singular then the total angle is $2\pi(k_\sigma+1)$, where $k_\sigma\in\N$ is the multiplicity of $\sigma$. Then
\begin{equation}\label{eq:genus}
\sum_{\sigma\in\Sigma}k_\sigma=2g-2,
\end{equation}
where $g$ is the genus of the surface $M$.

\medskip

For every $\theta\in\R/2\pi\Z$ let $X_{\theta}$ be a tangent vector field on $M\setminus\Sigma$ which is the pullback of the unit constant vector field $e^{i\theta}$ on $\C$ through the charts of the atlas.
Since the derivative of any transition map is the identity, the vector field  $X_{\theta}$ is well defined on $M\setminus\Sigma$.
Denote by $(\psi^{\theta}_t)_{t\in\R}$ the corresponding local flow, called the translation flow on $(M,\omega)$ in direction $\theta$. The flow preserves the measure $\lambda_{\omega}$ which is the pullback of the Lebesgue measure on $\C$. We distinguish the vertical flow $(\psi^v_t)_{t\in\R}$, i.e.\ for $\theta=\pi/2$.

For every $\theta\in\R/2\pi\Z$ and a translation surface $(M,\omega)$ denote by $(M,e^{i\theta}\omega)$ the rotated translation surface,
i.e.\ the  new charts in $e^{i\theta}\omega$ are defined by postcomposition of charts from $\omega$ with the rotation by $\theta$.
Then the flow $(\psi^{\theta}_t)_{t\in\R}$ on $(M,\omega)$ coincide with the vertical flow $(\psi^v_t)_{t\in\R}$ on $(M,e^{i(\frac{\pi}{2}-\theta)}\omega)$.

A \emph{saddle connection} in direction $\theta$ is an orbit segment of $(\psi^{\theta}_t)_{t\in\R}$ that goes
from a singularity to a singularity (possibly, the same one) and has no interior singularities. A semi-infinite orbit of $(\psi^{\theta}_t)_{t\in\R}$ that goes from or to a singularity is  called  a \emph{separatrix}. Recall that if $(M,\omega)$ has no saddle connection in direction $\theta$, then the flow $(\psi^{\theta}_t)_{t\in\R}$ is \emph{minimal}, i.e.\
every  orbit (which can be semi-infinite or double-infinite) is dense in $M$, see \cite{Yo}.

\subsection{From billiards to translation surfaces}\label{sec:billtotrans}
Formally the directional billiard flow on $\mathbf{P}(E_1)=\mathbf P_{E,E_1}$ in directions $\pm\pi/4, \pm 3\pi/4$ acts on the union of four copies of $\mathbf{P}(E_1)$, denoted by  $\mathbf{P}(E_1)_{\pi/4}$,
$\mathbf{P}(E_1)_{-\pi/4}$, $\mathbf{P}(E_1)_{3\pi/4}$, $\mathbf{P}(E_1)_{-3\pi/4}$. Each copy $\mathbf{P}(E_1)_{\vartheta}$ for $\vartheta\in \{\pm\pi/4$, $\pm 3\pi/4\}$ represents
all unit vectors flowing in the same direction $\vartheta$. After applying the horizontal or vertical reflection (or both) to each copy separately, we can arrange all unit vectors to flow in the same direction $\pi/4$.
More precisely, after such transformations, all unit vectors in $\mathbf{P}(E_1)_{\pi/4}$, $\gamma_h\mathbf{P}(E_1)_{-\pi/4}$, $\gamma_v\mathbf{P}(E_1)_{3\pi/4}$ and $\gamma_h\circ\gamma_v\mathbf{P}(E_1)_{-3\pi/4}$
flow in the same direction $\pi/4$. By gluing corresponding sides of these four polygons, we get a compact connected orientable surface $M(E_1)$ with a translation structure inherited from the Euclidean plan, see Figure~\ref{fig:surface}.
Moreover, the directional billiard flow on $\mathbf{P}(E_1)$ in directions $\pm\pi/4$, $\pm 3\pi/4$ is conjugate to the translation flow $(\psi^{\pi/4}_t)_{t\in\R}$ in direction $\pi/4$  on the translation surface $M(E_1)$. This is an example of using the so called unfolding procedure coming from \cite{Fox-Ker} and \cite{Ka-Ze}. Additionally, the surface $M(E_1)$ has a natural partition into $16$ staircase polygons $\{\mathbf{P}(E_1)^{\varsigma_1\varsigma_2}_{\sigma_1\sigma_2}: \varsigma_1,\varsigma_2,\sigma_1,\sigma_2\in\{\pm\}\}$ so that
\begin{gather*}
\mathbf{P}(E_1)^{\varsigma_1\varsigma_2}_{++}=\mathbf{P}(E_1)^{\varsigma_1\varsigma_2}_{\pi/4},\quad
\mathbf{P}(E_1)^{\varsigma_1\varsigma_2}_{+-}=\gamma_h\mathbf{P}(E_1)^{\varsigma_1\varsigma_2}_{-\pi/4},\\
\mathbf{P}(E_1)^{\varsigma_1\varsigma_2}_{-+}=\gamma_v\mathbf{P}(E_1)^{\varsigma_1\varsigma_2}_{3\pi/4},\quad
\mathbf{P}(E_1)^{\varsigma_1\varsigma_2}_{--}=\gamma_h\circ \gamma_v\mathbf{P}(E_1)^{\varsigma_1\varsigma_2}_{-3\pi/4}.
\end{gather*}

\begin{figure}[h]
\includegraphics[width=1 \textwidth]{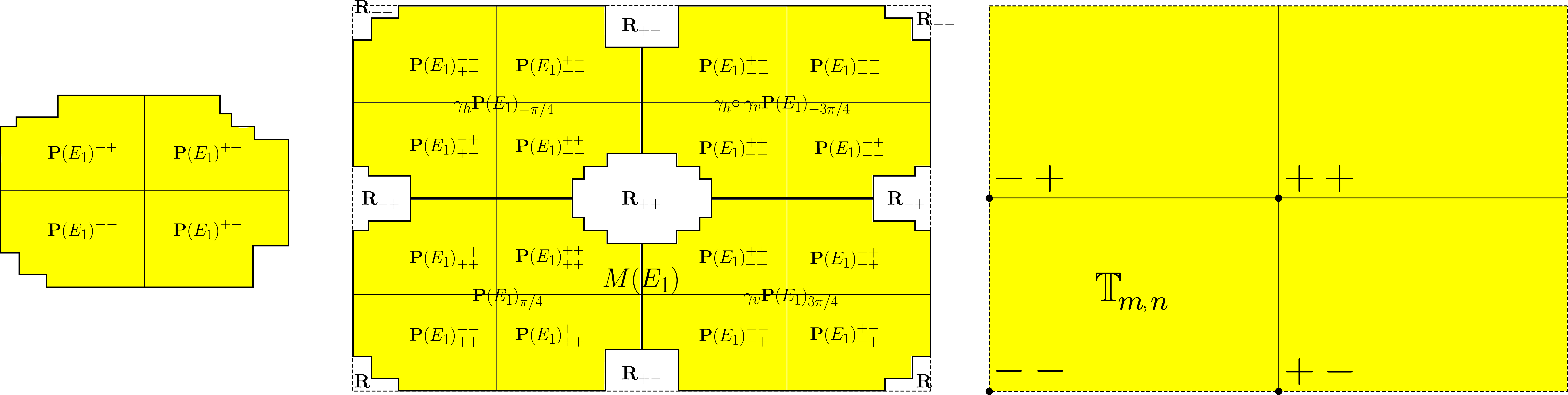}
\caption{The billiard table $\mathbf{P}(E_1)$, the translation surface $M(E_1)$ and the marked torus $\T_{m,n}$. }
\label{fig:surface}
\end{figure}

\begin{proof}[Proof of Theorem \ref{thm:genus}]
The above defined translation surface $M(E_1)$ is formed by cutting out four polygons, denoted by $\mathbf R^{++},\mathbf R^{+-},\mathbf R^{-+},\mathbf R^{--}$, from the torus and gluing their opposite sides (see white polygons in the middle part of Figure~\ref{fig:surface}).

Let us consider corners of a removed polygon $\mathbf R^{\varsigma_1\varsigma_2}$.
Notice that convex corners (of $\mathbf{P}(E_1)^{\varsigma_1\varsigma_2}$) give rise to regular points whereas concave corners give rise to singular points of multiplicity $2$.
Indeed, each convex corner is a point on $M(E_1)$ with the total angle $2\pi$ (the pink point on Figure~\ref{fig:removed}), whereas
 each  concave corner is a point on $M(E_1)$ with the total angle $4\tfrac{3\pi}{2}=6\pi$ (the blue point on Figure~\ref{fig:removed}).
\begin{figure}[h]
\includegraphics[width=0.4 \textwidth]{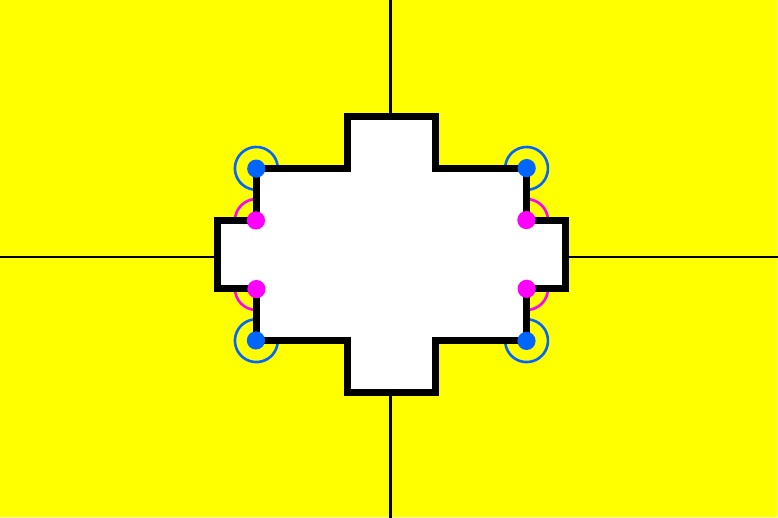}
\caption{Singularities of  $M(E_1)$.}
\label{fig:removed}
\end{figure}
In view of Eq.~\eqref{eq:genus}, it follows that the genus $g(E,E_1)$ of $M(E_1)$ (and equivalently  the genus of the level set $S^P_{E,E_1}$) is the number of concave corners in \(P\cap R^{(E,E_{1})}\) plus one:
\begin{align*}
g(E,E_{1})&=1+\sum_{\varsigma_1,\varsigma_2\in\{\pm\}}\max\{0,(\overline{k}^{\varsigma_1\varsigma_2}(E_1)-\underline{k}^{\varsigma_1\varsigma_2}(E_1))\}\\
&=1+\sum_{\varsigma_1,\varsigma_2\in\{\pm\}}\#\{1\leq k<k(\bar{x}^{\varsigma_1\varsigma_2},\bar{y}^{\varsigma_1\varsigma_2}):V_1(x^{\varsigma_1\varsigma_2}_k)<E_1<E-V_2(y^{\varsigma_1\varsigma_2}_{k+1})\}.
\end{align*}
Moreover,
\begin{align*}
g(E,E_1)&\leq
1+\sum_{\varsigma_1,\varsigma_2\in\{\pm\}}(k(\bar{x}^{\varsigma_1\varsigma_2},\bar{y}^{\varsigma_1\varsigma_2})-1)\\
&=\sum_{\varsigma_1,\varsigma_2\in\{\pm\}}k(\bar{x}^{\varsigma_1\varsigma_2},\bar{y}^{\varsigma_1\varsigma_2})-3=:g_{\max}
\end{align*}
and the equality holds if and only if $\overline{k}^{\varsigma_1\varsigma_2}(E_1)=k(\bar{x}^{\varsigma_1\varsigma_2},\bar{y}^{\varsigma_1\varsigma_2})$ and $\underline{k}^{\varsigma_1\varsigma_2}(E_1)=1$
for any $\varsigma_1,\varsigma_2\in\{\pm\}$. If
\[E_{1}\in[0,\min_{\varsigma_1,\varsigma_2\in\{\pm\}}V_1(x_{1}^{\varsigma_1\varsigma_2})]\cup[\max_{\varsigma_1,\varsigma_2\in\{\pm\}}E-V_2(y_{k(\bar{x}^{\varsigma_1\varsigma_2},\bar{y}^{\varsigma_1\varsigma_2})}^{\varsigma_1\varsigma_2}),E],\]
then the polygon \(\mathbf{P}_{E,E_{1}}\) is a rectangle, so, for all \(E\), close to the end points of  the \(E_{1}\)-interval, the genus of \(M(E_1)\) is one.
Suppose that
\[E>\max_{\varsigma_1,\varsigma_2\in\{\pm\}}V_1(x_{k(\bar{x}^{\varsigma_1\varsigma_2},\bar{y}^{\varsigma_1\varsigma_2})}^{\varsigma_1\varsigma_2})+\max_{\varsigma_1,\varsigma_2\in\{\pm\}}V_2(y_1^{\varsigma_1\varsigma_2}).\]
Then, for
\[ E_1\in I_{\max}:=\big[\max_{\varsigma_1,\varsigma_2\in\{\pm\}}V_1(x_{k(\bar{x}^{\varsigma_1\varsigma_2},\bar{y}^{\varsigma_1\varsigma_2})}^{\varsigma_1\varsigma_2}),
E-\max_{\varsigma_1,\varsigma_2\in\{\pm\}}V_2(y_1^{\varsigma_1\varsigma_2})\big]\Rightarrow g(E,E_1)=g_{\max}.\]
In view of Proposition~\ref{prop:largegmax}, the flow
$(\varphi^{P,E,E_1}_t)_{t\in\R}$ is uniquely ergodic for a.e.\ $E_1\in I_{\max}$. Notice that in general, the interval of \(E_{1}\) values for which \(g(E,E_1)=g_{\max}\)  is larger than \(I_{\max}\)  as the maximal genus interval includes cases at which all concave corners of the polygon \(P \) are inside \(R^{(E,E_1)}\) yet not   all of its extremal boundaries are inside \(R^{(E,E_1)}\).
\end{proof}

\section{The case when $V_1$ and $V_2$ are quadratic and $\Omega$ is rational}\label{sec:quadrat}

In this section we deal with the remaining case when both potentials $V_1$ and $V_2$ are quadratic with $V_1=\Omega^2 V_2$, $\Omega$ rational
and the energy $E$ is not too high, i.e.\ the condition \eqref{eq:bigE} does not hold.
The rationality of  $\Omega$ enables the appearance of periodic orbits and some coexistence of periodic orbits with uniquely ergodic components.

First we focus on the case when the energy is low.
\begin{proof}[Proof of Proposition~\ref{prop:formIcp}]
As $I$ satisfies \eqref{eq:formIcp},  in view of \eqref{eq:rectab}, \(R^{(E,E_1)}\cap P\) is a rectangle for every $E_1\in I$ so that  its width and height are $\tfrac{1}{2}T_{1}(E_1)$ and  $\tfrac{1}{2}T_{2}(E-E_1)$. Since $T_1$ and $T_2$ are constants with $T_{2}/T_{1}=\Omega$ rational, the  flow
$(\varphi^{P,E,E_1}_t)_{t\in\R}$ is isomorphic to the translation flow in a rational direction on the standard torus $\R^2/\Z^2$, so it is completely periodic for all $E_1\in I$.
\end{proof}

\begin{proof}[Proof of Proposition~\ref{prop:lowE}]
 Eq.~\eqref{eq:lowE} implies that for every $\varsigma_1,\varsigma_2\in\{\pm\}$ we have
\[E\leq V_1(x^{\varsigma_1\varsigma_2}_k)+V_2(y^{\varsigma_1\varsigma_2}_{k+1})\text{ for every }0\leq k\leq k(\bar{x}^{\varsigma_1\varsigma_2},\bar{y}^{\varsigma_1\varsigma_2}).\]
Hence, for all \(E_{1}\in [0,E]\),  \(R^{(E,E_1)}\cap P\) is a rectangle (with no intersections of \(P\) boundaries), so as above the motion is periodic.
\end{proof}

Propositions~\ref{prop:lowE}~and~\ref{prop:formIcp} describe completely all intervals $I\in\mathcal J_E$ for which \textbf{(cp)}
scenario can appear (the pure grey regions  in the IEMBD figures).
We now focus on the description of all intervals for which \textbf{(ue)} or \textbf{(coex)} scenario occur.
They are presented in Theorems~\ref{thm:highmedH}~and~\ref{thm:U+-}. Unfortunately, Theorem~\ref{thm:mainfr} is not effective enough to prove these
theorems. Their proofs (presented in Section~\ref{sec:proofs})
need a more subtle version of Theorem~\ref{thm:mainfr}, i.e.\ Theorem~\ref{thm:mainquadrational} in Section~\ref{sec:spec}.
Theorem~\ref{thm:mainquadrational} helps to prove the coexistence of periodic orbits with uniquely ergodic components for intermediate energies. To implement this plan we must now go deeper into the proof
of Theorem~2.11 in \cite{frkaczek2019recurrence} and into the framework of translation surfaces.

In the rational case there is no problem with verifying the condition ($ii_{-+}$) in Theorem~\ref{thm:mainfr}. Verifying the condition ($i$) in Theorem~\ref{thm:mainfr} is impossible.
Indeed, the condition ($i$)  implies the minimality of the billiard flow (every orbit is dense) for a.e.\ parameter $E_1$. Because of the existence of periodic orbit, the condition ($i$) cannot be met.

\medskip

Suppose that $\Omega=n/m$ for some coprime natural numbers $m$ and $n$. Then $V_1(x)=n^2c^2x^2/2$ and $V_2(y)=m^2c^2y^2/2$ for some $c>0$.
In view of Lemma~\ref{lem:square}, we have $T_{1}(E_1)=\tfrac{2\pi}{n{c}}$ and $T_{2}(E-E_1)=\tfrac{2\pi}{m{c}}$,
so
\[\frac{1}{4}T_{1}(E_1)=mC\frac{\pi}{2}\text{ and }\frac{1}{4}T_{2}(E-E_1)=nC\frac{\pi}{2}\text{ for }C=\frac{1}{mn{c}}.\]
Moreover,
\begin{align*}
\psi_{1}(x,E_1)&=mC\arcsin\sqrt{\frac{V_1(x)}{E_1}}\text{ for }E_1\geq V_1(x),\\
\psi_{2}(y,E-E_1)&=nC\arcsin\sqrt{\frac{V_2(y)}{E-E_1}}\text{ for }E_1\leq E- V_2(y).
\end{align*}

Suppose that
\begin{equation}\label{eq:assumE}
E<\min_{\varsigma_1,\varsigma_2\in\{\pm\}}V_1(x^{\varsigma_1\varsigma_2}_{k(\bar{x}^{\varsigma_1\varsigma_2},\bar{y}^{\varsigma_1\varsigma_2})})+\min_{\varsigma_1,\varsigma_2\in\{\pm\}}V_2(y^{\varsigma_1\varsigma_2}_{1}).
\end{equation}
For higher energies, Theorem~\ref{thm:highE} yields \textbf{(ue)}-scenario for the interval $[0,E]$.

In fact, by Theorem~\ref{thm:highmedH},  for \(E>\min\{\min_{\varsigma_1,\varsigma_2\in\{\pm\}}V_2(y_1^{\varsigma_1\varsigma_2}),\min_{\varsigma_1,\varsigma_2\in\{\pm\}}V_1(x_{k(\bar{x}^{\varsigma_1\varsigma_2},\bar{y}^{\varsigma_1\varsigma_2})}^{\varsigma_1\varsigma_2})\}\), the same conclusion occurs  in the intervals
\[ \big[0,E-\min_{\varsigma_1,\varsigma_2\in\{\pm\}}V_2(y_1^{\varsigma_1\varsigma_2})\big]\text{ and }
\Big[\min_{\varsigma_1,\varsigma_2\in\{\pm\}}V_1(x_{k(\bar{x}^{\varsigma_1\varsigma_2},\bar{y}^{\varsigma_1\varsigma_2})}^{\varsigma_1\varsigma_2}),E\Big],\]
where at least one of the extremal side of \(P\) intersects  \( R^{(E,E_{1})}\). Therefore, we focus our attention on
 $I\in\mathcal J_E$ such that
\begin{equation}\label{eq:assumI}
I\subset I_{intimp}=\big(E-\min_{\varsigma_1,\varsigma_2\in\{\pm\}}V_2(y_1^{\varsigma_1\varsigma_2}),\min_{\varsigma_1,\varsigma_2\in\{\pm\}}V_1(x_{k(\bar{x}^{\varsigma_1\varsigma_2},\bar{y}^{\varsigma_1\varsigma_2})}^{\varsigma_1\varsigma_2})\big)\setminus I_{nonimp}.
\end{equation}
So now we consider level sets that do impact some parts of the polygon sides but none of its extremal sides, namely, the width of the polygon $\mathbf{P}_{E,E_1}$
is $\frac{1}{2}T_{1}(E_1)$, its height is $\frac{1}{2}T_{2}(E-E_1)$ and  $\mathbf{P}_{E,E_1}$ is not a rectangle.
Then the width of $M(E_1)$ is $T_{1}(E_1)=2\pi mC$ and its height is $T_{2}(E-E_1)=2\pi nC$. Let us consider the translation torus
\[\T_{m,n}=\R^2/(2\pi mC\Z\times 2\pi nC\Z)\]
with four marked Weierstrass points $(0,0),(\pi mC,0),(0,\pi nC),(\pi mC,\pi nC)$ labeled by $--,+-,-+,++$ respectively, see Figure~\ref{fig:surface}. Then $M(E_1)$ can be treated as $\T_{m,n}$ with four polygons cut out. In this \(E_{1}\) interval at least one removed polygon is non-empty.
{ Removed polygons (even trivial) are denoted by $\mathbf{R}^{--},\mathbf{R}^{+-},\mathbf{R}^{-+},\mathbf{R}^{++}$ as in Figure~\ref{fig:surface}, i.e.\ $\mathbf{R}^{\varsigma_1\varsigma_2}$ is associated with the polygon $\mathbf{P}(E_1)^{\varsigma_1\varsigma_2}$.
Each removed polygon $\mathbf{R}^{\varsigma_1\varsigma_2}$ has only vertical and horizontal sides,
is vertically and horizontally symmetric and its center of symmetry coincide with the marked point $\varsigma_1\varsigma_2$, see Figure~\ref{fig:surface}.} Moreover, the opposite sides of each removed polygon are identified in $M(E_1)$.

We equip the translation  torus $\T_{m,n}$  with the quotient taxicab metric $d$ (the sum of horizontal and vertical distances).
Denote by $e^{i\pi}:\T_{m,n}\to\T_{m,n}$ the hyperelliptic involution that fixes each marked point, namely \(e^{i\pi}\) is the rotation by \(\pi\).

Let us consider periodic orbits in direction $\pi/4$ on $\T_{m,n}$ passing thorough all four marked points. There are exactly two such periodic
orbits $red$ and $green$, each of them contains two marked points (their pairing depends on the parity of $m$ and $n$), see Figure~\ref{fig:cylinders}.
As all marked points are labeled by $++,+-,-+,--$,
this gives a partition $\{red,green\}$ of $\{++,+-,-+,--\}$ into two two-element subsets, according to this correspondence. This is exactly the partition defined by \eqref{eq:green}.


{
For every $\underline{k}^{\varsigma_1\varsigma_2}(E_1)\leq k< \overline{k}^{\varsigma_1\varsigma_2}(E_1)$ let
\[D_{E,E_{1}}(x_k^{\varsigma_1\varsigma_2},{y_{k+1}^{\varsigma_1\varsigma_2}}):=\frac{1}{4}T_{1}(E_1)-\psi_1(x_k^{\varsigma_1\varsigma_2},E_1)+\frac{1}{4}T_{2}(E-E_1)-\psi_2({y_{k+1}^{\varsigma_1\varsigma_2}},E-E_1).\]
Then $D_{E,E_{1}}(x_k^{\varsigma_1\varsigma_2},{y_{k+1}^{\varsigma_1\varsigma_2}})$ is the maximal distance of the four corners of $\mathbf{R}^{\varsigma_1,\varsigma_2}$ associated with $(x_k^{\varsigma_1\varsigma_2},{y_{k+1}^{\varsigma_1\varsigma_2}})\in P^{\varsigma_1\varsigma_2}\cap R^{(E,E_1)}$ from the segment of the orbit in direction $\pi/4$ passing through the marked point $\varsigma_1\varsigma_2$  (the center of $\mathbf{R}^{\varsigma_1,\varsigma_2}$) in the time interval \((-\min\{m,n\}C\pi,\min\{m,n\}C\pi)\). Indeed, for this segment, we can consider the taxicab metric on $\R^2$, then the distance of a point $(x_1,y_1)$ from the line $\{(x_0+t,y_0+t):t\in\R\}$ is equal to
\[\min_{t}|x_1-x_0-t|+|y_1-y_{0}-t|=|x_1-x_0-(y_{1}-y_{0})|.\]
Shifting  the center of  $\mathbf{R}^{\varsigma_1\varsigma_2}$ to \((\frac{\sigma_{1}}{4}T_{1}(E_1),\frac{\sigma_{2}}{4}T_{2}(E-E_1))\) with \(\sigma_{1},\sigma_2\in\{\pm\}\), one of the corners considered is \((\sigma_{1}\psi_1(x_k^{\varsigma_1\varsigma_2},E_1),\sigma_{2}\psi_2({y_{k+1}^{\varsigma_1\varsigma_2}},E-E_1))\) and then its distance from the corresponding orbit segment is
\begin{displaymath}
\Big|(\frac{\sigma_{1}}{4}T_{1}(E_1)-\sigma_{1}\psi_1(x_k^{\varsigma_1\varsigma_2},E_1))-(\frac{\sigma_{2}}{4}T_{2}(E-E_1)-\sigma_{2}\psi_2({y_{k+1}^{\varsigma_1\varsigma_2}},E-E_1))\Big|\leq D_{E,E_{1}}(x_k^{\varsigma_1\varsigma_2},{y_{k+1}^{\varsigma_1\varsigma_2}})
\end{displaymath}
and equality is realized when \(\sigma_{1}=-\sigma_{2}\).}

\medskip

By the definition of $\delta^{colour}$, for every $colour\in\{red,green\}$ and $E_1\in I$ we have
\begin{align*}
&C\delta^{colour}(E_1):
=\max_{\substack{\varsigma_1\varsigma_2\in colour\\
\underline{k}^{\varsigma_1\varsigma_2}(E_1)\leq k< \overline{k}^{\varsigma_1\varsigma_2}(E_1)}}
C\Big(m\arccos\sqrt{\frac{V_1(x_k^{\varsigma_1\varsigma_2})}{E_1}}+n\arccos\sqrt{\frac{V_2(y_{k+1}^{\varsigma_1\varsigma_2})}{E-E_1}}\Big)\\
&=\max_{\substack{\varsigma_1\varsigma_2\in colour\\\underline{k}^{\varsigma_1\varsigma_2}(E_1)\leq k< \overline{k}^{\varsigma_1\varsigma_2}(E_1)}}
C\Big(m\Big(\frac{\pi}{2}-\arcsin\sqrt{\frac{V_1(x_k^{\varsigma_1\varsigma_2})}{E_1}}\Big)+n\Big(\frac{\pi}{2}-\arcsin\sqrt{\frac{V_2(y_{k+1}^{\varsigma_1\varsigma_2})}{E-E_1}}\Big)\Big)\\
&
=\max_{\substack{\varsigma_1\varsigma_2\in colour\\
\underline{k}^{\varsigma_1\varsigma_2}(E_1)\leq k< \overline{k}^{\varsigma_1\varsigma_2}(E_1)}}
\big(\frac{1}{4}T_{1}(E_1)-\psi_1(x_k^{\varsigma_1\varsigma_2},E_1)+\frac{1}{4}T_{2}(E-E_1)-\psi_2({y_{k+1}^{\varsigma_1\varsigma_2}},E-E_1)\big)\\
&
=\max_{\substack{\varsigma_1\varsigma_2\in colour\\
\underline{k}^{\varsigma_1\varsigma_2}(E_1)\leq k< \overline{k}^{\varsigma_1\varsigma_2}(E_1)}}D_{E,E_{1}}(x_k^{\varsigma_1\varsigma_2},{y_{k+1}^{\varsigma_1\varsigma_2}}).
\end{align*}
Therefore $C\delta^{colour}(E_1)$ measures the maximal distance of points in removed polygons $\mathbf{R}^{colour}:=\bigcup_{\varsigma_1\varsigma_2}\mathbf{R}^{\varsigma_1\varsigma_2}$
from the corresponding segments of the $colour$ periodic orbit.
Let us consider two closed cylinders $R$ and $G$ in $\T_{m,n}$ that consist of points distant from the $red$ or  $green$ periodic orbit by no more than $C\delta^{red}(E_1)$ or $C\delta^{green}(E_1)$ respectively.
Then $R$ and $G$ are two closed cylinders in direction $\pi/4$ consisting of periodic orbits passing through removed polygons whose centers lie on $red$ or $green$ periodic orbit respectively. Denote by $\partial R$ and $\partial G$ the boundary of $R$ and $G$ respectively.
Since all removed polygons are $e^{i\pi}$-invariant, both cylinders $R$ and $G$ also are.

\begin{figure}[h]
\includegraphics[width=1 \textwidth]{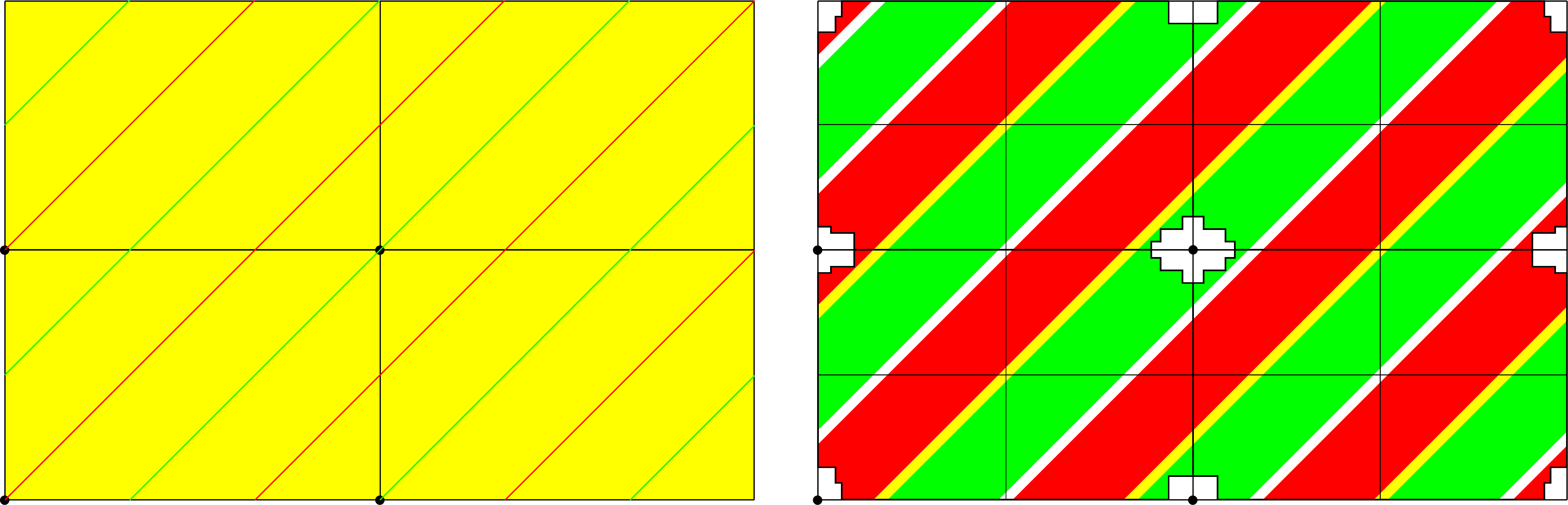}
\caption{Two periodic orbits ($red$ and $green$) on $\T_{m,n}$ and the corresponding cylinders on $M(E_1)$.}
\label{fig:cylinders}
\end{figure}

\begin{remark}\label{rem:meetbound}
Note that the distance between $red$ and $green$ periodic orbits is $\pi C$.
 Recall that \(M(E_1)=\T_{m,n}\setminus(\mathbf{R}^{green}\cup\ \mathbf{R}^{red})\), where the opposite sides of removed polygons are identified.  Now we consider two cases:

\noindent \textbf{Case 1.} If $C\delta^{red}(E_1)+C\delta^{green}(E_1)<\pi C$, namely, $E_1\in J\in\mathcal U_I^-$, then $\T_{m,n}\setminus(R\cup G)$ is non-empty and consists of two open cylinders in direction $\pi/4$, $yellow$ and $white$ on Figure~\ref{fig:cylinders}.
Since $yellow$ and $white$ cylinders  do not meet the removed polygons, they correspond to periodic cylinders $Y$ and $W$ on the translation surface $M(E_1)$. Moreover, the flow $(\psi_t^{\pi/4})_{t\in\R}$ on
$M(E_1)$ has two more invariant sets, $R\backslash\mathbf{R}^{red}$ and $G\backslash\mathbf{R}^{green}$, which are traces of the corresponding cylinders from $\T_{m,n}$. The sets $R\backslash\mathbf{R}^{red} ,G\backslash\mathbf{R}^{green}\subset M(E_1)$ are translation surfaces with  boundary and we will analyze them later.
Notice  that all $(\psi_t^{\pi/4})_{t\in\R}$-orbits on $R\backslash\mathbf{R}^{red}$ and $G\backslash\mathbf{R}^{green}$  meet the boundary of \(\mathbf{R}^{red}\) and \(\mathbf{R}^{green}\), and both $\partial (R\backslash\mathbf{R}^{red})$, $\partial (G\backslash\mathbf{R}^{green})$ consist of
saddle connections for $(\psi_t^{\pi/4})_{t\in\R}$ on $M(E_1)$.

\noindent \textbf{Case 2.} If $C\delta^{red}(E_1)+C\delta^{green}(E_1)\geq \pi C$, namely, $E_1\in J\in\mathcal U_I^+$, then $R\cup G$ fills the whole torus $\T_{m,n}$. Therefore all $(\psi_t^{\pi/4})_{t\in\R}$-orbits on $M(E_1)$ meet the boundary of either  \(\mathbf{R}^{red}\) or \(\mathbf{R}^{green}\) or both.
\end{remark}
\subsection{Periodic cylinders (case 1)}
We first note some non-degeneracy properties.
{ Recall that each removed set, \(\mathbf{R}^{colour}\), is composed of two removed polygons $\mathbf{R}_{\varsigma_1\varsigma_2}$, $\varsigma_1\varsigma_2\in colour$,  which are associated with the \(\mathbf{P}^{\varsigma_1\varsigma_2}(E_{1}), \varsigma_1\varsigma_2\in colour\).  We first note that saddle connections between different corners of one removed polygon $\mathbf{R}^{\varsigma_1\varsigma_2}$ rarely occur. We then note that only when there is a degeneracy between the polygons \(P^{\varsigma_1\varsigma_2}\) it can happen that saddle connections between the two different removed polygons persist for all \(E_{1}\in I\).} Finally we note that only rarely the periodic cylinders degenerate to a line.
\begin{lemma}\label{lem:dotyk}
{
Let $colour\in\{red,green\}$.
\begin{enumerate}
\item
For every  $\varsigma_1\varsigma_2\in colour$ and for all but countably many $E_1\in I$ we have
\[C\delta^{colour}(E_1)=D_{E,E_{1}}(x_k^{\varsigma_1\varsigma_2},{y_{k+1}^{\varsigma_1\varsigma_2}})\]
for at most one $1\leq k< k(\bar{x}^{\varsigma_1\varsigma_2},\bar{y}^{\varsigma_1\varsigma_2})$.
\item If $\varsigma_1\varsigma_2$, $\varsigma'_1\varsigma'_2$ are different elements  of a  $colour$ and
\begin{align}\label{eq:twoint}
\begin{aligned}
C\delta^{colour}(E_1)&=D_{E,E_{1}}(x_k^{\varsigma_1\varsigma_2},{y_{k+1}^{\varsigma_1\varsigma_2}})=D_{E,E_{1}}(x_l^{\varsigma_1'\varsigma_2'},{{y_{l+1}^{\varsigma_1'\varsigma_2'}}})
\end{aligned}
\end{align}
for uncountably many $E_1\in I$, then $(x_k^{\varsigma_1\varsigma_2},y_{k+1}^{\varsigma_1\varsigma_2})=(x_l^{\varsigma'_1\varsigma'_2},
y_{l+1}^{\varsigma'_1\varsigma'_2})$.
\item Moreover,
\[\delta^{red}(E_1)+\delta^{green}(E_1)\neq\pi \]
for all but countably many $E_1\in(0,E)$.
\end{enumerate}}
In addition, all three properties are met on an open set.
\end{lemma}

\begin{proof}
\textbf{(1)} Suppose, contrary to our claim, that there exist $k\neq l$  such that
\[D_{E,E_{1}}(x_k^{\varsigma_1\varsigma_2},{y_{k+1}^{\varsigma_1\varsigma_2}})=D_{E,E_{1}}(x_l^{\varsigma_1\varsigma_2},{y_{l+1}^{\varsigma_1\varsigma_2}})\]
holds for uncountably many $E_1$. Then \begin{displaymath}
\psi_1(x_k^{\varsigma_1\varsigma_2},E_1)+\psi_2({y_{k+1}^{\varsigma_1\varsigma_2}},E-E_1)=\psi_1(x_l^{\varsigma_1\varsigma_2},E_1)+\psi_2({y_{l+1}^{\varsigma_1\varsigma_2}},E-E_1)
\end{displaymath}for uncountably many $E_1$, which contradicts Proposition~\ref{prop:indep2}.

\textbf{(2)} Suppose that \eqref{eq:twoint} holds for uncountably many $E_1$. Then, similarly
\[\psi_1(x_k^{\varsigma_1\varsigma_2},E_1)+\psi_2({y_{k+1}^{\varsigma_1\varsigma_2}},E-E_1)=\psi_1(x_l^{\varsigma_1'\varsigma_2'},E_1)+\psi_2(,{{y_{l+1}^{\varsigma_1'\varsigma_2'}}},E-E_1)\]
for uncountably many $E_1$. By Proposition~\ref{prop:indep2}, it follows that $x_k^{\varsigma_1\varsigma_2}=x_l^{\varsigma'_1\varsigma'_2}$
and $y_{k+1}^{\varsigma_1\varsigma_2}=y_{l+1}^{\varsigma'_1\varsigma'_2}$.

\textbf{(3)} Now suppose, contrary to our claim, that
for uncountably many $E_1$\[C\delta^{red}(E_1)+C\delta^{green}(E_1)=D_{E,E_{1}}(x_k^{\varsigma_1\varsigma_2},{y_{k+1}^{\varsigma_1\varsigma_2}})+D_{E,E_{1}}(x_l^{\varsigma_1'\varsigma_2'},{y_{l+1}^{\varsigma_1'\varsigma_2'}})=C\pi \]
for some $\varsigma_1\varsigma_2\in red$,  $\varsigma'_1\varsigma'_2\in green$,  $1\leq k< k(\bar{x}^{\varsigma_1\varsigma_2},\bar{y}^{\varsigma_1\varsigma_2})$
and  $1\leq l< k(\bar{x}^{\varsigma'_1\varsigma'_2},\bar{y}^{\varsigma'_1\varsigma'_2})$. { Then,
since \(\frac{1}{4}T_{1}=mC\frac{\pi }{2},\frac{1}{4}T_{2}=nC\frac{\pi }{2}\), we have
\begin{gather*}
\psi_1(x_k^{\varsigma_1\varsigma_2},E_1)+\psi_2({y_{k+1}^{\varsigma_1\varsigma_2}},E-E_1)+\psi_1(x_l^{\varsigma_1'\varsigma_2'},E_1)+\psi_2(y_{l+1}^{\varsigma_1'\varsigma_2'},E-E_1)\\
=C\pi(m+n-1)=\frac{m+n-1}{2m}T_{1}(E_1)
\end{gather*}}
holds on an uncountable subset. This again contradicts Proposition~\ref{prop:indep2}.

Finally, the complement of the set of $E_1$'s for which all three properties hold is a closed set - the set of zeros for finitely many continuous maps.
\end{proof}

By Lemma~\ref{lem:dotyk}, for every $E>0$ satisfying \eqref{eq:assumE} and for every $I\in \mathcal J_E$ satisfying \eqref{eq:assumI}, there are two families $\mathcal U^+_I$ and  $\mathcal U^-_I$ of open subintervals of $I$
such that:
\begin{itemize}
\item for every $E_1 \in\bigcup_{J\in\mathcal U^+_I}J\cup\bigcup_{J\in \mathcal U^-_I}J\subset I$ all conclusions of Lemma~\ref{lem:dotyk} hold;
\item $\delta^{red}(E_1)+\delta^{green}(E_1)>\pi $ for all $E_1\in J$ if $J\in \mathcal U^+_I$;
\item $\delta^{red}(E_1)+\delta^{green}(E_1)<\pi $ for all $E_1\in J$ if $J\in \mathcal U^-_I$.
\end{itemize}
Notice that by  Lemma~\ref{lem:dotyk},  $\bigcup_{J\in\mathcal U^+_I}J\cup\bigcup_{J\in \mathcal U^-_I}J\subset I$ has at most countable complement in $I$.
The following result  describes fully the dynamics of the flow $(\psi^{\pi/4}_t)_{t\in\R}$ on $M(E_1)$ for
a.e.\ $E_1\in I$. This gives full information about dynamics of the flow $(\varphi^{P,E,E_1}_t)_{t\in\R}$ on $S^P_{E,E_1}$ for  almost every $E_1\in I$.

\begin{theorem}\label{thm:mainquad}
If $J\in \mathcal U^+_I$ then the flow $(\psi^{\pi/4}_t)_{t\in\R}$ on $M(E_1)$ is uniquely ergodic for a.e.\ $E_1\in J$.
If $J\in \mathcal U^-_I$ then the flows $(\psi^{\pi/4}_t)_{t\in\R}$ on $(R\setminus\mathbf{R}^{red})\setminus\partial(R\setminus\mathbf{R}^{red})\subset M(E_{1})$ and on $(G\setminus\mathbf{R}^{green})\setminus\partial(G\setminus\mathbf{R}^{green})$ are uniquely ergodic for a.e.\ $E_1\in J$ and are periodic on the complement of these sets with motion that does not impact the polygon \(P\).
\end{theorem}

Since the flow $(\psi^{\pi/4}_t)_{t\in\R}$ on $M(E_1)$ is isomorphic to the flow $(\varphi^{P,E,E_1}_t)_{t\in\R}$ on $S^P_{E,E_1}$,
Theorem~\ref{thm:U+-} is a direct consequence  of Theorem~\ref{thm:mainquad}.
To prove this result we need a more subtle version of Theorem 2.11 in \cite{frkaczek2019recurrence} applied to $(\psi^{\pi/4}_t)_{t\in\R}$ on some completions
$\overline{R}(E_1)$ of $(R\setminus\mathbf{R}^{red})\setminus\partial(R\setminus\mathbf{R}^{red})\subset M(E_{1})$ and $\overline{G}(E_1)$ of
$(G\setminus\mathbf{R}^{green})\setminus\partial(G\setminus\mathbf{R}^{green})$ as described below.
We formulate a bit more general construction in Sections~\ref{sec:compl}-\ref{sec:spec} and then return to the proof of Theorem~\ref{thm:mainquad} in Section~\ref{sec:proofs}.
\subsection{Construction of completion for  $J\in \mathcal U^-_I$ }\label{sec:compl}
 Suppose that $E_1\in J$ for some $J\in\mathcal U_I^-$, so $\delta^{red}(E_1)+\delta^{green}(E_1)<\pi $.
From now on we proceed only with the set $(R\setminus\mathbf{R}^{red})\setminus\partial(R\setminus\mathbf{R}^{red})$,  but the following construction of completion obviously works with   $(G\setminus\mathbf{R}^{green})\setminus\partial(G\setminus\mathbf{R}^{green})$.
Property (1) of Lemma~\ref{lem:dotyk} imply that  for  $E_{1}\in J\in \mathcal U^-_I$,   the boundary of the two removed polygons has one or two intersection points with \(\partial(R\setminus\mathbf{R}^{red})\)  in $M(E_1)$.

\begin{figure}[h]
\includegraphics[width=1 \textwidth]{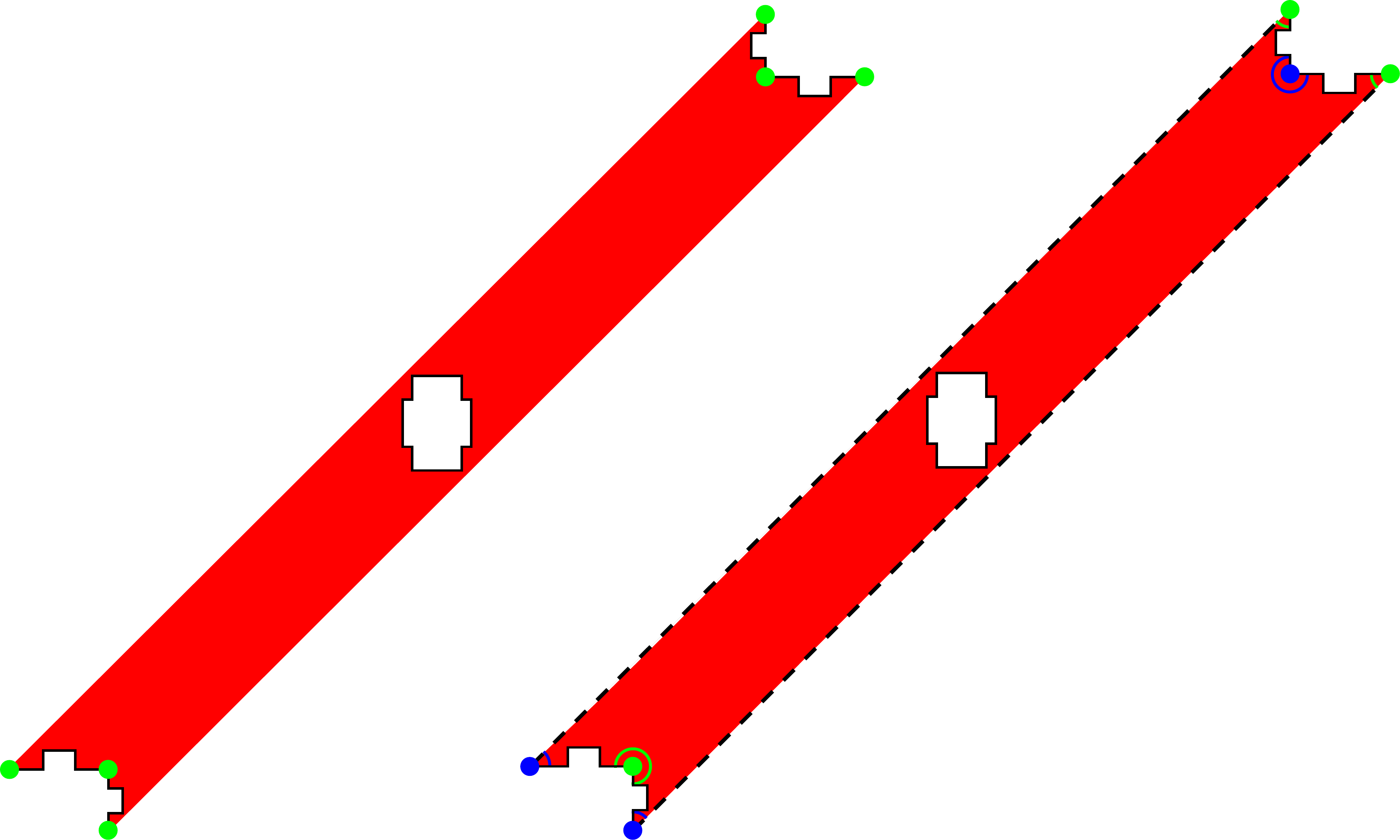}
\caption{The subset $R\setminus \partial R$ of $M(E_1)$ and its completion $\overline{R}(E_1)$.}
\label{fig:completion}
\end{figure}

We assume first that there is only one intersection, then we show that for two intersections the construction of completion
is almost the same.
Denote by $r\in M(E_1)$ the single intersection point (the green point on the left part of Figure~\ref{fig:completion}). The total angle around $r$ in $R\setminus\mathbf{R}^{red}$ is $4\pi$.
Recall that $R\setminus\mathbf{R}^{red}$ is a translation surface with the boundary and its boundary consists of two saddle connection starting and ending at the point $r$.
Both saddle connection have the same direction $\pi/4$ and the same length. In the completion $\overline{R}(E_1)$ we separate the beginning from the end of the saddle connections, their beginning we denote by $r_b$ (the blue point on the right part of Figure~\ref{fig:completion}) and their end by $r_e$ (the green point on the right part of Figure~\ref{fig:completion}),
and then we glue the intervals thus obtained (the dashed lines on the right part of Figure~\ref{fig:completion}), creating a kind of cylinder where its lower boundary is connected to the upper boundary through the inner green and inner blue corners.
The resulting object $\overline{R}(E_1)$ is a compact translation surface such that the total angle around $r_b$ and $r_e$ is $2\pi$, so they became regular points.

\begin{figure}[h]
\includegraphics[width=1 \textwidth]{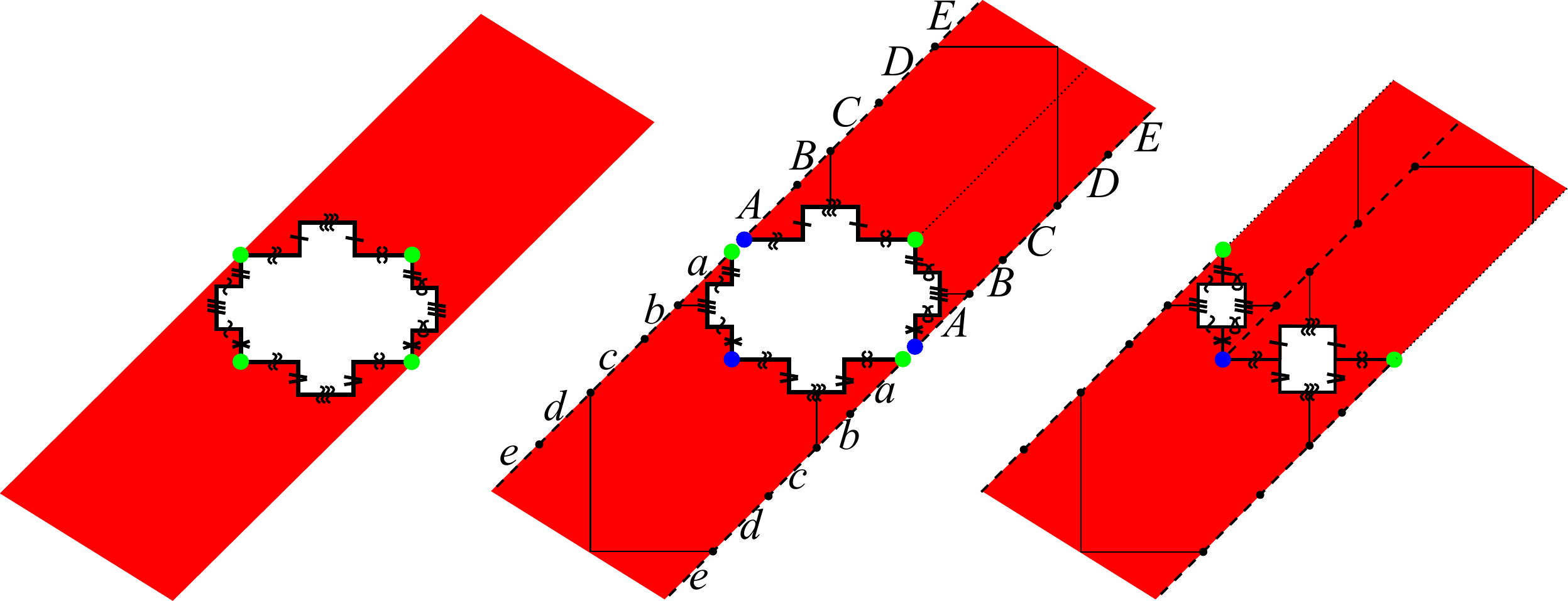}
\caption{Separation/regularization procedure}
\label{fig:completion2}
\end{figure}

If $\partial (R\setminus\mathbf{R}^{red})$ has two intersection points,  $r_1,r_{2}\in M(E_{1})$  with the removed polygons,   by property (1) of Lemma~\ref{lem:dotyk}, each one of them must belong to a distinct removed polygon.
By property (2) of Lemma~\ref{lem:dotyk}, the two points are symmetric (they have  the same critical energies \((V_{1}(x_{k}),V_{2}(y_{k+1}))\)). Thus, on the surface, they correspond to corners of the two removed polygons that are at identical horizontal and vertical distances from the removed polygon center. Hence, \(r_{2}\) is exactly at the middle point of the line connecting \(r_{1}\) to itself (namely $r_2=\psi^{\pi/4}_{mnC\pi}r_1$ and similarly  $r_1=\psi^{\pi/4}_{mnC\pi}r_2$). We carry out the separation procedure for both intersection points $r_1$ and $r_2$.
Since all four saddle connection creating $\partial (R\setminus\mathbf{R}^{red})$ have the same length, we can finalize gluing the pairs of relevant intervals on the boundary.
In this way we get rid of two singular points creating four regular points.

\begin{remark}\label{rem:hit}
Every $(\psi_t^{\pi/4})_{t\in\R}$ orbit on $\overline{R}(E_1)$ hits the boundary of a removed polygon.
\end{remark}

\subsection{Partitions of translation surfaces into polygons}\label{sec:parttranssurf}
In this section we recall some basic concepts introduced in \cite{frkaczek2019recurrence} and modify them so that the completion  surface, which has natural partition into polygons with vertical and  distinguished sides, can be treated with these tools.

\begin{definition}\label{def:part}
Let $(M,\omega)$ be a compact translation surface with singular points at \(\Sigma\). A finite partition $\mathcal{P}=\{P_\alpha:\alpha\in\mathcal A\}$ of $M$ is called a \emph{proper partition} if
\begin{itemize}
\item[$(i)$] every $P_\alpha$, $\alpha\in\mathcal A$ is closed connected subset of $M$ and  $\bigcup_{\alpha\in \mathcal A}P_\alpha=M$;
\item[$(ii)$]
for every $\alpha\in \mathcal A$ there exists a chart $\zeta_\alpha:U_\alpha\to \C$ in $\omega$ such that
\begin{itemize}
\item the interior of $P_\alpha$ is a subset of $U_\alpha$;
\item $\zeta_\alpha(\operatorname{Int}P_\alpha)$ is the interior of a compact connected polygon $\widetilde{P}_\alpha\subset \C$;
\item  $\zeta^{-1}_\alpha:\operatorname{Int}\widetilde{P}_\alpha\to\operatorname{Int}{P}_\alpha$ has a  continuous extension $\overline{{\zeta}^{-1}_\alpha}:\widetilde{P}_\alpha\to{P}_\alpha$;
\end{itemize}
then ${P}_\alpha$ is called a polygon and the $\overline{{\zeta}^{-1}_\alpha}$-image of any side/corner in $\widetilde{P}_\alpha$ is called a side/corner of ${P}_\alpha$; in particular, a side \(s\) of ${P}_\alpha$ is called vertical if  $(\overline{{\zeta}^{-1}_\alpha})^{-1}(s)$   is vertical;

\item[$(iii)$]  $\overline{{\zeta}^{-1}_\alpha}:\widetilde{P}_\alpha\setminus\{\text{vertical sides}\}\to{P}_\alpha\setminus\{\text{vertical sides}\}$
is a homeomorphism, its inverse we denote by  $\bar{\zeta}_\alpha:{P}_\alpha\setminus\{\text{vertical sides}\}\to\widetilde{P}_\alpha\setminus\{\text{vertical sides}\}$;

\item[$(iv)$] if $P_\alpha\cap P_\beta\neq \emptyset$ then it is the union of common sides and corners of the polygons $P_\alpha$, $P_\beta$;
\item[$(v)$] if $\sigma\in P_\alpha\cap \Sigma$ then $\sigma$ is a corner of $P_\alpha$;
\item[$(vi)$] if a common side of the polygons $P_\alpha$, $P_\beta$ is vertical then its ends are regular points in $(M,\omega)$;\end{itemize}
\end{definition}
Notice that we modify the corresponding definition in  \cite{frkaczek2019recurrence} to allow polygons with vertical boundaries. Non vertical sides always connect different tiles, whereas vertical sides can glue within a single tile or with other tiles having vertical boundaries.
 When vertical sides of a single polygon are glued, this polygon becomes cylindrical under the completion procedure, so the function \(  \overline{{\zeta}^{-1}_\alpha}\) is not one to one on this polygon vertical boundary, thus in condition \((iii)\) the vertical boundaries are excluded. This condition is needed since in some cases vertical boundaries of different polygons are glued together (see Figure~\ref{fig:completion2}). Condition $(v)$ and  $(vi)$  are  imposed for convenience so that the singularities in \(P_\alpha\)  do not belong to  glued vertical boundaries.  \begin{definition}[\cite{frkaczek2019recurrence}]
 For any proper partition  $\mathcal{P}=\{P_\alpha:\alpha\in\mathcal A\}$ of the translation surface $(M,\omega)$ let:
\begin{itemize}
\item ${\mathbf{D}}={D}(\omega,\mathcal P)$ be the set of all non-vertical sides in $\mathcal P$;
\item $\mathbf{B}=B(\omega,\mathcal P)$ the set of pairs $(\sigma,\beta)\in\Sigma\times \mathcal A$ for which $\sigma\in P_\beta$ and
there exits a vertical (upward) orbit segment in $P_\beta$ which begins in $\sigma$;
\item $\mathbf{E}=E(\omega,\mathcal P)$ the set of pairs $(\sigma,\alpha)\in\Sigma\times \mathcal A$ for which $\sigma\in P_\alpha$ and
there exits a vertical (upward) orbit segment in $P_\alpha$ which ends in $\sigma$.
\end{itemize}

The sets \(\mathbf{D},\mathbf{B},\mathbf{E}\) code some topological properties of the partition.
\begin{itemize}
\item Let $s_{\alpha\beta}\in \mathbf{D}$ be a common  side of $P_\alpha$ and $P_\beta$ and suppose that every vertical (upward) orbit through the side $s_{\alpha\beta}$ passes from $P_\alpha$ to $P_\beta$.
Then the displacement $\mathfrak{D}_{\omega}(s_{\alpha\beta}):=\bar{\zeta}_\alpha(x)-\bar{\zeta}_\beta(x)$ does not depend on the choice of $x\in s_{\alpha\beta}$.
\item For every $(\sigma,\beta)\in \mathbf{B}$ let $\mathfrak{B}_{\omega}(\sigma,\beta)=- \bar{\zeta}_\beta(\sigma)$.
\item For every $(\sigma,\alpha)\in \mathbf{E}$ let $\mathfrak{E}_{\omega}(\sigma,\alpha)= \bar{\zeta}_\alpha(\sigma)$.
\end{itemize}
\end{definition}
Notice that the corners  arising from the regularized points of the completion surface are not in \(\mathbf{B}\) nor in \(\mathbf{E}\).

\begin{definition}\label{def:curve1}
Let $\{(M,\omega_{E_1})\}_{E_1\in J}$ be a family of translation surfaces.
We call $J\ni E_1\mapsto ((M,\omega_{E_1}),\mathcal P(E_1))$  a \emph{$C^\infty$-curve of translation surfaces equipped with proper partition} if
there is a finite open cover $(U_\alpha)_{\alpha\in\mathcal A}$ of $M\setminus\Sigma$ such that:
\begin{itemize}
\item $\mathcal P(E_1)=\{P_\alpha(E_1):\alpha\in \mathcal A\}$ is a proper partition of $(M,\omega_{E_1})$ into polygons for every $E_1\in J$;
\item $\operatorname{Int}P_{\alpha}(E_1)\subset U_\alpha$ for every $\alpha\in \mathcal A$ and $E_1\in J$;
\item for every $\alpha\in \mathcal A$ the polygons $\widetilde{P}_\alpha(E_1)$, $E_1\in J$ are diffeomorphic  and the coordinates of their sides/corners
vary $C^\infty$-smoothly with $E_1\in J$;
\item if $J\ni E_1\mapsto s_{\alpha\beta}(E_1)\subset P_\alpha(E_1)\cap P_\beta(E_1)$ is a $C^{\infty}$-curve of common sides and $s_{\alpha\beta}(E_{10})$ is vertical for some $E_{10}\in J$, then $s_{\alpha\beta}(E_1)$ is vertical for all $E_1\in J$.
\end{itemize}
\end{definition}
For every  $C^\infty$-curve $\mathcal{J}:E_1\in J\mapsto ((M,\omega_{E_1}),\mathcal P(E_1))$ the topological data of the partitions, in particular the sets $D(\omega_{E_1},\mathcal P(E_1))$, $B(\omega_{E_1},\mathcal P(E_1))$, $E(\omega_{E_1},\mathcal P(E_1))$, do not depend on $E_1\in J$. Therefore, we  write $\mathbf{D}$, $\mathbf{B}$ and $\mathbf{E}$ for short, for all \(E_1\in J\).

Let us distinguish a non-empty subset of sides $\mathbf{D}^*\subset \mathbf{D}$. They correspond to sides with impacts.
We will deal with four finite families in $C^\infty(J,\C)$:
\begin{gather*}
\mathscr D :=\{\mathfrak{D}_{\omega_{E_1}}(s_{\alpha\beta}):s_{\alpha\beta}\in \mathbf{D}\},\quad
\mathscr D^* :=\{\mathfrak{D}_{\omega_{E_1}}(s_{\alpha\beta}):s_{\alpha\beta}\in \mathbf{D}^*\},\\
\mathscr B :=\{\mathfrak{B}_{\omega_{E_1}}(\sigma,\alpha):(\sigma,\alpha)\in \mathbf{B}\},\quad
\mathscr E :=\{\mathfrak{E}_{\omega_{E_1}}(\sigma,\alpha):(\sigma,\alpha)\in \mathbf{E}\}.
\end{gather*}

\begin{theorem}\label{thm:gencrit}
Suppose that $\mathcal{J}:J\ni E_1\mapsto ((M,\omega_{E_1}),\mathcal P(E_1))$ is a $C^\infty$-curve of translation surfaces equipped with proper partition and that there exists a distinguished set $\mathbf{D}^*$ such that:
\begin{itemize}
\item[$(*)$] for every $E_1\in J$ every vertical orbit in $(M,\omega_{E_1})$ hits at least one side in $\mathbf{D}^*$.
\end{itemize}
Let $\ell:J\to\R_{>0}$ be a $C^\infty$-map.
Assume that  \((\mathcal{J},\mathbf{D}^*,\ell) \) satisfy:
\begin{itemize}

\item[$(i)$]  Consider any sequence $(n_h)_{h\in \mathscr D}$  in $\Z_{\geq 0}$ such that $n_h>0$ for some $h\in \mathscr D^*$. Assume that for any $f\in \mathscr B$, $g\in \mathscr E$
  such that the map $f+g+\sum_{h\in \mathscr D}n_hh$ is non-zero, we have
        \[\rp f(E_1)+\rp g(E_1)+\sum_{h\in \mathscr D} n_h\rp h(E_1)\neq 0\ \text{ for a.e. }\ E_1\in J;\]
\item[$(ii_+)$] $[\rp h,\ell](E_1)\geq 0$ for all $h\in  \mathscr D$ and $E_1\in J$  with
\[\sum_{h\in  \mathscr D}[\rp h,\ell](E_1)> 0 \quad\text{for a.e. $E_1\in J$, or}\]
\item[$(ii_-)$] $[\rp h,\ell](E_1)\leq 0$ for all $h\in  \mathscr D$ and $E_1\in J$  with
\[\sum_{h\in  \mathscr D}[\rp h,\ell](E_1)< 0\quad\text{for a.e.}\ E_1\in J.\]
\end{itemize}
Then the vertical flow $(\psi^v_t)_{t\in\R}$ on $(M,\omega_{E_1})$ is uniquely ergodic for a.e.\ $E_1\in J$.
\end{theorem}

\begin{proof}
Note that Theorem~\ref{thm:gencrit} is a more general version of Theorem~2.11 in \cite{frkaczek2019recurrence} and their proofs are similar as explained next.
The difference between them is that Theorem~2.11 in \cite{frkaczek2019recurrence} prohibits the existence of vertical sides in $\mathcal{P}(E_1)$
and assumes that $\mathbf{D}^*=\mathbf{D}$, i.e.\ all sides are distinguished. Then the assumption $(*)$ about hitting the sides is obviously fulfilled.

 The proof of  Theorem~2.11 in \cite{frkaczek2019recurrence} consists of two parts:
\begin{itemize}
\item[$(I)$] showing that the condition $(i)$ (together with the two restrictions listed above) implies the absence of vertical saddle connections in $(M,\omega_{E_{1}})$ for a.e.\ $E_1\in J$;
\item[$(II)$] showing that the assumption $(ii_\pm)$ implies that the corresponding piecewise constant function defined in Minsky-Weiss \cite{MiWe2014}, $L_{E_1}:I\to\R$, for the return map to any horizontal section, takes
non-negative/non-positive values with at least one positive/negative value (cf.\ Theorem~2.4 in \cite{frkaczek2019recurrence}, i.e.\ a copy of  Theorem~6.2 in \cite{MiWe2014}) for a.e.\ $E_1\in J$.
\end{itemize}
In view of Corollary~2.6 in \cite{frkaczek2019recurrence} (a direct consequence of Theorem~6.2 in \cite{MiWe2014}), both properties (the absence of saddle connections and some positivity/negativity of $L_{E_1}$) give the unique ergodicity of $(\psi^v_t)_{t\in\R}$ on $(M,\omega_{E_1})$ for a.e.\ $E_1\in J$.

We follow the same line. As the current part $(II)$ of the proof does not differ from the corresponding part in the proof of Theorem~2.11 in \cite{frkaczek2019recurrence},
we focus only on showing the current assumption $(i)$ implies the absence of vertical saddle connections.

\medskip

If $\gamma$ is a vertical saddle connection in $(M,\omega_{E_1})$  then
\begin{equation*}
\langle \omega_{E_1},\gamma\rangle=\int_\gamma\omega_{E_1}=i\tau_\gamma,
\end{equation*}
where $\tau_\gamma>0$ is the length of $\gamma$.  Indeed,
  \(\omega_{E_{1}}\) is the holomorphic one-form (Abelian differential) on $(M,\omega_{E_1})$ which is given by \(dz \) in the local coordinates on $M\setminus\Sigma$. By Theorem 2.8 in \cite{frkaczek2019recurrence}, if $\gamma$ is a  piecewise linear curve with ends at $\Sigma$, $\gamma$ avoids  vertical sides and corners of $\mathcal{P}(E_1)$, and passes  upward through the non-vertical sides of  $\mathcal{P}(E_1)$, then
\begin{equation*}
\left\langle\omega_{E_{1}},\gamma \right\rangle =f_\gamma(E_1)+g_{\gamma}(E_1)+\sum_{s_{\alpha\beta}\in\mathbf{D}}n_{s_{\alpha\beta},\gamma}h_{s_{\alpha\beta}}(E_1),
\end{equation*}
with
\begin{itemize}
\item $f_\gamma(E_1)=\mathfrak{B}_{\omega_{E_1}}(\sigma_+,\alpha)$, where $\sigma_+\in\Sigma\cap P_\alpha(E_1)$ is the beginning of $\gamma$, so $f_{\gamma}\in \mathscr B$;
\item $g_\gamma(E_1)=\mathfrak{E}_{\omega_{E_1}}(\sigma_-,\beta)$, where $\sigma_-\in\Sigma\cap P_\beta(E_1)$ is the end of $\gamma$, so $g_{\gamma}\in \mathscr E$;
\item for every $h_{s_{\alpha\beta}}\in \mathscr D$ of the form $h_{s_{\alpha\beta}}(E_1)=\mathfrak{D}_{\omega_{E_1}}(s_{\alpha\beta})$ for some $s_{\alpha\beta}\in \mathbf{D}$, $n_{s_{\alpha\beta},\gamma}$ is the crossing number - the number of hits of the side $s_{\alpha\beta}$ by the curve $\gamma$.
\end{itemize}
Clearly, if \(\gamma',\gamma''\) are homologous in $H_1(M,\Sigma,\Z)$, then \(\left\langle\omega_{E_{1}},\gamma' \right\rangle =\left\langle\omega_{E_{1}},\gamma'' \right\rangle \).

For any vertical saddle connection that avoids vertical sides and corners of $\mathcal{P}(E_1)$, we obtain immediately that
\begin{equation}\label{eq:neq0}
i\tau_\gamma=\langle\omega_{E_{1}},\gamma\rangle=f_{\gamma}(E_1)+g_{\gamma}(E_1)+\sum_{s_{\alpha\beta}\in\mathbf{D}}n_{s_{\alpha\beta},\gamma}h_{s_{\alpha\beta}}(E_1).
\end{equation}
Moreover, by  assumption $(*)$, such a connection always crosses at least once one distinguished side, \(s_{\alpha\beta}\in\mathbf{D}^{*}\), so the crossing number is positive: $n_{s_{\alpha\beta},\gamma}>0$.

We show that similar formula holds  even when the saddle connection $\gamma$ runs along vertical sides or meets some corner points of $\mathcal{P}(E_1)$. For such connections, the crossing numbers through vertical sides and corners are ill defined, yet, as we show next, small deformations of \(\gamma\) provide the same value for the saddle connection length. Indeed, by assumption (\(vi\)) of Definition \ref{def:part}, we can find a piecewise linear curve $\gamma'$ homologous with $\gamma$ (in $H_1(M,\Sigma,\Z)$) such that $\gamma'$ does not meet any vertical side nor corner of $\mathcal{P}(E_1)$, $\gamma'$ passes  upward though the non-vertical sides of $\mathcal{P}(E_1)$,  and $\gamma'$ is very close to $\gamma$. Then, for any such \(\gamma'\) we have
\begin{equation}\label{eq:neq00}
i\tau_{\gamma}=\langle\omega_{E_{1}},\gamma\rangle=\langle\omega_{E_{1}},\gamma'\rangle=f_{\gamma'}(E_1)+g_{\gamma'}(E_1)+\sum_{s_{\alpha\beta}\in\mathbf{D}}n_{s_{\alpha\beta},\gamma'}h_{s_{\alpha\beta}}(E_1).
\end{equation}
Again, by  assumption $(*)$, there exists at least one side \(s_{\alpha\beta}\in\mathbf{D}^{*}\)  such that  \(\gamma\) crosses this side.
Hence,  since  $\gamma'$ is close enough to $\gamma$, $n_{s_{\alpha\beta},\gamma'}>0$.
\medskip

Summarizing, in view of \eqref{eq:neq0} and \eqref{eq:neq00}, for any vertical saddle connection in $(M,\omega_{E_1})$ there exist $f\in\mathscr{B}$,
$g\in\mathscr{E}$ and a sequence $(n_h)_{h\in\mathscr{D}}$ in $\Z_{\geq 0}$ such that $n_h>0$ for some $h\in\mathscr{D}^*$ and
\begin{equation}\label{eq:neq000}
i\tau_\gamma =f(E_1)+ g(E_1)+\sum_{h\in \mathscr D} n_h h(E_1).
\end{equation}
Hence, if a vertical saddle connection exists at \(E_{1}\), by assumption $(i)$, we have
\[\rp f(E'_1)+\rp g(E'_1)+\sum_{h\in \mathscr D} n_h\rp h(E'_1)\neq 0\ \text{ for a.e. }\ E'_1\in J.\]
On the other hand, by \eqref{eq:neq000}, we have
\[
0=\rp f(E_1)+\rp g(E_1)+\sum_{h\in \mathscr D}n_h\rp h(E_1).
\]
This gives the absence of vertical saddle connections for a.e.\ $E_1\in J$, which completes the proof.
\end{proof}

The above theorem proves unique ergodicity for almost all \(E_{1}\) once the triplet  \((\mathcal{J},\mathbf{D}^*,\ell) \) satisfying assumptions \((*),(i),(ii)\) are found.

 \subsection{Construction of the  \((\mathcal{J},\mathbf{D}^*,\ell) \) triplet for resonant staircase  dynamics  }\label{sec:spec}
We construct a proper partition to the $(\psi^{\pi/4}_t)_{t\in\R}$ flow on \(\hat M(E_1)\), where \(\hat M(E_1)\) is either  $M(E_1)$ for $J\in \mathcal U^+_I$   or on the completions
$\overline{R}(E_1)$ and $\overline{G}(E_1)$  for  $J\in \mathcal U^-_I$.

The set of distinguished sides, \(\mathbf{D}^*\), is naturally defined by the sides of the removed polygons \(\mathbf{R}^{colour}\).  Assumption \((*)\) follows from the division of the \(E_{1}\) interval to the segments $J\in \mathcal U^{\pm}_I$ and the completion construction.
The function \(\ell\) is taken to be constant as in previous sections.
The main challenge is to compute the functional form of \(\mathcal{J}\). We first construct a tiling and show that it induces specific computable rules that satisfy assumptions \((i)\) and \((ii)\).

Recall that  \( M(E_1)\) is composed of $16$ tiles, each of these corresponds to a staircase polygon
$\mathbf{P}(E_1)^{\varsigma_1\varsigma_2}_{\sigma_1\sigma_2}= P(\sigma_1\overline{x}^{\varsigma_1\varsigma_2}(E_1),\sigma_2\overline{y}^{\varsigma_1\varsigma_2}(E_1))$ (see Figure~\ref{fig:surface}). For  $J\in \mathcal U^+_I$   these tiles will be used as the proper partition of   \( M(E_1)\). The translation surfaces $\overline{R}(E_1)$ and $\overline{G}(E_1)$ have natural partition  which arises as the intersection of tiles of $M(E_1)$
with the strip $R\setminus\mathbf{R}^{red}$ (or $G\setminus\mathbf{R}^{green}$) (see Figure~\ref{fig:cylinders}).
To calculate  \(\mathcal{J}\), we divide these tiles  to those having distinguished sides and are called distinguished and to those not having distinguished sides, called non-distinguished  tiles. To ease notation, we omit the dependence on \(\varsigma_1,\varsigma_2\) and on \(E_{1}\) when they are inessential.

More precisely, we define these tiles so that they satisfy the following properties of basic polygons:

\begin{definition}\label{def:BP}
The class of \emph{basic polygons}, $BP$, consists of polygons  of the form $P(\sigma_1\overline{x},\sigma_2\overline{y})\cap B$.  $P(\sigma_1\overline{x},\sigma_2\overline{y})$ is a staircase polygon, with, possibly, distinguished sides corresponding to all sides of concave corners (having one end  of the form $(\sigma_1 x_l,\sigma_2 y_{l+1})$ for $1\leq l<k(\overline{x},\overline{y})$, green dashed sides on Figure~\ref{fig:basicpolygonstrip}).  $B\subset \R^2$ is a stripe in direction $\pi/4$. The stripe \(B\) is chosen so that the polygons in \(BP\) satisfy the below properties:
\begin{enumerate}
\item
 Basic polygons containing distinguished sides are called  \textit{distinguished} and are  said to belong to $BP^*$.
\item The sides in direction $\pi/4$  of distinguished  basic polygons are either  disjoint from the distinguished chain or touch it in a single concave corner (without crossing the chain), see Figure~\ref{fig:basicpolygonstrip}.
In the latter case, the distinguished basic polygon
is split to two smaller distinguished basic polygon, called \(SBP^*\), so that each of them is connected when this point is removed; namely, if $P(\sigma\overline{x},-\sigma\overline{y})\cap B\in  BP^*$ and the boundary of $B$ intersect $P(\sigma\overline{x},-\sigma\overline{y})$ at a corner $(\sigma x_l,-\sigma y_{l+1})$, see Figure~\ref{fig:basicpolygonstrip1}, this corner point breaks the basic polygon $P(\sigma\overline{x},-\sigma\overline{y})\cap B$ into two smaller distinguished basic polygons in  \(SBP^*\subset BP^*\) as in Figure~\ref{fig:basicpolygonstrip1}. Each of these smaller polygon is of the form
$P(\sigma\overline{x},-\sigma\overline{y})\cap B^{\pm}$, where  $B^{\pm}$ is a half-strip.
\item
The sides in direction $\pi/4$ may have additional artificial corners (black dots on Figure~\ref{fig:basicpolygonstrip}) that split such sides into smaller
sides. Each smaller side glues to a unique  basic polygon, see Figure~\ref{fig:completion2}.
\item On every basic polygon $P(\sigma_1\overline{x},\sigma_2\overline{y})\cap B$ ($B$ is a strip or half-strip) we deal always with local coordinates inherited from local coordinates on $P(\sigma_1\overline{x},\sigma_2\overline{y})$. We call $x_{k(\overline{x},\overline{y})}$ the \emph{formal width} of the basic polygon $P(\sigma_1\overline{x},\sigma_2\overline{y})\cap B$ and $y_{1}$ its \emph{formal height}.
\end{enumerate}
\end{definition}

\begin{figure}[h]
\includegraphics[width=0.6 \textwidth]{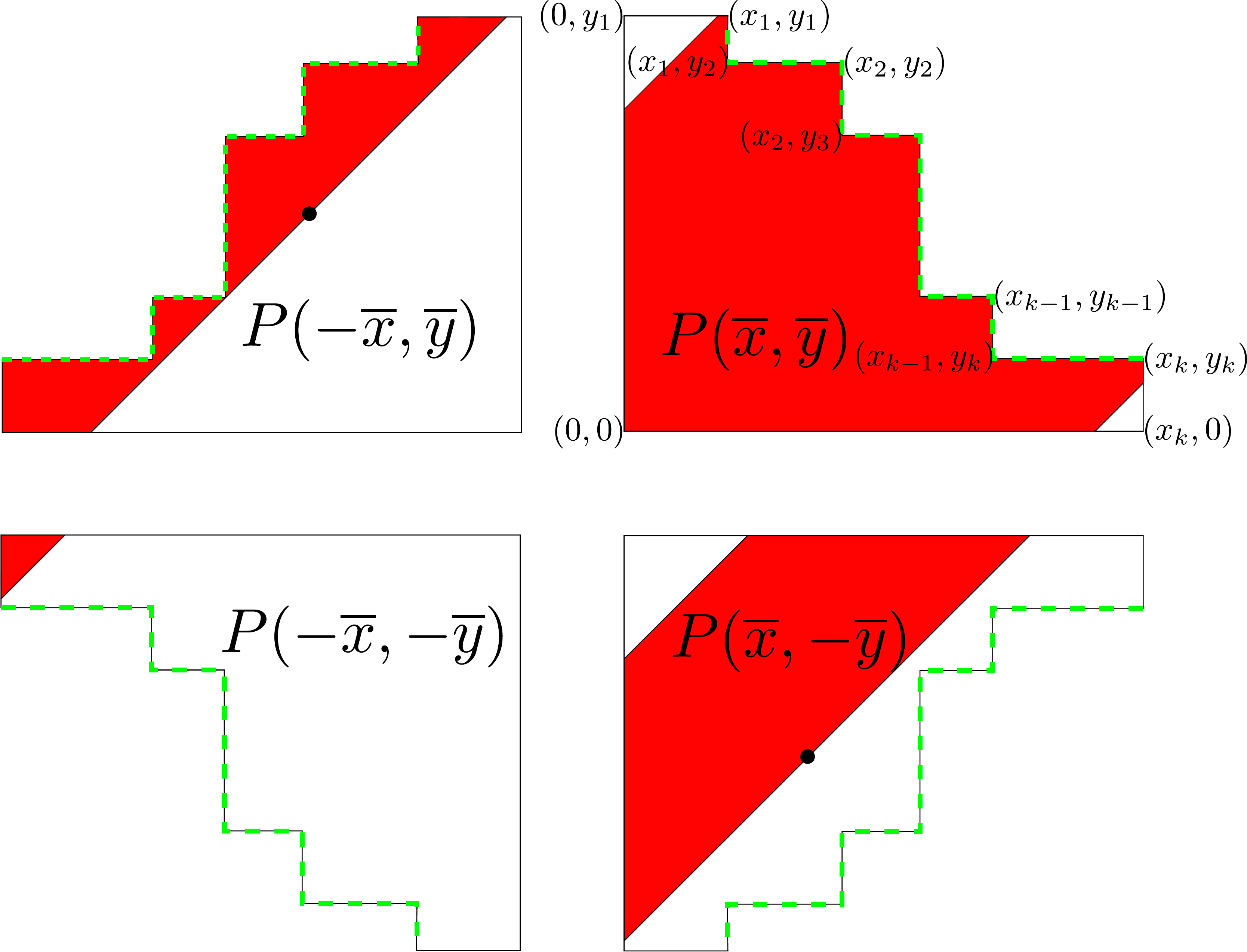}
\caption{New basic polygons.}\label{fig:basicpolygonstrip}
\end{figure}

Notice that by properties (1) and (2) the  sides in direction $\pi/4$ do not cross the distinguished chain.
  By item (1) of Lemma  \ref{lem:dotyk} the above constructed tiles of the surfaces  $\overline{R}(E_1)$ and $\overline{G}(E_1)$ satisfy this property and the above definition so they are basic polygons.

We will deal with a family $\mathcal S$ of compact translation surfaces equipped with partitions
into basic polygons described above.
\begin{figure}[h]
\includegraphics[width=0.8 \textwidth]{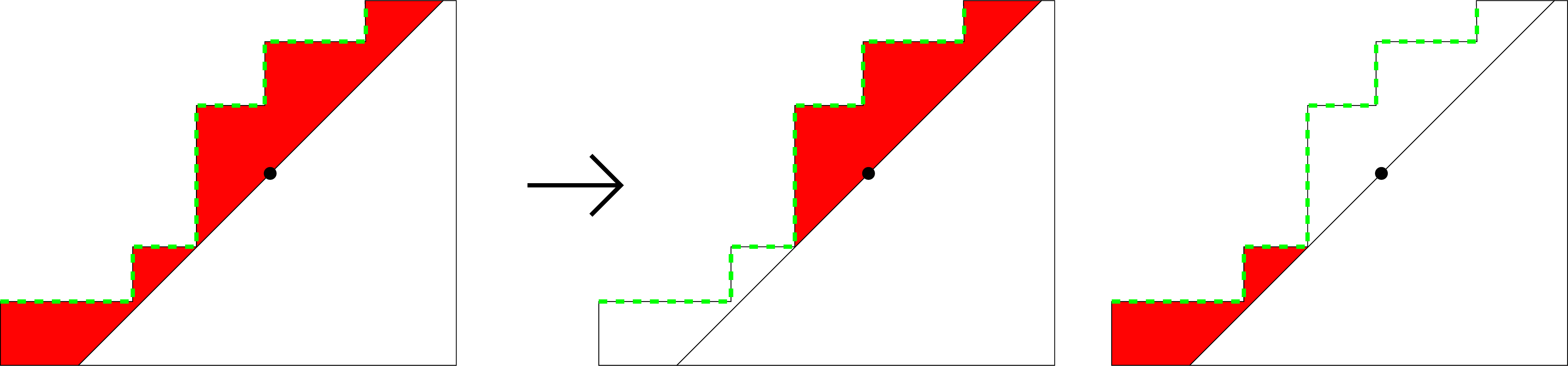}
\caption{Breaking procedure.}\label{fig:basicpolygonstrip1}
\end{figure}

\begin{definition}\label{def:mathcalS}
We say that $((M,\omega),\mathcal{P})$ belongs to $\mathcal S$ if:

\begin{itemize}
\item[$(i)$] $(M,\omega)$ is a compact translation surface;
\item[$(ii)$] $\mathcal{P}=(P_\alpha)_{\alpha\in \mathcal A}$ is partition of $M$ such that the rotated partition
$e^{i\pi/4}\mathcal{P}$ is a partition a polygons of the rotated translation surface $(M,e^{i\pi/4}\omega)$ in the sense
of Definition~\ref{def:part} (with  $e^{i\pi/4}\widetilde{P}_\alpha$ tiles);
\item[$(iii)$] for every $\alpha\in \mathcal A$ the corresponding polygon $\widetilde{P}_\alpha\in BP$. Denote by $\mathcal{A}^*\subset\mathcal{A}$
the set of  $\alpha\in \mathcal A$ such that $\widetilde{P}_\alpha\in BP^*$. The polygons $P_\alpha\in\mathcal{P}$ for $\alpha\in \mathcal A^*$
are called distinguished;
\item[$(iv)$] the set $\mathcal{A}^*$ (or equivalently, the subset of distinguished polygons in the partition $\mathcal{P}$)
is divided into four-tuples or six-tuples as follows:
\begin{itemize}
\item if $P_{\alpha_{++}}$, $P_{\alpha_{+-}}$, $P_{\alpha_{-+}}$, $P_{\alpha_{--}}$ form a four-tuple, then there exists $(\bar{x},\bar{y})\in\Xi$ such that $\widetilde P_{\alpha_{\sigma_{1}\sigma_{2}}}=P(\sigma_{1}\overline{x},\sigma_{2}\overline{y})\cap B_{\sigma_{1},\sigma_{2}}\in BP^*$, where
$B_{\sigma_{1},\sigma_{2}}$ is a strip  for $\sigma_1,\sigma_2\in\{\pm\}$;
\item if $P_{\alpha_{++}}$, $P_{\alpha_{+-}^+}$,  $P_{\alpha_{+-}^-}$, $P_{\alpha_{-+}^+}$, $P_{\alpha_{-+}^-}$, $P_{\alpha_{--}}$ form a six-tuple, then there exists $(\bar{x},\bar{y})\in\Xi$ such that $\widetilde P_{\alpha_{\sigma\sigma}}=P(\sigma\overline{x},\sigma\overline{y})\cap B_{\sigma,\sigma}\in BP^*$, where $B_{\sigma,\sigma}$ is a strip  for $\sigma\in\{\pm\}$, and
$\widetilde P_{\alpha_{\sigma,-\sigma}^{\varsigma}}=P(\sigma\overline{x},-\sigma\overline{y})\cap B^{\varsigma}_{\sigma,-\sigma}\in SBP^*$, where $B_{\sigma,-\sigma}^{\varsigma}$ is a half-strip  for $\varsigma,\sigma\in\{\pm\}$, the \(SBP^*\) polygons
$\widetilde P_{\alpha_{\sigma,-\sigma}^{+}}$, $\widetilde P_{\alpha_{\sigma,-\sigma}^{-}}$ arise from a polygon $P(\sigma\overline{x},-\sigma\overline{y})\cap B_{\sigma,-\sigma}\in BP^*$ in the breaking procedure described in item (2) of Definition~\ref{def:BP}, i.e.\ the strip $B_{\sigma,-\sigma}$ splits into two half-strips $B_{\sigma,-\sigma}^+$ and $B_{\sigma,-\sigma}^-$.  For convenience (see \((v)\)) we denote $P_{\alpha_{\sigma,-\sigma}}:=P_{\alpha_{\sigma,-\sigma}^{+}}\cup P_{\alpha_{\sigma,-\sigma}^{-}}$ for $\sigma\in\{\pm\}$, yet, note that this is a slight\ abuse of notation as in this case $\alpha_{+-}$ and $\alpha_{-+}$ do not belong to $\mathcal{A}$;

\end{itemize}
\item[$(v)$]  the distinguished vertical sides of $P_{\alpha_{++}}$ ($P_{\alpha_{+-}}$ resp.) are glued with the corresponding sides of  $P_{\alpha_{-+}}$ ($P_{\alpha_{--}}$ resp.)
and the distinguished horizontal sides of $P_{\alpha_{++}}$ ($P_{\alpha_{-+}}$ resp.) are glued with the corresponding sides of  $P_{\alpha_{+-}}$ ($P_{\alpha_{--}}$ resp.);
\item[$(vi)$] all singular points in $(M,\omega)$ come only from the concave corners of polygons in the partition $\mathcal{P}$ (by definition,
such polygons have to belong to $BP^*$);
\item[$(vii)$] the formal widths of all basic polygons in $\mathcal P$ are identical, herafter denoted by  $w>0$. Similarly,
their formal heights are identical and denoted by $h>0$.
\item[$(viii)$]
Note that for every vertical/horizontal common side $s_{\alpha\beta}$ of $P_\alpha$ and $P_\beta$ such that every  orbit through the side $s_{\alpha\beta}$ in direction $\pi/4$ passes from $P_\alpha$ to $P_\beta$, the displacement (of local coordinates) for $s_{\alpha\beta}$ is given by the difference of the local charts
 $\mathfrak{D}_{\omega}(s_{\alpha\beta})=\bar{\zeta}_\alpha(x)-\bar{\zeta}_\beta(x)$ for any $x\in s_{\alpha\beta}$ (by the flat translation structure it does not depend on the choice of $x\in s_{\alpha\beta}$).
We  assume that
 if $s_{\alpha\beta}$ is a non-distinguished vertical side in $\mathcal P$ then $\mathfrak{D}_{\omega}(s_{\alpha\beta})$
is $0$ or $2w$ and if $s_{\alpha\beta}$ is a non-distinguished  horizontal side  in $\mathcal P$ then $\mathfrak{D}_{\omega}(s_{\alpha\beta})$
is $0$ or $2hi$.
\end{itemize}
\end{definition}

\begin{remark}\label{rem:displ}
Notice that, in view of condition $(v)$, if $s_{\alpha\beta}$ is a distinguished vertical side in $\mathcal P$ (so $s_{\alpha\beta}$ comes from a distinguished side of $P(\sigma_1\overline{x},\sigma_2\overline{y})\cap B\in BP^*$) then $\mathfrak{D}_{\omega}(s_{\alpha\beta})=2 x_l$ for some $1\leq l<k(\bar{x},\bar{y})$.
Similarly, if $s_{\alpha\beta}$ is a distinguished horizontal side in $\mathcal P$ then  $\mathfrak{D}_{\omega}(s_{\alpha\beta})=2 y_li$ for some $1< l\leq k(\bar{x},\bar{y})$.

With the above remark, we see that the set \(\mathfrak{D}_{\omega}\) of all  vertical and horizontal sides in $\mathcal{P}$ is fully defined by items \((v)\) and  \((viii)\).
Next we examine the form of \(\mathfrak{B}_{\omega}\) and  \(\mathfrak{E}_{\omega}\).

Suppose that $\sigma\in P_\beta$ is a singularity in $(M,\omega)$ and $\sigma$ is the beginning of an orbit segment in direction $\pi/4$
contained in $P_\beta$. By $(vi)$, $\widetilde P_\beta=P(\sigma_1\overline{x},\sigma_2\overline{y})\cap B_{\sigma_1,\sigma_2}\in BP^*$  and
$\bar{\zeta}_\beta(\sigma)=\sigma_1x_l+i\sigma_2y_{l+1}$ for some  $1\leq l<k(\overline{x},\overline{y})$. However, the point
$x_l+iy_{l+1}\in P(\overline{x},\overline{y})\cap B_{++}$ can not be the starting point of any orbit segment in direction $\pi/4$
in the polygon $P(\overline{x},\overline{y})\cap B_{++}$. It follow that at least one $\sigma_1$ or $
\sigma_2$ is $-$, so \(\mathfrak{B}_{\omega}\), which consists of  the values \(-\bar{\zeta}_\beta(\sigma)\) is of the form \(x_l+iy_{l+1},x_l-iy_{l+1},-x_l+iy_{l+1}\).

Similar argument shows that if $\sigma\in P_\alpha$ is a singularity in $(M,\omega)$ and $\sigma$ is the end of an orbit segment in direction $\pi/4$
contained in $P_\alpha$, then $\widetilde P_\alpha=P(\sigma_1\overline{x},\sigma_2\overline{y})\cap B_{\sigma_1,\sigma_2}\in BP^*$,
$\bar{\zeta}_\alpha(\sigma)=\sigma_1x_l+i\sigma_2y_{l+1}$ for some  $1\leq l<k(\overline{x},\overline{y})$ and at least one $\sigma_1$ or $
\sigma_2$ is $+$,
so \(\mathfrak{E}_{\omega}\) which consists of  the values \(\bar{\zeta}_\alpha(\sigma)\), is  of the form \(x_l+iy_{l+1},x_l-iy_{l+1},-x_l+iy_{l+1}\).\end{remark}

Next we examine how the numerical data appear in the rotated partition.
Suppose that $((M,\omega),\mathcal P)\in\mathcal S$ and let us consider the rotated surface $(M,e^{\pi i/4}\omega)$.
The flow $(\psi_t^{\pi/4})_{t\in\R}$ on $(M,\omega)$ is equivalent to the vertical flow on $(M,e^{\pi i/4}\omega)$
and the rotated partition $e^{\pi i/4}\mathcal P=(e^{\pi i/4}P_{\alpha})_{\alpha\in\mathcal A}$ (i.e.\ all local coordinated are rotated by $\pi/4$)
is a proper partition of $(M,e^{\pi i/4}\omega)$ into polygons in the sense of Definition~\ref{def:part}. Denote by $\mathbf D^*=D^*(e^{\pi i/4}\omega,e^{\pi i/4}\mathcal P)$ the set of sides coming from distinguished sides in the partition  $\mathcal{P}$ of $(M,\omega)$.

\begin{remark}\label{rem:displfunct}
In view Definition~\ref{def:mathcalS} and Remark~\ref{rem:displ}, we have:
\begin{itemize}
\item if $s_{\alpha\beta}$ is a non-distinguished side in $e^{\pi i/4}\mathcal P$ then
\[\mathfrak{D}_{e^{\pi i/4}\omega}(s_{\alpha\beta})=\left\{
\begin{array}{l}
0\text{ or }\\
e^{\pi i/4}2w=\sqrt{2}(w+iw)\text{ or }\\
e^{\pi i/4}2hi=\sqrt{2}(-h+ih);
\end{array}\right.\]
\item if $s_{\alpha\beta}$ is a distinguished side in $e^{\pi i/4}\mathcal P$ then
\[\mathfrak{D}_{e^{\pi i/4}\omega}(s_{\alpha\beta})=\left\{
\begin{array}{l}
e^{\pi i/4}2x_k=\sqrt{2}(x_k+ix_k)\text{ or }\\
e^{\pi i/4}2y_li=\sqrt{2}(-y_l+iy_l)
\end{array}\right.\]
for some $1\leq k< k(\bar{x},\bar{y})$ or $1< l\leq k(\bar{x},\bar{y})$;
\item if $(\sigma,e^{\pi i/4}P_{\beta})\in B(e^{\pi i/4}\omega,e^{\pi i/4}\mathcal P)$ then
\begin{align*}
\mathfrak{B}_{e^{\pi i/4}\omega}(\sigma,e^{\pi i/4}P_{\beta})&=-e^{\pi i/4}\bar{\zeta}_{\beta}(\sigma)\\
&=\left\{
\begin{array}{l}
e^{\pi i/4}(x_l+iy_{l+1})=\frac{\sqrt{2}}{2}((x_l-y_{l+1})+i(x_l+y_{l+1}))\text{ or }\\
e^{\pi i/4}(x_l-iy_{l+1})=\frac{\sqrt{2}}{2}((x_l+y_{l+1})+i(x_l-y_{l+1}))\text{ or }\\
e^{\pi i/4}(-x_l+iy_{l+1})=\frac{\sqrt{2}}{2}((-x_l-y_{l+1})+i(-x_l+y_{l+1}))
\end{array}\right.
\end{align*}
for some $1\leq l< k(\bar{x},\bar{y})$;
\item if $(\sigma,e^{\pi i/4}P_{\alpha})\in E(e^{\pi i/4}\omega,e^{\pi i/4}\mathcal P)$ then
\begin{align*}
\mathfrak{E}_{e^{\pi i/4}\omega}(\sigma,e^{\pi i/4}P_{\alpha})&=e^{\pi i/4}\bar{\zeta}_{\alpha}(\sigma)\\
&=\left\{
\begin{array}{l}
e^{\pi i/4}(x_l+iy_{l+1})=\frac{\sqrt{2}}{2}((x_l-y_{l+1})+i(x_l+y_{l+1}))\text{ or }\\
e^{\pi i/4}(x_l-iy_{l+1})=\frac{\sqrt{2}}{2}((x_l+y_{l+1})+i(x_l-y_{l+1}))\text{ or }\\
e^{\pi i/4}(-x_l+iy_{l+1})=\frac{\sqrt{2}}{2}((-x_l-y_{l+1})+i(-x_l+y_{l+1}))
\end{array}\right.
\end{align*}
for some $1\leq l< k(\bar{x},\bar{y})$.
\end{itemize}
\end{remark}

\begin{definition}\label{def:curve2}
A curve $\mathcal{J}:E_1\in J\mapsto ((M,\omega_{E_1}),\mathcal P(E_1))\in\mathcal S$,  where \(\mathcal S\) is defined by Definition \ref{def:mathcalS}, is called a \emph{$C^\infty$-curve in $\mathcal S$} if the rotated curve   $E_1\mapsto (M,e^{\pi i/4}\omega_{E_1},e^{\pi i/4}\mathcal P(E_1))$ is a $C^{\infty}$-curve of translation surfaces
equipped with proper partitions in the sense of Definition~\ref{def:curve1}.
\end{definition}

 We next show that  each such curves  \(\mathcal{J}\) determines  two sets of functions $\mathscr X$ and $\mathscr Y$ which help to verify  the assumptions of Theorem \ref{thm:gencrit}. By assumption
topological data of $\mathcal{J}(E_1)$ does not change for \(E_1\in J\), so, for every $\alpha\in\mathcal A$ we have $P_\alpha(E_1)=P(\sigma_1^\alpha \bar{x}^\alpha(E_1),\sigma_2^\alpha \bar{y}^\alpha(E_1))\cap B^\alpha(E_1)$,
where $k^\alpha:=k(\bar{x}^\alpha(E_1),\bar{y}^\alpha(E_1))$ does not depend on $E_1\in J$ and the map
\[J\ni E_1\mapsto (\bar{x}^\alpha(E_1),\bar{y}^\alpha(E_1))\in \R_{> 0}^{k^\alpha}\times \R_{> 0}^{k^\alpha}\] is of class $C^\infty$. Moreover, there are two $C^\infty$-maps $w,h:J\to\R_{>0}$ such that
\[y^\alpha_1(E_1)=h(E_1)\text{ and }x^\alpha_{k^\alpha}(E_1)=w(E_1)\text{ for all }E_1\in J\text{ and }\alpha\in\mathcal{A}.\]
Let us consider two finite families of real $C^\infty$ maps on $J$:
\[\mathscr X=\{x^\alpha_l:\alpha\in\mathcal A,1\leq l<k^\alpha\},\quad \mathscr Y=\{y^\alpha_l:\alpha\in\mathcal A,1< l\leq k^\alpha\}.\]

\begin{theorem}\label{thm:mainquadrational}
Let $J\ni E_1\mapsto ((M,\omega_{E_1}),\mathcal P(E_1))\in\mathcal S$ be a $C^\infty$-curve in $\mathcal{S}$ so that:
\begin{itemize}
\item[$(*)$] for every $E_1 \in J$ every $(\psi_t^{\pi/4})_{t\in\R}$-orbit on $(M,\omega_{E_1})$ hits some distinguished side of the partition
$\mathcal P(E_1)$.
\end{itemize}
Suppose that:
\begin{itemize}
\item[($i$)] for any choice of integer numbers $n_{\mathbf{x}}$ for $\mathbf{x}\in\mathscr{X}$ and $m_{\mathbf{y}}$ for $\mathbf{y}\in\mathscr{Y}$
such that not all of them are zero and any  $n_{w},m_{h}\in\Z$, we have
\begin{equation*}
\sum_{\mathbf{x}\in\mathscr{X}\cup\{w\}}n_{\mathbf{x}}\mathbf{x}(E_1)+\sum_{\mathbf{y}\in\mathscr{Y}\cup\{h\}}m_{\mathbf{y}}\mathbf{y}(E_1)\neq 0\text{ for a.e. }E_1\in J;
\end{equation*}
\item[($ii_{+-}$)] for every $E_1\in J$ we have  $\mathbf{x}'(E_1)> 0$ for all $\mathbf{x}\in\mathscr{X}$, $\mathbf{y}'(E_1)< 0$ for all $\mathbf{y}\in\mathscr{Y}$,  $w'(E_1)\geq 0$ and $h'(E_1)\leq 0$
or;
\item[($ii_{-+}$)] for every $E_1\in J$ we have  $\mathbf{x}'(E_1)< 0$ for all $\mathbf{x}\in\mathscr{X}$, $\mathbf{y}'(E_1)> 0$ for all $\mathbf{y}\in\mathscr{Y}$,  $w'(E_1)\leq 0$ and $h'(E_1)\geq 0$.
\end{itemize}
Then for a.e.\ $E_1\in J$ the  flow $(\psi_t^{\pi/4})_{t\in\R}$ on $(M,\omega_{E_1})$  is uniquely ergodic.
\end{theorem}

\begin{proof}
We apply Theorem~\ref{thm:gencrit} to the curve $E_1\mapsto ((M,e^{\pi i/4}\omega_{E_1}),e^{\pi i/4}\mathcal P(E_1))$
of translation surfaces equipped with partitions into polygons and to the reference function $\ell=1$.
Denote by $\mathbf D^*$ the set of rotated distinguished sides coming
from distinguished basic polygons in $\mathcal P(E_1)$. By assumption $(*)$, for every $E_1\in J$
every vertical orbit in $(M,e^{\pi i/4}\omega_{E_1})$ hits at least one side in $\mathbf D^*$, the condition $(*)$ in Theorem~\ref{thm:gencrit} is verified.

In view Remark~\ref{rem:displfunct} we have:
\begin{itemize}
\item every map in $\mathscr B$ and $\mathscr E$ is of the form
\begin{equation}\label{eq:BE}
\tfrac{\sqrt{2}}{2}(\mathbf{x}-\mathbf{y}+i(\mathbf{x}+\mathbf{y}))\;\text{ or }\;
\tfrac{\sqrt{2}}{2}(\mathbf{x}+\mathbf{y}+i(\mathbf{x}-\mathbf{y}))\;\text{ or }\;
\tfrac{\sqrt{2}}{2}(-\mathbf{x}-\mathbf{y}+i(-\mathbf{x}+\mathbf{y}))
\end{equation}
for some $\mathbf{x}\in\mathscr{X}$ and $\mathbf{y}\in\mathscr{Y}$;
\item every map in $\mathscr D^*$  is of the form
\begin{equation}\label{eq:D*}
{\sqrt{2}}(\mathbf{x}+i\mathbf{x})\quad\text{ or }\quad
{\sqrt{2}}(-\mathbf{y}+i\mathbf{y})
\end{equation}
for some $\mathbf{x}\in\mathscr{X}$ or $\mathbf{y}\in\mathscr{Y}$;
\item every map in $\mathscr D\setminus \mathscr D^*$  is of the form
\begin{equation}\label{eq:D}
0\text\quad\text{ or }\quad
{\sqrt{2}}(w+iw)\quad\text{ or }\quad
{\sqrt{2}}(-h+ih).
\end{equation}
\end{itemize}
We now verify the condition $(i)$ in Theorem~\ref{thm:gencrit}  for the curve $E_1\mapsto ((M,e^{\pi i/4}\omega_{E_1}),e^{\pi i/4}\mathcal P(E_1))$. Suppose, contrary to our claim, that there are $f\in \mathscr B$, $g\in \mathscr E$ and a sequence $(n_h)_{h\in\mathscr D}$ in $\Z_{\geq 0}$
such that $n_h>0$ for some $h\in \mathscr D^*$ and
\[\rp f+\rp g+\sum_{h\in \mathscr D} n_h\rp h= 0\ \text{ on a subset of $J$ of positive measure.}\]
In view of \eqref{eq:BE}, \eqref{eq:D*} and \eqref{eq:D}, we have
\[\sigma_1^B\mathbf{x}_B-\sigma_2^B\mathbf{y}_B+\sigma_1^E\mathbf{x}_E-\sigma_2^E\mathbf{y}_E+
\sum_{\mathbf{x}\in\mathscr X}2n_{\mathbf{x}}\mathbf{x}-\sum_{\mathbf{y}\in\mathscr Y}2m_{\mathbf{y}}\mathbf{y}
+2n_{w}w-2m_{h}h=0\]
on a subset of positive measure with:
\begin{itemize}
\item $n_{\mathbf{x}}\in\Z_{\geq 0}$ for $\mathbf{x}\in\mathscr X\cup\{w\}$ and
$m_{\mathbf{y}}\in\Z_{\geq 0}$ for $\mathbf{y}\in\mathscr Y\cup\{h\}$  such that
at least one  $n_{\mathbf{x}}$, $\mathbf{x}\in\mathscr X$ or
$m_{\mathbf{y}}$, $\mathbf{y}\in\mathscr Y$ is positive;
\item $\mathbf{x}_B,\mathbf{x}_E\in\mathscr X$ and $\mathbf{y}_B,\mathbf{y}_E\in\mathscr Y$;
\item $(\sigma_1^B,\sigma_2^B),(\sigma_1^E,\sigma_2^E)\in\{(1,1),(1,-1),(-1,1)\}$.
\end{itemize}
If follows that
\begin{equation}\label{eq:positive}
\sum_{\mathbf{x}\in\mathscr X}{n}_{\mathbf{x}}+\sum_{\mathbf{y}\in\mathscr Y}m_{\mathbf{y}}>0 , \quad\sigma_1^B
+\sigma_2^B\geq 0\quad\text{and}\quad\sigma_1^E+\sigma_2^E\geq 0.
\end{equation}
Moreover,
\begin{equation}\label{eq:zerotilde}
\sum_{\mathbf{x}\in\mathscr X}\widetilde{n}_{\mathbf{x}}\mathbf{x}-\sum_{\mathbf{y}\in\mathscr Y}\widetilde m_{\mathbf{y}}\mathbf{y}
+2n_{w}w-2m_{h}h=0
\end{equation}
on a subset of positive measure, where
\begin{align*}
\widetilde{n}_{\mathbf{x}}=2n_{\mathbf{x}}+\sigma_1^B\delta_{\mathbf{x},\mathbf{x}_B}+\sigma_1^E\delta_{\mathbf{x},\mathbf{x}_E},\quad
\widetilde{m}_{\mathbf{y}}=2m_{\mathbf{y}}+\sigma_2^B\delta_{\mathbf{y},\mathbf{y}_B}+\sigma_2^E\delta_{\mathbf{y},\mathbf{y}_E}.
\end{align*}
Notice that at least one  $\widetilde n_{\mathbf{x}}$ for $\mathbf{x}\in\mathscr X$ or
$\widetilde m_{\mathbf{y}}$ for $\mathbf{y}\in\mathscr Y$ is non-zero.
Indeed, by the definition of $\widetilde{n}_{\mathbf{x}}$ and $\widetilde{m}_{\mathbf{y}}$ and using \eqref{eq:positive}, we have
\[\sum_{\mathbf{x}\in\mathscr X}\widetilde{n}_{\mathbf{x}}+\sum_{\mathbf{y}\in\mathscr Y}\widetilde m_{\mathbf{y}}=
\sum_{\mathbf{x}\in\mathscr X}2{n}_{\mathbf{x}}+\sum_{\mathbf{y}\in\mathscr Y}2m_{\mathbf{y}}+\sigma_1^B
+\sigma_2^B+\sigma_1^E+\sigma_2^E>0.\]
Hence, we have \eqref{eq:zerotilde} on a subset of positive measure with at least one  $\widetilde n_{\mathbf{x}}$, $\mathbf{x}\in\mathscr X$ or   $\widetilde m_{\mathbf{y}}$, $\mathbf{y}\in\mathscr Y$ is positive. This contradicts the assumption $(i)$ of the theorem. It follows that
the condition $(i)$ in Theorem~\ref{thm:gencrit} holds for the curve $E_1\mapsto ((M,e^{\pi i/4}\omega_{E_1}),e^{\pi i/4}\mathcal P(E_1))$.

We now show that the assumption $(ii_{+-})$ of the theorem implies  the condition $(ii_+)$ in Theorem~\ref{thm:gencrit} holds for the curve $E_1\mapsto ((M,e^{\pi i/4}\omega_{E_1}),e^{\pi i/4}\mathcal P(E_1))$. In view of \eqref{eq:D*} and \eqref{eq:D}, every $h\in \mathscr D$
is of the form
\[0\ \text{or}\ \sqrt{2}(\mathbf{x}+i\mathbf{x})\ \text{or}\ \sqrt{2}(-\mathbf{y}+i\mathbf{y})\ \text{or}\ \sqrt{2}(w+iw)\ \text{or}\ \sqrt{2}(-h+ih)\]
for some $\mathbf{x}\in\mathscr X$ or $\mathbf{y}\in\mathscr Y$. It follows that
\begin{equation}\label{eq:hform}
[\rp h,\ell](E_1)=\tfrac{d}{dE_1}\rp h(E_1)
=\left\{
\begin{array}{cl}
0&\text{or}\\
\sqrt{2}\mathbf{x}'(E_1)&\text{or}\\
-\sqrt{2}\mathbf{y}'(E_1)&\text{or}\\
\sqrt{2}w'(E_1)&\text{or}\\
-\sqrt{2}h'(E_1).&
\end{array}
\right.
\end{equation}
In view of the assumption $(ii_{+-})$, we have $[\rp h,\ell](E_1)\geq 0$ for all $h\in \mathscr D$ and $E_1\in J$ and
 $[\rp h,\ell](E_1)> 0$ for all $h\in \mathscr D^*$ and $E_1\in J$, which  gives the condition $(ii_+)$ in Theorem~\ref{thm:gencrit}.

The proof that the assumption $(ii_{-+})$ of the theorem implies  the condition $(ii_-)$ in Theorem~\ref{thm:gencrit} also follows directly from \eqref{eq:hform}.

We finish the proof by applying Theorem~\ref{thm:gencrit} to  $E_1\mapsto ((M,e^{\pi i/4}\omega_{E_1}),e^{\pi i/4}\mathcal P(E_1))$.
This yields the unique ergodicity for the  flow $(\psi_t^{v})_{t\in\R}$ on $(M,e^{\pi i/4}\omega_{E_1})$ for a.e.\ $E_1\in J$.
Thus, the flow $(\psi_t^{\pi/4})_{t\in\R}$ on $(M,\omega_{E_1})$ is uniquely ergodic  for a.e.\ $E_1\in J$.
\end{proof}

\section{The proof of  Theorem~\ref{thm:mainquad},~\ref{thm:U+-}~and~\ref{thm:highmedH}}\label{sec:proofs}
\begin{proof}
[Proof of Theorem~\ref{thm:mainquad}]
Assume that $J\in\mathcal U_I^+$. Then we use the natural partition \(\mathcal P(E_1)\) of the translation surface \(  M(E_1)\) into $16$ staircase polygons $\{P(E_1)^{\varsigma_1\varsigma_2}_{\sigma_1\sigma_2}:\varsigma_1,\varsigma_2,\sigma_1,\sigma_2\in\{\pm\}\}$, shown in  Figure~\ref{fig:surface}, for every $E_1\in J$. It follows that the corresponding curve   $\mathcal{J}$ has the following properties:
\begin{itemize}
\item has a natural partition $\mathcal P(E_1)$ into basic polygons, see Figure~\ref{fig:surface};
\item every element of $\mathcal P(E_1)$ is a staircase polygon, i.e.\ has no sides in direction $\pi/4$;
\item the width of every polygon in $\mathcal P(E_1)$ is $\tfrac{1}{4}T_1(E_1)$ and its height is $\tfrac{1}{4}T_2(E-E_1)$;
\item at least one removed polygon in  $M(E_1)$ is non-trivial ($J\subset I\subset I_{intimp}$);
\item every $(\psi^{\pi/4}_t)_{t\in\R}$-orbit on $M(E_1)$ hits the boundary of some removed polygon, see Remark~\ref{rem:meetbound}.
\end{itemize}
Denote by $\mathbf D^*$ all sides in $\mathcal P(E_1)$ which are part of the boundary of removed polygons.
By Remark~\ref{rem:XYI},    $\mathcal{J}$ is a $C^\infty$-curve in $\mathcal S$, $J\ni E_1\mapsto(M(E_1),\mathcal P(E_1))\in\mathcal{S}$, where
\begin{gather*}
\mathscr X=\{\psi_1(x,E_1):x\in X_I\},\quad\mathscr Y=\{\psi_2(y,E-E_1):y\in Y_I\},\\
\quad w(E_1)=\tfrac{1}{4}T_1(E_1)\ \text{ and }\ h(E_1)=\tfrac{1}{4}T_2(E-E_1).
\end{gather*}
As for every $E_1 \in J$ every $(\psi_t^{\pi/4})_{t\in\R}$-orbit on $M(E_1)$ hits some side in $\mathbf D^*$, condition $(*)$ of  Theorem~\ref{thm:mainquadrational} is satisfied. Condition $(i)$ in Theorem~\ref{thm:mainquadrational} follows directly from Proposition~\ref{prop:indep2}  and
 condition $(i_{-+})$ in Theorem~\ref{thm:mainquadrational} follows directly from Lemma~\ref{lem:negposder} and the fact that both $T_1$ and $T_2$ are constant. This gives unique ergodicity of the flow $(\psi_t^{\pi/4})_{t\in\R}$ on $M(E_1)$ for a.e.\ $E_1\in J$.

\medskip

Assume next that $J\in\mathcal U_I^-$. Without loss of generality we deal only with $(\psi_t^{\pi/4})_{t\in\R}$ restricted to $(R\setminus\mathbf R^{red})\setminus\partial (R\setminus\mathbf R^{red})$.
The same arguments apply to the set  $(G\setminus\mathbf R^{green})\setminus\partial (G\setminus\mathbf R^{green})$. We will prove the unique ergodicity of  $(\psi_t^{\pi/4})_{t\in\R}$ on $\overline{R}(E_1)$ for a.e.\ $E_1\in J$, which implies the unique ergodicity on $(R\setminus\mathbf R^{red})\setminus\partial (R\setminus\mathbf R^{red})$. For every $E_1\in J$ let ${\mathcal P}_R(E_1)$ be a partition of $\overline{R}(E_1)$
such that each polygon in  ${\mathcal P}_R(E_1)$ is a connected component of the intersection of a polygon from $\mathcal P(E_1)$ (the partition of $M(E_1)$ used in the first part of the proof) and $\overline{R}(E_1)$
(see Figures~\ref{fig:cylinders}-\ref{fig:completion2}).
Then
\begin{itemize}
\item every polygon in ${\mathcal P}_R(E_1)$ is a basic polygon, see Figures~\ref{fig:cylinders} and  \ref{fig:completion2};
\item the ends of every side in direction $\pi/4$ are regular points in $\overline{R}(E_1)$, see the construction of $\overline{R}(E_1)$
in Section~\ref{sec:compl};
\item the formal width of every basic polygon in ${\mathcal P}_R(E_1)$  is $\tfrac{1}{4}T_1(E_1)$ and its formal  height is $\tfrac{1}{4}T_2(E-E_1)$;
\item the set $\mathbf D^*$ of distinguished sides is non-empty;
\item every $(\psi^{\pi/4}_t)_{t\in\R}$-orbit on $\overline{R}(E_1)$ hits $\mathbf D^*$, see Remark~\ref{rem:hit}.
\end{itemize}
Therefore, $J\ni E_1\mapsto(\overline{R}(E_1),{\mathcal P}_R(E_1))\in\mathcal{S}$ is a $C^\infty$-curve  in $\mathcal S$ such that
\begin{gather*}
\mathscr X\subset\{\psi_1(x,E_1):x\in X_I\},\quad\mathscr Y\subset\{\psi_2(y,E-E_1):y\in Y_I\},\\
\quad w(E_1)=\tfrac{1}{4}T_1(E_1)\ \text{ and }\ h(E_1)=\tfrac{1}{4}T_2(E-E_1).
\end{gather*}
and at least one set $\mathscr X$ or $\mathscr Y$ is non-empty.
As for every $E_1 \in J$ every $(\psi_t^{\pi/4})_{t\in\R}$-orbit on $\overline{R}(E_1)$ hits $\mathbf D^*$, condition $(*)$ of Theorem~\ref{thm:mainquadrational} is satisfied by the curve $E_1\mapsto(\overline{R}(E_1),{\mathcal P}_R(E_1))$.
Condition $(i)$ in Theorem~\ref{thm:mainquadrational} follows directly from Proposition~\ref{prop:indep2} and
the condition $(i_{-+})$ in Theorem~\ref{thm:mainquadrational} follows directly from Lemma~\ref{lem:negposder} and the fact that both $T_{1}$ and $T_{2}$ are constant. This gives unique ergodicity of the flow $(\psi_t^{\pi/4})_{t\in\R}$ on $\overline{R}(E_1)$ for a.e.\ $E_1\in J$, which completes the proof.
\end{proof}
\begin{proof}
[Proof of Theorem~\ref{thm:U+-}]
In  Section~\ref{sec:osctobil}  we showed that the flow \(\varphi_{t}^{P,E,E_1}\) is topologically conjugated to the directional  billiard flow in the billiard table   \(\mathbf{P}_{E,E_1}\) in the standard directions \((\frac{\pi}{4},-\frac{\pi}{4},\frac{3\pi}{4},-\frac{3\pi}{4})\). This flow is conjugated to the flow   $(\psi_t^{\pi/4})_{t\in\R}$   on the translational  \(M(E_{1})  \) (see Section~\ref{sec:billtotrans}). Thus, Theorem~\ref{thm:U+-} is a direct consequence  of Theorem~\ref{thm:mainquad}.
\end{proof}
\begin{proof}[Proof of Theorem~\ref{thm:highmedH}]
 As in the proof of the first part of Theorem~\ref{thm:mainquad} (when there was no need for the completion procedure), we construct a surface and partition on which  Theorem~\ref{thm:mainquadrational} can be applied. Some subtle adjustments are necessary since a straight forward application of the procedure leads to basic polygons with different heights or widths.
Suppose that $I\in \mathcal J_E$ such that
\[I\subset \Big[0,E-\min_{\varsigma_1,\varsigma_2\in\{\pm\}}V_2(y_1^{\varsigma_1\varsigma_2})\Big]
\cup \Big[\min_{\varsigma_1,\varsigma_2\in\{\pm\}}V_1(x_{k(\bar{x}^{\varsigma_1\varsigma_2},\bar{y}^{\varsigma_1\varsigma_2})}^{\varsigma_1\varsigma_2}),E\Big].\]
In view of part ($\gamma$) of Theorem \ref{thm:albegacopdet}~we can assume that $I$ is not contained in
\[\Big[0,E-\max_{\varsigma_1,\varsigma_2\in\{\pm\}}V_2(y_1^{\varsigma_1\varsigma_2})\Big]
\cup \Big[\max_{\varsigma_1,\varsigma_2\in\{\pm\}}V_1(x_{k(\bar{x}^{\varsigma_1\varsigma_2},\bar{y}^{\varsigma_1\varsigma_2})}^{\varsigma_1\varsigma_2}),E\Big].\]
So we need to consider three cases:
\begin{align}\label{eq:energyintrvls}
\begin{split}
I \subset I_1=&\Big[E-\max_{\varsigma_1,\varsigma_2\in\{\pm\}}V_2(y_1^{\varsigma_1\varsigma_2}), E-\min_{\varsigma_1,\varsigma_2\in\{\pm\}}V_2(y_1^{\varsigma_1\varsigma_2})\Big]\\
&\cap\Big[0,\min_{\varsigma_1,\varsigma_2\in\{\pm\}}V_1(x_{k(\bar{x}^{\varsigma_1\varsigma_2},\bar{y}^{\varsigma_1\varsigma_2})}^{\varsigma_1\varsigma_2})\Big) \\
I  \subset I_2=&\Big[\min_{\varsigma_1,\varsigma_2\in\{\pm\}}V_1(x_{k(\bar{x}^{\varsigma_1\varsigma_2},\bar{y}^{\varsigma_1\varsigma_2})}^{\varsigma_1\varsigma_2}),E-\min_{\varsigma_1,\varsigma_2\in\{\pm\}}V_2(y_1^{\varsigma_1\varsigma_2})\Big]
 \\
 I  \subset I_{3}=&\Big[\min_{\varsigma_1,\varsigma_2\in\{\pm\}}V_1(x_{k(\bar{x}^{\varsigma_1\varsigma_2},\bar{y}^{\varsigma_1\varsigma_2})}^{\varsigma_1\varsigma_2}),\max_{\varsigma_1,\varsigma_2\in\{\pm\}}V_1(x_{k(\bar{x}^{\varsigma_1\varsigma_2},\bar{y}^{\varsigma_1\varsigma_2})}^{\varsigma_1\varsigma_2})\Big]\\
&\cap\Big(E-\min_{\varsigma_1,\varsigma_2\in\{\pm\}}V_2(y_1^{\varsigma_1\varsigma_2}),E\Big].
\end{split}
\end{align}
First we consider \(I\subset I_{1}\), where impacts occur with one extremal horizontal boundary but not with the other.
Without loss of generality, we assume that the upper staircase polygons are taller than the lower ones, so the impact occurs with the lower extremal  horizontal boundary:\begin{equation}\label{eq:Iminmax1}
\max_{\varsigma_1,\varsigma_2\in\{\pm\}}V_2(y_1^{\varsigma_1\varsigma_2})=V_2(y_1^{++})=V_2(y_1^{-+})>\min_{\varsigma_1,\varsigma_2\in\{\pm\}}V_2(y_1^{\varsigma_1\varsigma_2})=V_2(y_1^{+-})=V_2(y_1^{--}).
\end{equation}
In view of Remark~\ref{rem:XYI},  and Eqs. (\ref{eq:energyintrvls}) and \eqref{eq:Iminmax1}, for every $E_1\in I_{1}$ the height of $\mathbf P^{++}_{E,E_1}$ and $\mathbf P^{-+}_{E,E_1}$ is $h=\tfrac{1}{4}T_2(E-E_1)$ and the height of $\mathbf P^{+-}_{E,E_1}$ and $\mathbf P^{--}_{E,E_1}$ is $\psi_2({y_1^{+-}},E-E_1)=\psi_2({y_1^{--}},E-E_1)<h$, and the width
of any $\mathbf P^{\varsigma_1\varsigma_2}_{E,E_1}$ is $w=\tfrac{1}{4}T_1(E_1)$.
It follows that the natural partition $\mathcal P(E_1)$ of $M(E_1)$ into staircase polygons formally does not give an element in the class $\mathcal S$ since not all polygons have the same heights (part $(vii)$ of Definition~\ref{def:mathcalS}).
To to get rid of this problem we  artificially increase the polygons $\mathbf P^{+-}_{E,E_1}$ and $\mathbf P^{--}_{E,E_1}$ to  $\widehat{\mathbf  P}^{+-}_{E,E_1}$ and $\widehat{\mathbf P}^{--}_{E,E_1}$ by adding a vertical segment as the first step in these polygons as in  Figure~\ref{fig:Degenerated}, so, formally:
\begin{gather*}
\widehat{\mathbf  P}^{+-}_{E,E_1}=P((0,\bar{\Psi}_1^{+-}(E,E_1)),-(h,\bar{\Psi}_2^{+-}(E,E_1))),\\
\widehat{\mathbf  P}^{--}_{E,E_1}=P(-(0,\bar{\Psi}_1^{--}(E,E_1)),-(h,\bar{\Psi}_2^{--}(E,E_1)))\text{ if }\\
{\mathbf  P}^{+-}_{E,E_1}=P(\bar{\Psi}_1^{+-}(E,E_1),-\bar{\Psi}_2^{+-}(E,E_1)),\
{\mathbf  P}^{--}_{E,E_1}=P(-\bar{\Psi}_1^{--}(E,E_1),-\bar{\Psi}_2^{--}(E,E_1)).
\end{gather*}
 Now, collecting these extended polygons, we denote by $\widehat{\mathbf  P}_{E,E_1}$ the extension of the polygon ${\mathbf  P}_{E,E_1}$ by one vertical interval so that $\widehat{\mathbf  P}_{E,E_1}$ is the union of
${\mathbf  P}^{++}_{E,E_1}$, ${\mathbf  P}^{-+}_{E,E_1}$, $\widehat{\mathbf  P}^{+-}_{E,E_1}$ and $\widehat{\mathbf  P}^{--}_{E,E_1}$, see Figure~\ref{fig:Degenerated}. Let $\widehat M(E_1)$ be the object arising after applying the unfolding procedure to the degenerated polygon $\widehat{\mathbf  P}_{E,E_1}$, see Figure~\ref{fig:Degenerated}. Formally, $\widehat M(E_1)$ is the translation surface $M(E_1)$ with two vertical loops attached. The directional flow on $\widehat M(E_1)$ in direction $\pi/4$ coincides with the flow $(\psi_t^{\pi/4})_{t\in\R}$ on $M(E_1)$.
\begin{figure}[h]
\includegraphics[width=1 \textwidth]{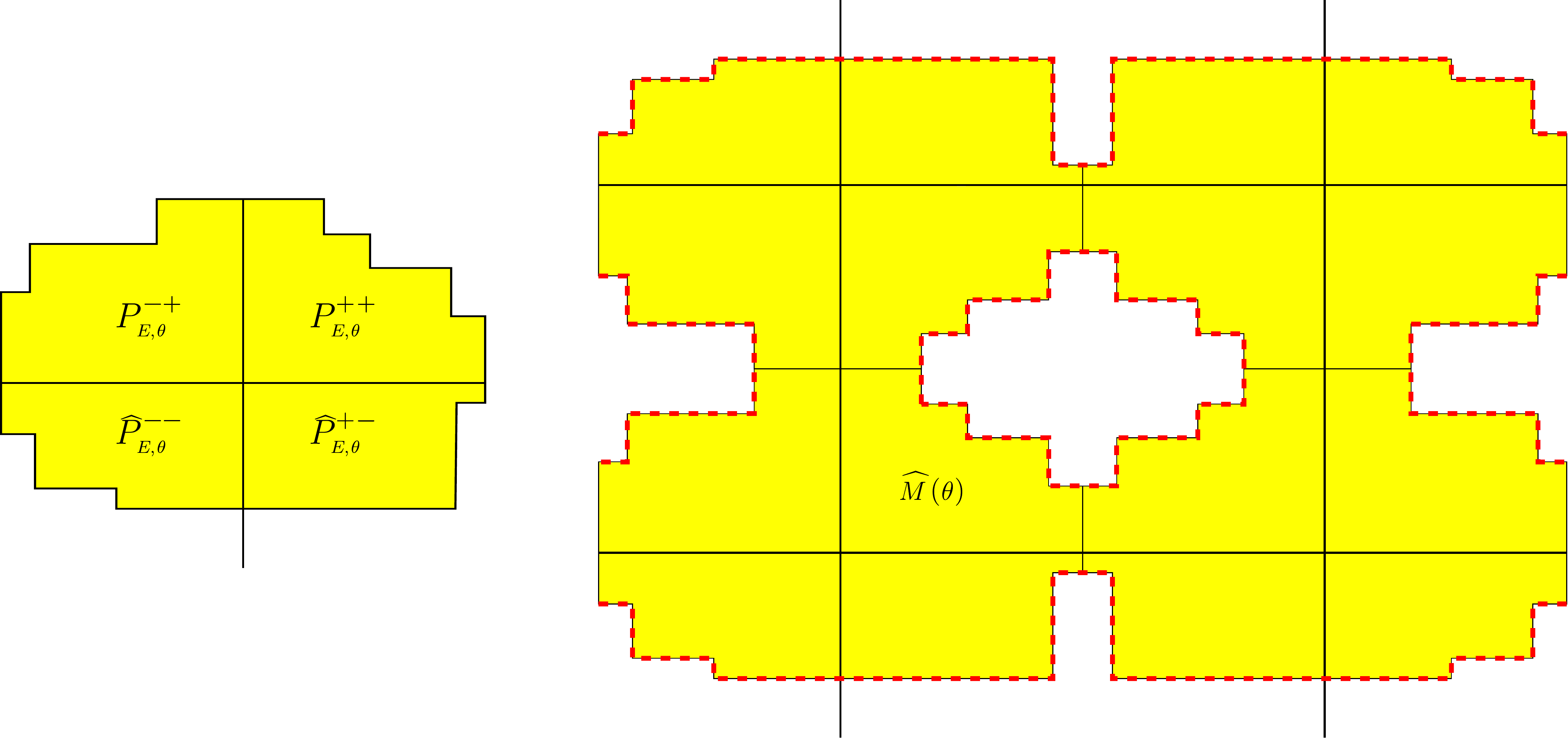}
\caption{The extended billiard table $\widehat{{\mathbf  P}}_{E,E_1}$ and  $\widehat M(E_1)$.}
\label{fig:Degenerated}
\end{figure}
Moreover, $\widehat M(E_1)$ has a natural partition $\widehat{\mathcal P}(E_1)$ into basic polygons, eight of them are degenerated having additional vertical segments.
Let us consider
the $C^\infty$-curve $I\ni E_1\mapsto(\widehat M(E_1),\widehat{\mathcal P}(E_1))$. Since for every $E_1\in I\subset I_{1}$ all sides of
$\widehat{\mathcal P}(E_1)$ are vertical or horizontal and every $(\psi^{\pi/4}_t)_{t\in\R}$-orbit in $\widehat M(E_1)$ hits $\mathbf D^*$ (the red dashed sides on Figure~\ref{fig:Degenerated}) and does not hit the additional vertical loops, we can apply the same arguments to those presented in Theorem~\ref{thm:mainquadrational} to this curve with\begin{gather*}
\mathscr X=\{\psi_1(x,E_1):x\in X_I\},\quad\mathscr Y=\{\psi_2(y,E-E_1):y\in Y_I\},\\
\quad w(E_1)=\tfrac{1}{4}T_1(E_1)\ \text{ and }\ h(E_1)=\tfrac{1}{4}T_2(E-E_1).
\end{gather*}
The final argument is the same as in the proof of the first part of Theorem~\ref{thm:mainquad}.

Similar arguments apply to the case of \(I\subset I_{3}\), where  one adds horizontal segments to the staircase polygons for achieving a fixed width to all of them.

Finally, for the case of  \(I\subset I_{2}\), adding both horizontal and vertical segments completes the proof.  \end{proof}

\begin{proof}[Proof of Theorem~\ref{thm:highE}]
Note that Theorem~\ref{thm:highE} is a simple consequence of Theorems~\ref{thm:highmedH} and part ($\gamma$) of Theorem \ref{thm:albegacopdet} as in this case\begin{displaymath}
[0,E]=\Big[0,E-\max_{\varsigma_1,\varsigma_2\in\{\pm\}}V_2(y_1^{\varsigma_1\varsigma_2})\Big]
\cup I_{1}\cup I_2 \cup I_3  \cup \Big[\max_{\varsigma_1,\varsigma_2\in\{\pm\}}V_1(x_{k(\bar{x}^{\varsigma_1\varsigma_2},\bar{y}^{\varsigma_1\varsigma_2})}^{\varsigma_1\varsigma_2}),E\Big].
\end{displaymath}
\end{proof}
This finishes the proof of the chain of results announced in the Introduction.

\section{Non-uniform ergodic averages in the configuration space}\label{subsec:ergodicav}

We show first that in the common case, when unique ergodcity of the motion on level sets is established, it induces a smooth measure for  ergodic averages in the configuration space. On the other hand, we show that when unique ergodcity holds only on the red/green sets (as established  in Theorem \ref{thm:U+-} for resonant quadratic potentials), it induces non-smooth measure in the configuration space. Notably, here we do not assume that the potentials are quadratic, yet we do assume that the level set is resonant and is partitioned to periodic ribbons and green and red invariant sets, with UE established on the green and red sets.

 Fix the energies $0<E_1<E$ and let us consider the invariant set $S_{E,E_1}^P$. Denote by $\pi:S_{E,E_1}^P\to R^{(E,E_1)}\cap P$ the projection on the configuration space.
 For every $(x,y)\in R^{(E,E_1)}\cap P$ and $\sigma_1,\sigma_2\in\{\pm\}$
let\begin{equation}\label{eq:psigma1sigma2}
p^{\sigma_1\sigma_2}(x,y)=(\sigma_1\sqrt{2}\sqrt{E_1-V_1(x)},\sigma_2\sqrt{2}\sqrt{E-E_1-V_2(y)}).
\end{equation}
For every $\sigma_1,\sigma_2\in\{\pm\}$
let
\[S_{E,E_1}^{\sigma_1,\sigma_2}=\{(x,y,p^{\sigma_1\sigma_2}(x,y)):(x,y)\in R^{(E,E_1)}\cap P\}.\]
Then $S_{E,E_1}^{\sigma_1,\sigma_2}$ is naturally identified (via $\pi$) with the configuration space $R^{(E,E_1)}\cap P$ and
$S_{E,E_1}^P$ is the union of $S_{E,E_1}^{\sigma_1,\sigma_2}$ for $\sigma_1,\sigma_2\in\{\pm\}$ so that
$(x,y,p^{\sigma_1+}(x,y))$ and $(x,y,p^{\sigma_1-}(x,y))$ {are identified} when $(x,y)$ lies on a horizontal side of $R^{(E,E_1)}\cap P$ and
$(x,y,p^{+\sigma_2}(x,y))$ and $(x,y,p^{-\sigma_2}(x,y))$ {are identified} when $(x,y)$ lies on a vertical side of $R^{(E,E_1)}\cap P$.
 On $S_{E,E_1}^P$ we have unique probability  Liouville measure $\mu_{E,E_1}$ such that $\mu_{E,E_1}$ restricted to $S_{E,E_1}^{\sigma_1,\sigma_2}$
is identified with the measure $\tfrac{1}{4}g(x,y)dxdy=\tfrac{1}{4}d\psi _{1}d\psi_{2}$ on $R^{(E,E_1)}\cap P$, where, by Eq. (\ref{def:psi}),
\[g(x,y)=\frac{1}{\sqrt{(E_1-V_1(x))(E-E_1-V_2(y))}}\frac{1}{\int_{R^{(E,E_1)}\cap P}\frac{dsdu}{\sqrt{(E_1-V_1(s))(E-E_1-V_2(u))}}}.\]
 For every $(x,y)\in R^{(E,E_1)}\cap P$, $\sigma_1,\sigma_2\in\{\pm\}$ and $t\in\R$ let
 \[(x_t^{\sigma_1\sigma_2},y_t^{\sigma_1\sigma_2}):= \pi(\varphi_t^{P,E,E_1}(x,y,p^{\sigma_1\sigma_2}(x,y))).\]

Suppose that the flow $(\varphi_t^{P,E,E_1})_{t\in\R}$ is uniquely ergodic. Then {for an observable of the configuration space}, namely every continuous map $f:R^{(E,E_1)}\cap P\to\R$ and all $(x,y)\in R^{(E,E_1)}\cap P$ and
$\sigma_1,\sigma_2\in\{\pm\}$ we have
\begin{align*}
\lim_{T\to\pm\infty}\frac{1}{T}&\int_{0}^Tf(x_t^{\sigma_1\sigma_2},y_t^{\sigma_1\sigma_2})\,dt=
\lim_{T\to\pm\infty}\frac{1}{T}\int_{0}^Tf\circ\pi(\varphi_t^{P,E,E_1}(x,y,p^{\sigma_1\sigma_2}(x,y)))\,dt\\
&=\int_{S_{E,E_1}^P}f\circ\pi\, d\mu_{E,E_1}=\sum_{\sigma_1,\sigma_2\in\{\pm\}}\int_{S_{E,E_1}^{\sigma_1,\sigma_2}}f\circ\pi\, d\mu_{E,E_1}\\
&=\sum_{\sigma_1,\sigma_2\in\{\pm\}}\int_{R^{(E,E_1)}\cap P}f(x,y)\frac{g(x,y)}{4}\,dx\,dy=\int_{R^{(E,E_1)}\cap P}f(x,y)g(x,y)\,dx\,dy.
\end{align*}
It follows that every infinite semi-orbit in the configuration space $R^{(E,E_1)}\cap P$ is equidistributed with respect to the measure $g(x,y)dxdy$. This measure has smooth density exploding to $+\infty$ on the boundary of
$R^{(E,E_1)}$.

\medskip

Now consider the case of level sets with coexistence of periodic and uniquely ergodic behaviour for $(\varphi_t^{P,E,E_1})_{t\in\R}$   (e.g.   $E_{1} \in J\in\ \mathcal U^-_I$ of Theorem  \ref{thm:U+-}). Namely, assume that  the phase space $S_{E,E_1}^P$ splits into two completely periodic (connected) components and
two uniquely ergodic components, {and assume each of these components have a positive width as in Theorem  \ref{thm:U+-}. In this case there are three types of averages one needs to consider: averages on periodic orbits, and averages over the green/red sets. We show next that the coloured averages induce non-smooth measures in the configuration space.}

Denote the uniquely ergodic component corresponding to the  subset $(R\setminus \mathbf R^{red})\setminus\partial(R\setminus \mathbf R^{red})$
of the surface $M(E_1)$ by $S_{E,E_1}^{red}$ (and similarly, for $(G\setminus \mathbf R^{green})\setminus\partial(G\setminus \mathbf R^{green})$, by $S_{E,E_1}^{green}$).
Recall that the surface $M(E_1)$ has a partition into four star-shaped polygons \(\mathbf P(E_1)_{\sigma_1\sigma_2}\):
\begin{gather*}\mathbf P(E_1)_{++}=\mathbf P(E_1)_{\pi/4},\ \mathbf P(E_1)_{+-}=\gamma_h\mathbf P(E_1)_{-\pi/4},\\
 \mathbf P(E_1)_{-+}=\gamma_v\mathbf P(E_1)_{3\pi/4},\ \mathbf P(E_1)_{--}=\gamma_h\circ\gamma_v\mathbf P(E_1)_{-3\pi/4},
 \end{gather*}
see Figure~\ref{fig:surface}. Since the surface $M(E_1)$ and the invariant set $R\setminus \mathbf R^{red}$ is $e^{i\pi}=\gamma_h\circ\gamma_v$ invariant, we have
\begin{gather}\label{eq:redpm}
\begin{split}
\gamma_h\circ\gamma_v(\mathbf P(E_1)_{--}\cap(R\setminus \mathbf R^{red}))=\mathbf P(E_1)_{++}\cap(R\setminus \mathbf R^{red}),\\
\gamma_v(\mathbf P(E_1)_{-+}\cap(R\setminus \mathbf R^{red}))=\gamma_h(\mathbf P(E_1)_{+-}\cap(R\setminus \mathbf R^{red})),
\end{split}
\end{gather}
see Figure~\ref{fig:cylinders}. Moreover, all four sets have the same Lebesgue measure  since the removed parts have identical measures in each set. Since the boundary of $\mathbf P(E_1)_{++}\cap(R\setminus \mathbf R^{red})$ consists of linear segment in direction $\pi/4$ and the boundary of $\gamma_h(\mathbf P(E_1)_{+-}\cap(R\setminus \mathbf R^{red}))$ consists of linear segment in direction $-\pi/4$, the sets differ.

In fact, by the same argument, we notice that
 $\mathbf P(E_1)_{++}\cap(R\setminus \mathbf R^{red})$ intersects the horizontal folding of all the other invariant sets:
  $\gamma_h(\mathbf P(E_1)_{+-}\cap(G\setminus \mathbf R^{green}))$ as well as the periodic ribbons   $\gamma_h(\mathbf P(E_1)_{+-}\cap W)$ and    $\gamma_h(\mathbf P(E_1)_{+-}\cap Y)$. The periodic ribbons map one to the other under  \(\gamma_h\circ\gamma_v\) (rotation by \(\pi\)), as these correspond to the same periodic orbits in configuration space with opposite directions of motion along the orbit.

By the construction of the isomorphism between $S_{E,E_1}^P$ and $M(E_1)$, we have
\begin{align}\label{eq:redpm1}
\begin{split}
S_{E,E_1}^{red}\cap S_{E,E_1}^{++}&=\psi^{-1}\big(\mathbf P(E_1)_{++}\cap((R\setminus \mathbf R^{red})\setminus\partial(R\setminus \mathbf R^{red}))\big)\\
S_{E,E_1}^{red}\cap S_{E,E_1}^{+-}&=\psi^{-1}\circ \gamma_h\big(\mathbf P(E_1)_{+-}\cap((R\setminus \mathbf R^{red})\setminus\partial(R\setminus \mathbf R^{red}))\big)\\
S_{E,E_1}^{red}\cap S_{E,E_1}^{-+}&=\psi^{-1}\circ \gamma_v\big(\mathbf P(E_1)_{-+}\cap((R\setminus \mathbf R^{red})\setminus\partial(R\setminus \mathbf R^{red}))\big)\\
S_{E,E_1}^{red}\cap S_{E,E_1}^{--}&=\psi^{-1}\circ \gamma_h\circ\gamma_v\big(\mathbf P(E_1)_{--}\cap((R\setminus \mathbf R^{red})\setminus\partial(R\setminus \mathbf R^{red}))\big).
\end{split}
\end{align}
Every set $S_{E,E_1}^{red}\cap S_{E,E_1}^{\sigma_1,\sigma_2}$ is identified via the projection $\pi$ with a subset $P_{E,E_1}^{red,\sigma_1,\sigma_2}\subset R^{(E,E_1)}\cap P$.
In view of \eqref{eq:redpm} and \eqref{eq:redpm1}, we have
\[P_{E,E_1}^{red,--}=P_{E,E_1}^{red,++},\ P_{E,E_1}^{red,-+}=P_{E,E_1}^{red,+-},\  P_{E,E_1}^{red,+-}\neq P_{E,E_1}^{red,++}\]
and all four sets have the same $\mu_{E,E_1}$-measure, denoted by $\tfrac{1}{4}\mu^{red}$ ($\mu^{red}:=\mu_{E,E_1}( S_{E,E_1}^{red})$). As $(\varphi_t^{P,E,E_1})_{t\in\R}$ is uniquely ergodic on $S_{E,E_1}^{red}$, for every $(x_0^{\sigma_1\sigma_2},y_0^{\sigma_1\sigma_2})=(x_{0},y_{0},p^{\sigma_1,\sigma_2}(x_0,y_0))\in S_{E,E_1}^{red}$ and any continuous function $f:R^{(E,E_1)}\cap P$ we have
\begin{align*}
\lim_{T\to\pm\infty}\frac{1}{T}&\int_{0}^Tf(x_t^{\sigma_1\sigma_2},y_t^{\sigma_1\sigma_2})\,dt=
\lim_{T\to\pm\infty}\frac{1}{T}\int_{0}^Tf\circ\pi(\varphi_t^{P,E,E_1}(x,y,p^{\sigma_1\sigma_2}(x,y)))\,dt\\
&=\int_{S_{E,E_1}^{red}}f\circ\pi\, \frac{d\mu_{E,E_1}}{\mu^{red}}=\sum_{\sigma'_1,\sigma'_2\in\{\pm\}}\int_{S_{E,E_1}^{red}\cap S_{E,E_1}^{\sigma'_1,\sigma'_2}}f\circ\pi\, \frac{d\mu_{E,E_1}}{\mu^{red}}\\
&=\sum_{\sigma'_1,\sigma'_2\in\{\pm\}}\int_{P_{E,E_1}^{red,\sigma'_1,\sigma'_2}}f(x,y)\frac{g(x,y)}{4\mu^{red}}\,dx\,dy\\
&=\int_{R^{(E,E_1)}\cap P}f(x,y)g(x,y)\frac{\chi_{P_{E,E_1}^{red,++}}(x,y)+\chi_{P_{E,E_1}^{red,+-}}(x,y)}{2\mu^{red}}\,dx\,dy.
\end{align*}
It follows that if $(x,y)\in P_{E,E_1}^{red,\sigma_1,\sigma_2}$ , then every of its infinite $\sigma_1\sigma_2$-semi-orbit in the configuration space $R^{(E,E_1)}\cap P$ is equidistributed
on $P_{E,E_1}^{red,++}\cup P_{E,E_1}^{red,+-}\subset R^{(E,E_1)}\cap P$
with respect to the measure $g(x,y)\frac{\chi_{P_{E,E_1}^{red,++}}(x,y)+\chi_{P_{E,E_1}^{red,+-}}(x,y)}{2\mu^{red}}dxdy$. The same phenomenon is observed also on the  \emph{green} component.
Since \({P_{E,E_1}^{red,++}}\neq {P_{E,E_1}^{red,+-}}\),  the resulting measure is only piecewise smooth. Thus, we have three types of measures, the green and red measures which are piecewise smooth  and the continuum of singular measures supported on the periodic orbits. The denominator \(\chi_{P_{E,E_1}^{red,++}}(x,y)+\chi_{P_{E,E_1}^{red,+-}}(x,y)\) equals to 1 on the configuration points at which the red measure has common support with one of the other measures, and equals to 2 at points where only the red measure is supported. We established that the area of both of these sets  is positive.

\medskip

\section{some open problems}

\subsection*{Non-uniform ergodic  properties along non-linear resonant curves}

We considered curves parameterized by \(E_1\) on a given energy surface and proved either unique ergodicity for almost all \(E_1\) on each such surface or, for the resonant linear oscillators case, a more exciting divisions to subintervals, where, for some of which periodic and unique ergodic flows co-exist. More generally, we can consider any curve in the energy space, \((E(E_{1}),E_1),\  E_1\in J\), such that the topological data on this curve is fixed and the numerical data on this curve is properly non-degenerate and monotone, and apply the same tools to the resulting curve.
 In particular, inspired by Theorem \ref{thm:U+-}, one would like to study the dynamics along resonant curves. For example,   consider curves  of the form \(n T_1(E_1)= m T_2(E_{2}^{\frac{n}{m}}=E^{\frac{n}{m}}(E_{1})-E_1)\) for some \(m,n\in\mathbb{N}\), and for which at least one oscillator is non-harmonic. Such curves are of interest, as simulations show that similar to  the smooth case, under perturbations, they produce resonant islands of the impact flow.  Along the corresponding curves, similar  splitting as for the resonant linear oscillators case (Theorem \ref{thm:U+-}) are expected to emerge. Indeed, the construction of such curves and its division to  segments  \( \mathcal U^\pm_I \) is similar to the construction in the proof of Theorem \ref{thm:U+-}. Yet, we are unable to verify the conditions of  Theorem~\ref{thm:mainquadrational} for the general case and leave this to future work.

\subsection*{Quasi-integrable dynamics with other types of potentials and other types of right angled polygons}

 The frameworks introduced in \cite{Issi2019,PRk2020} allows to study a larger class of quasi-integrable HIS for which the ergodic properties are yet to be established. Major differences are expected to arise when one or both of the potentials have  local maxima (namely, is not unimodal). Then, the period dependence on energy is singular and non-monotone  and the scaled translation surfaces can be non-compact. Thus, the ergodic properties for this case are unknown.

 When the right angled polygon is either non star-shaped or even when it is star-shaped but its  kernel does not include the origin, our current methodologies are insufficient to prove unique ergodicity.
{Indeed, the main example presented in \cite{Issi2019}, of impacts of oscillators  from a single step, is included here, for symmetric potentials, iff  the step belongs to a single quadrant (Figure~5a in \cite{Issi2019}). Then, for any fixed \(E\) we can consider the star-shaped polygon which is composed of three sufficiently large rectangles (with respect to \(E\)) and a one-step staircase polygon. Then, the right angled polygons, \(P \cap R^{(E,E_{1})}\) always belong to $\mathscr{R}$. When the step crosses any of the axis (Figure~5b,c,d in \cite{Issi2019}), the kernel of the corresponding star-shaped polygon does not include the origin, so, presently, it cannot be analyzed with our tools. Similarly, when additional finite barriers  and beams are introduced  (see Figure~10 in \cite{Issi2019}), the right angled polygons \(P \cap R^{(E,E_{1})}\) do not always belong to the class  $\mathscr{R}$, and the current results hold only for segments of level sets of this type.}

\subsection*{Acknowledgements}
The authors would like to thank the University of Sydney for their hospitality during the conference
"Workshop on Mathematical Billiards", Sydney, 24-27 June 2019. This meeting was the beginning of the author's collaboration on the project that resulted in the present article.
The first author would like to thank Sasha Gomilko for help in fixing some analytical issues.

\appendix

\section{Examples of $Deck$-potentials}\label{sec:Deckpot}
\begin{proposition}\label{prop:anDeck}
If $V:\R\to\R_{\geq 0}$ is an even analytic map satisfying \eqref{eq:i}, then $V\in Deck$.
\end{proposition}

\begin{proof}
The conditions \eqref{eq:i} and \eqref{eq:ii} result directly from the assumptions. {
Thus for every $E_0>0$ we need to find a complex neighborhood \(U_{E_0}\) on which $V$ is biholomorphic and with image under \(V\)  including a droplet of \(E_0\) (condition  \eqref{eq:iii}), and on which  condition  \eqref{eq:iv}  is satisfied. }

Let $U\subset\C$ be an open neighborhood of $[0,+\infty)$ and $V:U\to\C$ be a holomorphic extension of $V:[0,+\infty)\to[0,  +\infty)$.
As $V(0)=V'(0)=0$, we have $m>1$ such that
\[V(0)=V'(0)=\ldots=V^{(m-1)}(0)=0\;\text{ and }\;V^{(m)}(0)\neq 0.\]
As $V$ is even, $m$ is also even.
In view of \cite[Sec.~3.12.5]{Sa-Ge}, there exists biholomorphic  $V_*:B(0,\vep)\to V_*(B(0,\vep))$ such that
\[V_*^m(z)=V(z)\;\text{ for every }\;z\in B(0,\vep)\;\text{ and }\; V_*(x)=V(x)^{1/m}\;\text{ for every }\; x\in [0,\vep).\]
Therefore, there exists $\widetilde{U}\subset U$ an open neighbourhood of $[0,+\infty)$ and $V_*:\widetilde{U}\to\C$ a holomorphic extension of $V^{1/m}:[0,+\infty)\to[0,+\infty)$
so that $V'_*(z)\neq 0$ for every $z\in \widetilde{U}$.


Take any $E_0>0$. Then there exists $R=R_{E_0}>0$ such that $[-R,x^{max}(E_0)+R]\times[-R,R]\subset \widetilde{U}$ and $V_*$ on $[-R,x^{max}(E_0)+R]\times[-R,R]$ is surjective.
Indeed, suppose, contrary to our claim, that for all $R>0$ the rectangle $[-R,x^{max}(E_0)+R]\times[-R,R]$ is not a subset of  $\widetilde{U}$. Then there exists a sequence
$(z_n)_{n\geq 1}$ in $\C$ such that $z_n \in [-1/n,x^{max}(E_0)+1/n]\times[-1/n,1/n]$ and $z_n\notin \widetilde{U}$. Passing to a subsequence, if necessary, we have $z_n\to x\in [0, x^{max}(E_0)]$ and  $x\notin \widetilde{U}$,
contrary to $[0,+\infty)\subset \widetilde{U}$.

Next,  suppose, contrary to our claim, that for every $R>0$ the map $V_*$ on the rectangle $[-R,x^{max}(E_0)+R]\times[-R,R]$ is not surjective.
Then there are two sequences
$(z_n)_{n\geq 1}$ and $(z'_n)_{n\geq 1}$ in $\C$ such that $z_n,z'_n \in [-1/n,x^{max}(E_0)+1/n]\times[-1/n,1/n]$, $z_n\neq z'_n$ and $V_*(z_n)=V_*(z'_n)$. Passing to  subsequences, if necessary, we have
$z_n\to x\in [0,x^{max}(E_0)]$, $z'_n\to x'\in [0,x^{max}(E_0)]$ and  $V(x)^{1/m}=V(x')^{1/m}$. Since $V:[0,+\infty)\to[0,  +\infty)$ is strictly increasing, we have $x=x'$. This contradicts local invertibility of $V_*$
around $x\in \widetilde{U}$.

Summarizing,
\[V_*:[-R,x^{max}(E_0)+R]\times[-R,R]\to V_*([-R,x^{max}(E_0)+R]\times[-R,R])\] is biholomorphic. Since, {by definition \(V(x^{max}(E_0))=E_0\), so}
\[V_*([0,x^{max}(E_0)])=V^{1/m}([0,x^{max}(E_0)])=[0,V^{1/m}(x^{max}(E_0))]=[0,E_0^{1/m}],\]
the set $V_*((-R,x^{max}(E_0)+R)\times(-R,R))$ is an open neighbourhood of $[0,E_0^{1/m}]$. Let $\C_+=\{z\in \C: \rp z> 0\}$.
Denote by $\overline{\C}_+\ni z\mapsto z^{1/m}\in \overline{\C}_+$  the power complex map so that on $\C_+$ it is a holomorphic extension of the real power map.
Then there exists $0<r<E_0$ such that $\overline{C(E_0,r)}^{1/m}\subset V_*((-R,x^{max}(E_0)+R)\times(-R,R))$, where \(C(E_0,r) \) is the droplet emanating from \(E_0\). This follows from $\bigcap_{r>0}\overline{C(E_0,r)}^{1/m}=[0,E_0^{1/m}]$.
Let us consider $z^{max}$, the complex extension of \(x^{max}(E)\) on:
\[z^{max}:V((-R,x^{max}(E_0)+R)\times(-R,R))\cap\overline{\C}_+\to\C\]
defined by
\[z^{max}(E):=V^{-1}_*(E^{1/m})\]
($E$ is a complex variable in the proof).
Then, \(V_{*}(z^{max}(E))=E^{1/m}\), in particular, \(V_{*}(x_{max}(E))=E^{1/m}\) for real positive \(E\).
Define
\begin{align}\label{def:Uy}
\begin{aligned}
U_{E_0}:&=z^{max}(V((-R,x^{max}(E_0)+R)\times(-R,R))\cap\overline{\C}_+)\\
&\subset (-R,x^{max}(E_0)+R)\times(-R,R)\subset\widetilde{U}\subset U.
\end{aligned}
\end{align}
{Since we showed that }for every \(E_0>0\)  and $R=R_{E_0}>0$ there exists \(0<r<E_0\) such that { $\overline{C(E_0,r)}^{1/m}\subset V_*((-R,x^{max}(E_0)+R)\times(-R,R))$, we obtain that }
\[{C(E_0,r)}\subset V((-R,x^{max}(E_0)+R)\times(-R,R))\cap \C_+.\]
Moreover, for every $E\in V((-R,x^{max}(E_0)+R)\times(-R,R))\cap\C_+$ we have
\[V(z^{max}(E))=V^m_*(V^{-1}_*(E^{1/m}))=E.\]
It follows that $V:U_{E_0}\to V(U_{E_0})$ is biholomorphic and ${C(E_0,r)}\subset V(U_{E_0})$, so the condition  \eqref{eq:iii} holds.

As $V_*:[-R,x^{max}(E_0)+R]\times[-R,R]\to V_*([-R,x^{max}(E_0)+R]\times[-R,R])$ is biholomorphic, there exists $C_x$ such that
\[\left|\frac{V_*''(z)V_*(z)}{(V_*'(z))^2}\right|\leq c_{E_0}\;\text{ for all}\;z\in [-R,x^{max}(E_0)+R]\times[-R,R].\]
Since $V_*^m=V$, we have
\[V'(z)=mV_*(z)^{m-1}V_*'(z),\;V''(z)=m(m-1)V_*(z)^{m-2}(V_*'(z))^2+mV_*(z)^{m-1}V_*''(z).\]
Hence, by \eqref{def:Uy}, for every $z\in U_{E_0}$ we have
\[\left|\frac{V''(z)V(z)}{(V'(z))^2}\right|=\left|\frac{m-1}{m}+\frac{V_*''(z)V_*(z)}{m(V_*'(z))^2}\right|\leq C_{E_0}:=\frac{c_{E_0}}{m}+1,\]
so the condition \eqref{eq:iv} also holds.
\end{proof}


\begin{lemma}\label{lem:ye}
The function $V:\R\to \R_{\geq 0}$ given by $V(x)=|x|e^{-1/|x|}$ for $x\neq 0$ and $V(0)=0$ is a convex $Deck$-potential.
\end{lemma}

\begin{proof}
First note that the conditions \eqref{eq:i} and \eqref{eq:ii} are obviously satisfied. We focus only on \eqref{eq:iii} and \eqref{eq:iv}.

Let us consider its holomorphic extension $V:\C\setminus \{0\}\to\C$ given by $V(z)=z\exp(-1/z)$. Let $I:\C\setminus\{0\}\to\C\setminus\{0\}$,
$I(z)=1/z$. Then $V=I\circ\widetilde{V}\circ I^{-1}$ on $\C\setminus \{0\}$, where
$\widetilde{V}:\C\to\C$ is the holomorphic map $\widetilde{V}(z)=z\exp(z)$. Since the map $[0,\pi/2)\ni y\mapsto y\tan(y)\in[0,+\infty)$ is strictly increasing with $\lim_{y\to\pi/2}y\tan(y)=+\infty$,
there exists $y_0\in(0,\pi/2)$ such that $y_0\tan(y_0)=1$. Then $y\tan(y)<1$ for all $y\in(-y_0,y_0)$.
Therefore
\[\rp \widetilde{V}'(z)>0\quad\text{for all}\quad z\in A=\{z\in\C;\;\rp z> 0,\;|\ip z|<y_0\}.\]
Indeed, if $z=x+iy$ with $x>0$ and $|y|<y_0$, then
\[\rp \widetilde{V}'(z)=\rp [(z+1)\exp(z)]=e^x((x+1)\cos y-y\sin y)>0.\]
It follows that $\widetilde{V}$ on {the half strip} $A$ is injective. Indeed, suppose, contrary to our claim, that
$\widetilde{V}(z_1)=\widetilde{V}(z_2)$ for distinct $z_1$, $z_2$ in $A$. Then
\[0=\rp\frac{\widetilde{V}(z_2)-\widetilde{V}(z_1)}{z_2-z_1}=\int_0^1\rp \widetilde{V}'(z_1+t(z_2-z_1))\,dt>0.\]

As $\widetilde{V}'(z)\neq 0$ for every $z\in A$, the restriction $\widetilde{V}:A\to\widetilde{V}(A)$ is biholomorphic.
Moreover, $\widetilde{V}(A)$ is an open domain located between the three curves $\{iye^{iy}:y\in(-y_0,y_0)\}$, $\{(x+iy_0)e^{x+iy_0}:x>0\}$ and $\{(x-iy_0)e^{x-iy_0}:x>0\}$.
{Next, we show that }
\begin{equation}\label{inc:VA}
\mathcal{S}_{y_0}:=\{z\in\C\setminus\{0\}:\operatorname{Arg}(z)\in (-y_0,y_0)\}\subset \widetilde{V}(A).
\end{equation}
Indeed, suppose that $z\in \mathcal{S}_{y_0}$, i.e.  $\rp z>0$ and $\Big|\frac{\ip z}{\rp z}\Big|<\tan(y_0)$ {(recall that \(y_0\in(0,\pi/2) \))}. Take $x\geq 0$ such that
\[\rp[(x+iy_0)e^{x+iy_0}]=\rp z=\rp[(x-iy_0)e^{x-iy_0}].\]
Then it is enough to show that
\[\ip[(x+iy_0)e^{x+iy_0}]>\ip z>\ip[(x-iy_0)e^{x-iy_0}]=-\ip[(x+iy_0)e^{x+iy_0}],\]
or equivalently
\[\Big|\frac{\ip z}{\rp z}\Big|<\frac{\ip[(x+iy_0)e^{x+iy_0}]}{\rp[(x+iy_0)e^{x+iy_0}]}.\]
By assumption, we need to show that
\[\tan(y_0)=\frac{\ip e^{iy_0}}{\rp e^{iy_0}}<\frac{\ip[(x+iy_0)e^{x+iy_0}]}{\rp[(x+iy_0)e^{x+iy_0}]}.\]
 It follows from
\begin{gather*}\ip((x+iy_0)e^{x+iy_0})\rp e^{iy_0}-\rp((x+iy_0)e^{x+iy_0})\ip e^{iy_0}\\
=\ip[(x+iy_0)e^{x+iy_0}\overline{e^{iy_0}}]=y_0e^x>0
\end{gather*}
and $\rp e^{iy_0}>0$, $\rp((x+iy_0)e^{x+iy_0})=\rp z>0$. This gives $z\in \widetilde{V}(A)$, and hence \eqref{inc:VA} holds.

\medskip

Let $U:=I^{-1}(\widetilde{V}^{-1}(\mathcal{S}_{y_0}))$. By \eqref{inc:VA}, $U$ is an open subset of $I^{-1}(A)\subset\C_+$ which contains the half-line $(0,+\infty)$.
Since $I(\mathcal{S}_{y_0})=\mathcal{S}_{y_0}$ and $\widetilde{V}:A\to\widetilde{V}(A)$ is biholomorphic, the map $V:U\to \mathcal{S}_{y_0}$ is biholomorphic, so \eqref{eq:iii}  holds.
Moreover, for every
$z\in \C\setminus\{0\}$ we have
\begin{equation}\label{eq:derV}
V'(z)=\Big(1+\frac{1}{z}\Big)\exp(-1/z)\quad\text{and}\quad V''(z)=\frac{1}{z^3}\exp(-1/z),
\end{equation}
and hence
\begin{equation*}
\frac{V(z)V''(z)}{(V'(z))^2}=\frac{1}{(1+z)^2}.
\end{equation*}
As $U\subset I^{-1}(A)\subset\C_+$, it follows that $|{V(z)V''(z)}/{(V'(z))^2}|\leq 1$ for $z\in U$, so \eqref{eq:iv} holds with $U_E=U$ and $C_E=1$.

The convexity of $V$ follows from \eqref{eq:derV}.
\end{proof}

\begin{lemma}\label{lem:i-iv}
If $V:\R\to\R_{\geq 0}$ is a $Deck$-potential, then $V^m:\R\to\R_{\geq 0}$ is a $Deck$-potential for every $m\in\N$.
\end{lemma}

\begin{proof}
As usual, the only challenge in the proof is to show \eqref{eq:iii} and \eqref{eq:iv}.
Assume that $V:U\to \C$ is a holomorphic extension of $V:(0,+\infty)\to(0,+\infty)$ such that $(0,+\infty)\subset U$ is open.
Let $U_m:=V^{-1}(\mathcal{S}_{\pi/m})$. Then $V^m:U_m\to\C$ is a holomorphic extension of $V^m:(0,+\infty)\to(0,+\infty)$ such that $U_m$ is open.

Take any $E_0>0$. As $V$ satisfies \eqref{eq:iii} and \eqref{eq:iv}, there exists $0<r<E_0^{1/m}$, an open $U_{E_0^{1/m}}\subset U$ and $C_{E_0^{1/m}}>0$ such that $V:U_{E_0^{1/m}}\to V(U_{E_0^{1/m}})$ is biholomorphic,
\begin{equation}\label{C1m}
C(E_0^{1/m},r)\subset V(U_{E_0^{1/m}}),\quad C(E_0^{1/m},r)\subset \mathcal{S}_{\pi/m}
\end{equation}
and
\begin{equation}\label{C1m1}
\left|\frac{V''(z)V(z)}{(V'(z))^2}\right|\leq C_{E_0^{1/m}}\;\text{ for all}\;z\in U_{E_0^{1/m}}.
\end{equation}
Taking $U_{E_0}^m:=U_{E_0^{1/m}}\cap U_m$ we have $V^m:U_{E_0}^m\to V^m(U_{E_0}^m)$ biholomorphic and $C(E_0^{1/m},r)^m\subset V^m(U_{E_0^{1/m}})$.
Then there exists $0<r_m<E_0$ such that
\[C(E_0,r_m)\subset C(E_0^{1/m},r)^m.\]
In view of \eqref{C1m}, it follows that
\[C(E_0,r_m)\subset  \big(V(U_{E_0^{1/m}})\cap \mathcal{S}_{\pi/m}\big)^m\subset  V^m(U_{E_0^{1/m}}\cap U_m)=V^m(U_{E_0}^m),\]
which gives \eqref{eq:iii} for $V^m$. Moreover, by\eqref{C1m1}, for every $z\in U_{E_0}^m\subset U_{E_0^{1/m}}$ we have
\[\left|\frac{(V^m)''(z)V^m(z)}{((V^m)'(z))^2}\right|=\left|\frac{m-1}{m}+\frac{V''(z)V(z)}{m(V'(z))^2}\right|\leq \frac{C_{E_0^{1/m}}}{m}+1,\]
which gives \eqref{eq:iv} for $V^m$.
\end{proof}

\begin{example}
In view of Lemma~\ref{lem:ye}~and~\ref{lem:i-iv}, for every $m\geq 2$ the function $V_m:\R\to \R_{\geq 0}$ given by $V_m(x)=|x|^me^{-1/|x|}$ for $x\neq 0$ and $V_m(0)=0$ is also a $Deck$-potential. Indeed, it immediately follows from  $V_m(x)=\tfrac{1}{m^m}V_1^m(mx)$. Moreover, $V_m$ satisfies the key condition \eqref{eq:vi}. Indeed, by \eqref{eq:derV}, for every $x>0$ we have
\[\frac{(V_m)''(x)V_m(x)}{((V_m)'(x))^2}=\frac{m-1}{m}+\frac{V_1''(x)V_1(x)}{m(V_1'(x))^2}>\frac{m-1}{m}\geq \frac{1}{2}.\]
\end{example}

\begin{proposition}\label{prop:smpot}
Let $V:\R\to\R_{\geq 0}$ be an analytic even uni-modal potential. If $V$ satisfies \eqref{eq:v} then $V''(0)\geq 0$ and $V^{(4)}(0)\geq 0$.
Conversely, if $V^{(2m)}(0)\geq 0$ for  all $m\geq 1$ then $V$ satisfies \eqref{eq:v}.
\end{proposition}

\begin{proof}
By the proof of Proposition~\ref{prop:anDeck}, there exists an analytic map $V_*:\R\to\R$ such that $V(x)=(V_*(x))^2$ for all $x\in\R$.
Moreover, $V_*(\R_{\geq 0})=\R_{\geq 0}$ and $V_*$ is even or odd.
Suppose that $V$ satisfies \eqref{eq:v}, i.e.\ $V_*:\R_{\geq 0}\to\R_{\geq 0}$ is convex.
Recall that
\[ V''=2V_*V''_*+2(V'_*)^2,\  V^{(4)}=2V_*V^{(4)}_*+8V'_*V'''_*+6(V''_*)^2.
\]

Assume that $V_*$ is even. Then $V_*(0)=V'_*(0)=V'''_*(0)=0$, and hence
\[ V''(0)=0,\  V^{(4)}(0)=6(V''_*(0))^2\geq 0.
\]

Assume that $V_*$ is odd. Then $V_*(0)=V''_*(0)=0$ and
\[V_*(x)=V'_*(0)x+O(x^2),\quad V''_*(x)=6V'''_*(0)x+O(x^2).\]
By assumption, $V_*(x)\geq 0$ and $V''_*(x)\geq 0$ for all $x\geq 0$, so $V'_*(0)\geq 0$ and $V'''_*(0)\geq 0$.
It follows that
\[ V''(0)=2(V'_*(0))^2\geq 0,\  V^{(4)}(0)=8V'_*(0)V'''_*(0)\geq 0,
\]
which completes the proof of the first part.

\medskip

Now suppose that  $V:\R\to\R_{\geq 0}$ is an analytic even such that $V^{(2m)}(0)\geq 0$ for  all $m\geq 1$. As $V$ is even, we also have
$V^{(2m-1)}(0)=0$ for  all $m\geq 1$. By definition, we need to show that the analytic map
\[W(x)=V(x)V''(x)-\frac{1}{2}(V'(x))^2\]
takes only non-negative values. As $W$ is analytic and even, it is enough to prove that $W^{(2m)}(0)\geq 0$ for all $m\geq 0$.
By the general Leibniz rule, since $V(0)=0$ and $V^{(2k+1)}(0)=0$ for $k\geq 0$, we have $W(0)=V(0)V''(0)-\tfrac{1}{2}(V'(0))^2=0$ and for $m\geq 1$
\begin{align*}
W^{(2m)}(0)&=\sum_{k=0}^{2m}\binom{2m}{k}V^{(k)}(0)V^{(2m-k+2)}(0)-\frac{1}{2}\sum_{k=0}^{2m}\binom{2m}{k}V^{(k+1)}(0)V^{(2m-k+1)}(0)\\
&=\sum_{k=1}^{m}\binom{2m}{2k}V^{(2k)}(0)V^{(2m-2k+2)}(0)-\frac{1}{2}\sum_{k=1}^{m}\binom{2m}{2k-1}V^{(2k)}(0)V^{(2m-2k+2)}(0)\\
&=\frac{1}{2}\sum_{k=1}^{m}\left(\binom{2m}{2k}+\binom{2m}{2k-2}-\binom{2m}{2k-1}\right)V^{(2k)}(0)V^{(2m-2k+2)}(0).
\end{align*}
Moreover, we have
\begin{align*}
&\binom{2m}{2k}+\binom{2m}{2k-2}-\binom{2m}{2k-1}\\&={}\binom{2m-1}{2k}+\binom{2m-1}{2k-1}+\binom{2m-1}{2k-2}+\binom{2m-1}{2k-3}
-\binom{2m-1}{2k-1}-\binom{2m-1}{2k-2}\\
&{}=\binom{2m-1}{2k}+\binom{2m-1}{2k-3}\geq 0.
\end{align*}
Since $V^{(2k)}(0)\geq 0$ for all $k\geq 0$, it follows that $W^{(2m)}(0)\geq 0$ for all $m\geq 0$, which completes the proof.
\end{proof}

\bibliographystyle{abbrv}
\bibliography{springsbib,Impact-Zones-bib,myrefs,citations2}

\begin{thebibliography}{10}

\bibitem{Issi2019}
L.~Becker, S.~Elliott, B.~Firester, S.~G. Cohen, M.~Pnueli, and V.~Rom-Kedar.
\newblock Impact hamiltonian systems and polygonal billiards.
\newblock to appear in the "Proceedings of the MSRI 2018 Fall semester on
  Hamiltonian Systems", \url{https://arxiv.org/abs/2001.03726}.

\bibitem{berglund1996billiards}
N.~Berglund.
\newblock Billiards in a potential: variational methods, periodic orbits and
  {KAM} tori.
\newblock 1996.

\bibitem{Dragovic2014a}
V.~Dragovi{\'{c}} and M.~Radnovi{\'{c}}.
\newblock Bicentennial of the great poncelet theorem (1813{\textendash}2013):
  Current advances.
\newblock {\em Bulletin of the American Mathematical Society}, 51(3):373--445,
  apr 2014.

\bibitem{DrRa2014}
V.~Dragovi\'{c} and M.~Radnovi\'{c}.
\newblock Pseudo-integrable billiards and arithmetic dynamics.
\newblock {\em J. Mod. Dyn.}, 8(1):109--132, 2014.

\bibitem{dullin1998linear}
H.~Dullin.
\newblock Linear stability in billiards with potential.
\newblock {\em Nonlinearity}, 11(1):151, 1998.

\bibitem{Fedorov2001}
Y.~N. Fedorov.
\newblock An ellipsoidal billiard with a quadratic potential.
\newblock {\em Funktsional. Anal. i Prilozhen.,}, 35(3):48--59, 2001.

\bibitem{Fox-Ker}
R.~H. Fox and R.~B. Kershner.
\newblock Concerning the transitive properties of geodesics on a rational
  polyhedron.
\newblock {\em Duke Math. J.}, 2(1):147--150, 1936.

\bibitem{frkaczek2019recurrence}
K.~Fr{\k{a}}czek.
\newblock Recurrence for smooth curves in the moduli space and application to
  the billiard flow on nibbled ellipses.
\newblock to appear in Analysis and Partial Differential Equations,
  \url{https://arxiv.org/abs/1904.12715}.

\bibitem{FrShUl2018}
K.~Fr{\k{a}}czek, R.~Shi, and C.~Ulcigrai.
\newblock Genericity on curves and applications: pseudo-integrable billiards,
  {E}aton lenses and gap distributions.
\newblock {\em J. Mod. Dyn.}, 12:55--122, 2018.

\bibitem{KozlovBook}
V.~Kozlov and D.~Treschev.
\newblock {\em A genetic introduction to the dynamics of systems with impacts}.
\newblock AMS, Providence, 1991.

\bibitem{Kozlova99}
T.~Kozlova.
\newblock On polynomial integrals of systems with elastic impacts.
\newblock {\em Regular and Chaotic Dynamics}, 4:1:83--90, 1999.

\bibitem{McM}
C.~McMullen.
\newblock The figure ``trapped'' in the gallery.
\newblock \url{http://people.math.harvard.edu/~ctm/gallery/index.html}.

\bibitem{MiWe2014}
Y.~Minsky and B.~Weiss.
\newblock Cohomology classes represented by measured foliations, and {M}ahler's
  question for interval exchanges.
\newblock {\em Ann. Sci. \'{E}c. Norm. Sup\'{e}r. (4)}, 47(2):245--284, 2014.

\bibitem{PRk2020}
M.~Pnueli and V.~Rom-Kedar.
\newblock On the structure of hamiltonian impact systems.
\newblock to appear in Nonlinearity, \url{https://arxiv.org/abs/1903.08851}.

\bibitem{pnueli2018near}
M.~Pnueli and V.~Rom-Kedar.
\newblock On near integrability of some impact systems.
\newblock {\em SIAM Journal on Applied Dynamical Systems}, 17(4):2707--2732,
  2018.

\bibitem{Radnovic2015}
M.~Radnovic.
\newblock Topology of the elliptical billiard with the hooke's potential.
\newblock {\em Theoretical and Applied Mechanics}, 42(1):1--9, 2015.

\bibitem{Sa-Ge}
G.~Sansone and J.~Gerretsen.
\newblock {\em Lectures on the theory of functions of a complex variable. {I}.
  {H}olomorphic functions}.
\newblock P. Noordhoff, Groningen, 1960.

\bibitem{Viana2008lecturenotes}
M.~Viana.
\newblock Dynamics of interval exchange transformations and teichmüller flows.
\newblock Lecture notes available from \url{http://w3. impa. br/~
  viana/out/ietf. pdf} (2008).

\bibitem{Wojtkowski1999}
M.~P. Wojtkowski.
\newblock Complete hyperbolicity in {H}amiltonian systems with linear potential
  and elastic collisions.
\newblock In {\em Proceedings of the {XXX} {S}ymposium on {M}athematical
  {P}hysics ({T}oru\'{n}, 1998)}, volume~44, pages 301--312, 1999.

\bibitem{Yo}
J.-C. Yoccoz.
\newblock Interval exchange maps and translation surfaces.
\newblock In {\em Homogeneous flows, moduli spaces and arithmetic}, volume~10
  of {\em Clay Math. Proc.}, pages 1--69. Amer. Math. Soc., Providence, RI,
  2010.

\bibitem{Ka-Ze}
A.~N. Zemljakov and A.~B. Katok.
\newblock Topological transitivity of billiards in polygons.
\newblock {\em Mat. Zametki}, 18(2):291--300, 1975.

\bibitem{zharnitsky2000invariant}
V.~Zharnitsky.
\newblock Invariant tori in {Hamiltonian} systems with impacts.
\newblock {\em Communications in Mathematical Physics}, 211(2):289--302, 2000.

\bibitem{Zorich2006}
A.~Zorich.
\newblock Flat surfaces.
\newblock {\em Frontiers in number theory, physics, and geometry I}, pages
  439--585, 2006.

\end{thebibliography}

\end{document}